\pdfoutput=1

\documentclass{daj}
\usepackage{amsmath,amsrefs, amssymb,  amsthm, enumerate, hyperref, mathrsfs, tikz}

\usetikzlibrary{arrows, calc, matrix}

\usepackage{esint} 

\DeclareFontFamily{U}{mathx}{\hyphenchar\font45}
\DeclareFontShape{U}{mathx}{m}{n}{
      <5> <6> <7> <8> <9> <10>
      <10.95> <12> <14.4> <17.28> <20.74> <24.88>
      mathx10
      }{}
\DeclareSymbolFont{mathx}{U}{mathx}{m}{n}
\DeclareFontSubstitution{U}{mathx}{m}{n}
\DeclareMathAccent{\widecheck}{0}{mathx}{"71}
\DeclareMathAccent{\wideparen}{0}{mathx}{"75}

\theoremstyle{plain}
\newtheorem{theorem}{Theorem}[section]
\newtheorem{lemma}[theorem]{Lemma}
\newtheorem{corollary}[theorem]{Corollary}
\newtheorem{claim}[theorem]{Claim}
\newtheorem{proposition}[theorem]{Proposition}

\newtheorem{conjecture}[theorem]{Conjecture}
\theoremstyle{definition}

\newtheorem{example}[theorem]{Example}
\newtheorem{remark}[theorem]{Remark}

\newtheorem{problem}[theorem]{Problem}

\newcommand{\C}{{\mathbb C}}
\newcommand{\R}{{\mathbb R}}
\newcommand{\Q}{{\mathbb Q}}
\newcommand{\Z}{{\mathbb Z}}
\newcommand{\N}{{\mathbb N}}

\newcommand{\ud}{\mathrm{d}}

\newcommand{\RN}[1]{%
  \textup{\uppercase\expandafter{\romannumeral#1}}%
}

\def\MR#1{\href{http://www.ams.org/mathscinet-getitem?mr=#1}{MR#1}}

\dajAUTHORdetails{%
  title = {The Fourier Restriction and Kakeya Problems over Rings of Integers Modulo $N$}, 
  author = {Jonathan Hickman and James Wright},
  plaintextauthor = {Jonathan Hickman and James Wright},
  plaintexttitle = {The Fourier Restriction and Kakeya Problems over Rings of Integers Modulo N}, 
  runningtitle = {Fourier Restriction and Kakeya over Rings of Integers}, 
  runningauthor = {Jonathan Hickman and James Wright},
  copyrightauthor = {J. Hickman and J. Wright},
  keywords = {Fourier restriction, Kakeya, Congruence equations.},
}   

\dajEDITORdetails{%
   year={2018},
   number={11},
   received={29 January 2018},   
   published={30 May 2018},  
   doi={10.19086/da.3682},       
}   

\begin{document}

\begin{frontmatter}[classification=text]

\title{The Fourier Restriction and Kakeya Problems over Rings of Integers Modulo $N$} 

\author[Hickman]{Jonathan Hickman \thanks{Supported by NSF Grant No. DMS-1440140}}
\author[Wright]{James Wright}

\begin{abstract}
The Fourier restriction phenomenon and the size of Kakeya sets are explored in the setting of the ring of integers modulo $N$ for general $N$ and a striking similarity with the corresponding euclidean problems is observed. One should contrast this with known results in the finite field setting.
\end{abstract}
\end{frontmatter}


\section{Introduction}

In \cite{Mockenhaupt2004} Mockenhaupt and Tao introduced a variant of the classical (euclidean) Fourier restriction problem in the setting of finite fields. The point of view espoused in \cite{Mockenhaupt2004}, following an initial proposal by Wolff for the Kakeya problem, is to seek a model discrete setting in which to study various modern harmonic analysis problems (the Fourier restriction, Kakeya and Bochner--Riesz conjectures, \emph{et cetera}) which should highlight certain aspects of the euclidean problem: for instance, the underlying combinatorial or incidence-geometric features. The following Fourier restriction problem\footnote{In \cite{Mockenhaupt2004} the problem was proposed only in the setting of vector spaces over finite fields, but it can be equally formulated over any finite abelian group.} was proposed in the setting of a finite abelian group $G$.

\begin{problem} Let $\Sigma \subseteq \widehat{G}^n$ be a set of frequencies in the $n$-fold product of the dual group $\widehat{G}$. Consider the $\ell^r - \ell^s$ Fourier restriction estimates
\begin{equation}\label{FR-estimate}
\big(\frac{1}{|\Sigma|} \sum_{\xi \in \Sigma} |\hat{F}(\xi )|^s \big)^{1/s} \ \le \ C_{r,s,n} \big(\sum_{x \in G^n} |F(x)|^r \big)^{1/r}
\end{equation}
where $|\Sigma|$ denotes the cardinality of $\Sigma$. The basic problem is to determine, for a given set of frequencies $\Sigma$, those Lebesgue exponents $1 \leq r, s \leq \infty$ for which $C_{r,s,n}$ can be taken to be `essentially' independent of the cardinality of $G$.
\end{problem}

 As the estimate \eqref{FR-estimate} indicates, here the dual group ${\widehat G}$ is equipped with normalised counting measure whereas counting measure is used for the Haar measure on the original group $G$. These choices for Haar measure define the corresponding Lebesgue $\ell^r$ norms on these groups and the Fourier transform of any $F \colon G^n \to \C$ by ${\hat F}(\xi) = \sum_{x\in G} F(x) \xi(-x)$ where $\xi$ denotes a character in the dual group.

An investigation of this problem was initiated in \cite{Mockenhaupt2004} in the case where $G$ is a finite-dimensional vector space over a finite field. This proved to be an interesting discrete model for the Fourier restriction problem, isolating and highlighting various combinatorial features. Furthermore, Dvir \cite{Dvir2009} later solved the finite field version of the Kakeya problem as proposed by Wolff and the more quantitative maximal Kakeya problem was then established by Ellenberg, R. Oberlin and Tao \cite{Ellenberg2010}. 

Naturally, questions arise regarding how well the finite field variant models the euclidean setting for these problems. One obvious difference is that there are few scales to work with in the finite field setting. This is clearly manifested when studying the Fourier transform of measures carried along curves or surfaces: these are exponential sums in the finite field setting and the famous A. Weil \cite{Weil1948} (or, more generally, Deligne \cite{Deligne1974}) estimates show that, typically, either the situation is completely non-degenerate (corresponding to the non-vanishing curvature case in euclidean restriction theory, with optimal exponential sum estimates) or it is completely degenerate and only trivial estimates hold. However, in moving from finite fields (for example, the integers modulo a prime $p$) to the setting of the finite ring of integers modulo $N$ for general $N$, the divisors of $N$ provide additional scales to work with. Consequently, it has been proposed that harmonic analysis over $\Z/N\Z$  may match the euclidean case more closely.

In this paper it is shown that this is indeed the case for the Fourier restriction problem. A sample theorem is the corresponding Stein--Tomas $\ell^2$ restriction theorem for the paraboloid 
\begin{equation}\label{paraboloid definition}
\Sigma := \big\{(\vec{\omega}, \omega_1^2 + \cdots + \omega_{n-1}^2) : \ \vec{\omega} = (\omega_1,\ldots,\omega_{n-1}) \in [\Z/N\Z]^{n-1} \big\},
\end{equation}
stated here informally.

\begin{theorem}[Informal]\label{Tomas}
Let $\Sigma$ be the paraboloid in $[\Z/N\Z]^n$ as described above. If $s = 2$, then the Fourier restriction estimate \eqref{FR-estimate} holds if and only if ${1\le r \leq 2(n+1)/(n+3)}$.
\end{theorem}

In the following section the Fourier restriction problem is precisely formulated in the setting of $[\Z/N\Z]^n$; see \eqref{general N problem inequality}.




\section{The basic setup}\label{basic setup section}

To begin some notation is introduced in order to facilitate the comparison between the the rings $\Z/N\Z$ and euclidean space. First, a notion of size or scale is allocated to elements $x \in \Z/N\Z$.  Set $|x| := N / \gcd(x,N)$ where $\gcd(a,b)$ denotes the greatest common divisor\footnote{More precisely, the function $|x| := N/\gcd(x, N)$ is defined for all integers $x \in \Z$; if $[x] \in \Z/N\Z$ is a coset containing $x \in \Z$, then $|[x]| := |x|$. It is easy to see that this function is well-defined.} of $a$ and $b$ (when $N = p^{\alpha}$ where $p$ is prime, one may think of $|\,\cdot\,|$ as a `normalised $p$-adic absolute value', where the normalisation is with respect to the ring ${\mathbb Z}/p^{\alpha}{\mathbb Z}$). It is remarked that, algebraically, $|x|$ is the cardinality of the ideal in $\Z/N\Z$ generated by $x$. This notation is extended to elements $\vec{x} = (x_1,\ldots,x_n) \in [\Z/N\Z]^n$ by $\|(x_1,\ldots,x_n)\| := N / \gcd(x_1,\ldots,x_n, N)$. Next define the partial ordering $\preceq$ amongst the integers by $a\preceq b$ if and only if $a \, | \, b$ (similarly, $a \prec b$ will be used to indicate that $a$ is a \emph{proper} divisor of $b$). This is used to compare various sizes $|\cdot|$; for example, $|x| \preceq |y|$ if and only if $\gcd(y,N) \, | \, x$. In order to isolate elements lying at different scales, one may introduce the family of balls $\{{\mathcal B}_d \}_{d \mid N}$, indexed by the divisors of $N$, given by
\begin{equation*}
{\mathcal B}_d \ := \ \{\vec{x} \in [\Z/N\Z]^n : \ \|\vec{x} \| \preceq d \}.
\end{equation*}
These balls will play a prominent r\^ole in the forthcoming analysis. One easily verifies that an element $(x_1,\ldots, x_n)$ lies in ${\mathcal B}_d$ if and only if $N/d$ divides each component $x_j$.

As mentioned above, this notation facilitates the analogy with familiar euclidean notions. The analogy is more precise if one restricts attention to powers $N=p^{\alpha}$ of a fixed prime $p$. In this case the divisors become totally ordered and, in particular, the above balls $\{{\mathcal B}_{p^{\beta}} \}_{0\le \beta \le \alpha}$ form a 1-parameter sequence of nested sets (similar to the family of euclidean balls centred at 0). The following example provides another simple illustration of this analogy, which is relevant to the discussion below.

\begin{example}\label{integral example}  Let $r \in \R$ and compare the `integrals'
\begin{equation*}
\int_{|x|\ge 1} \frac{1}{|x|^r} \, \ud x \qquad \textrm{and} \qquad \sum_{x=0}^{N-1} \frac{1}{|x|^r},
\end{equation*}
where the left-hand integral features the usual (euclidean) absolute value and the right-hand sum involves the absolute value on $\Z/N\Z$ defined above. The euclidean integral, of course, converges if and only if $r>1$. The mod $N$ sum disentangles as $\sum_{d \mid N} \phi(d) d^{-r}$ where $\phi$ is the Euler totient function. When $N=p^{\alpha}$ this sum is uniformly bounded if and only if $r>1$, whereas for general $N$ this is true only when $r>2$; however, in the range $1\le r \le 2$ the bound $\sum_{d \mid N} \phi(d) d^{-r} \leq C_{\varepsilon} N^{\varepsilon}$ holds for every $\varepsilon>0$. When $r<1$ the sum can grow like a positive power of $N$.
\end{example}

These observations suggest two natural ways in which to pose the Fourier restriction problem in the setting of the integers mod $N$: one formulation in which $N$ is only allowed to vary over powers of a fixed prime, and another for general values of $N$. These problems are described precisely below.

For simplicity, attention is restricted to the case where the set of frequencies $\Sigma$ is given by the graph of a polynomial mapping; this is a natural analogue of a smooth surface in euclidean space. In particular, let $1 \leq d \leq n-1$ and $P_1, \dots, P_{n-d} \in \Z[X_1, \dots, X_d]$ and define the polynomial mapping
\begin{equation*}
\Gamma \colon \vec{\omega} = (\omega_1, \dots, \omega_d) \mapsto (\omega_1, \dots, \omega_d, P_1(\vec{\omega}), \dots, P_{n-d}(\vec{\omega})).
\end{equation*}
For any positive integer $N$ one may reduce the coefficients of the polynomials modulo $N$ and consider $\Gamma$ as a mapping from $[\Z/N\Z]^d$ to $[\Z/N\Z]^n$. Thus, $\Gamma$ simultaneously parametrises a $d$-dimensional variety in $[\Z/N\Z]^n$ for each $N \in \N$. By an abuse of notation, in this situation $\Sigma$ will be used to denote any one of these varieties; the choice of variety (that is, the choice of $N$) should always be clear from the context.  


\begin{problem}\label{p power problem} Given a $d$-dimensional variety $\Sigma$ as above, determine the Lebesgue exponents $1 \leq r, s \leq \infty$ such that there is a constant $C = C_{\Sigma, r, s}$, depending only on $\Sigma$,  $r$ and $s$, for which the inequality
\begin{equation}\label{p power problem inequality}
\big(\frac{1}{|\Sigma|} \sum_{\xi \in \Sigma} |\hat{F}(\xi)|^s \big)^{1/s} \ \leq C
\big(\sum_{x \in [{\mathbb Z}/p^{\alpha}{\mathbb Z}]^n} |F(x)|^r \big)^{1/r}
\end{equation}
holds for all (or at least all sufficiently large) primes $p$ and all exponents $\alpha \in \N$. 
\end{problem}

It transpires that Problem \ref{p power problem} is remarkably close to the original euclidean Fourier restriction problem, both in terms of numerology and available methodologies. In fact, many of the techniques used to study the euclidean restriction problem can be translated wholesale into this discrete setting. The striking similarity between the two problems can be explained by the fact that Problem \ref{p power problem} is in fact equivalent, in some precise sense, to a Fourier restriction problem over the (continuous) field of $p$-adic numbers $\Q_p$.\footnote{The basic elements of $p$-adic analysis are reviewed later in $\S$\ref{p adic section}.} This equivalence follows from a `correspondence principle', which is a manifestation of the uncertainty principle, that allows one to `lift' restriction problems over the discrete rings $\Z/p^{\alpha}\Z$ to the continuous setting of $\Q_p$. Since the fields $\Q_p$ and $\R$ are in many ways closely related (by, for instance, Ostrowski's theorem), once this correspondence is understood it is natural to expect the two problems to behave similarly.

Working over the rings $\Z/p^{\alpha}\Z$ or the field $\Q_p$ provides an effective model for the restriction problem; many established techniques become substantially cleaner and simpler when translated into these settings. This is mainly due to the strong forms of the uncertainty principle available over $\Q_p$, owing to the fact that, unlike in the real case, in $\Q_p$ the closed unit ball forms a subgroup. Analysis over $\Q_p$ also naturally leads one to consider Fourier restriction over other local fields, and in particular the field $\mathbb{F}_q((X))$ of formal Laurent series, which, in many respects, offer even more effective model settings for harmonic analysis problems.

Problem \ref{p power problem} is investigated in detail in $\S$\ref{p adic section}, where the aforementioned correspondence principle is established. 
%

\begin{problem}\label{general N problem} Given a $d$-dimensional variety $\Sigma$ as above, determine the Lebesgue exponents $1 \leq r, s \leq \infty$ such that for every $\varepsilon>0$ there is a constant $C_{\varepsilon} = C_{\varepsilon, \Sigma, r, s}$, depending only on $\varepsilon$, $\Sigma$, $r$ and $s$, for which the inequality
\begin{equation}\label{general N problem inequality}
\big(\frac{1}{|\Sigma|} \sum_{\xi \in \Sigma} |\hat{F}(\xi)|^s \big)^{1/s} \ \leq \ C_{\varepsilon} N^{\varepsilon} \, \big(\sum_{x \in [\Z/N\Z]^n} |F(x)|^r \big)^{1/r}
\end{equation}
holds for all (or at least `most') $N \in \N$. 
\end{problem}

It is remarked that in practice it is often desirable to work with `most' rather than all $N$, avoiding certain values which lead to degenerate situations (in particular, $N$ with small prime factors relative to the ambient dimension $n$). When stating results, any such technical restrictions on $N$ will always be described explicitly. 

Once again it transpires that the numerology of this problem closely mirrors that of the euclidean case. However, the partially-ordered scale structure in $\Z/N\Z$ complicates matters and typically the arguments in this setting require additional number-theoretic information. For this reason, Problem \ref{general N problem} is, at least in some respects, arguably more complex than the euclidean problem and therefore perhaps unsuitable as a model. Restriction theory over $\Z/N\Z$ for general $N$ nevertheless appears to be rich and interesting in its own right, and the majority of the article will focus on exploring this formulation of the problem.

In order to understand the r\^ole of the scaling structure in both of these problems, it is useful to examine necessary conditions for the estimates \eqref{p power problem inequality} or \eqref{general N problem inequality} to hold when $\Sigma$ is the paraboloid, as defined in \eqref{paraboloid definition}. In this case \eqref{general N problem inequality}, for instance, can be written 
as\footnote{Throughout this article the notation $[\Z/N\Z]_*$ will be used to indicate the Pontryagin dual group of $\Z/N\Z$. The dual $[\Z/N\Z]_*$ is always tacitly identified with $\Z/N\Z$; in practice, the only distinction between $\Z/N\Z$ and $[\Z/N\Z]_*$ is that the latter is endowed with the \emph{normalised} counting measure.}
\begin{equation*}
\Bigl(\frac{1}{N^{n-1}} \sum_{\vec{\omega} \in [\Z/N\Z]_*^{n-1}} |{\hat F}(\vec{\omega}, \omega_1^2 + \cdots + \omega_{n-1}^2)|^s \Bigr)^{1/s} \ \le \ C_{\varepsilon} N^{\varepsilon} \, \|F\|_{\ell^r([\Z/N\Z]^n)},
\end{equation*}
where $\vec{\omega} = (\omega_1,\ldots,\omega_{n-1})$ and the $\ell^r$-norm on the right is the same as that appearing in \eqref{general N problem inequality}: that is, it is computed with respect to counting measure on $[\Z/N\Z]^n$.

The analysis follows the usual scaling argument in the euclidean setting; in particular, a discrete variant of the standard Knapp example is constructed. Fix a divisor $d$ of $N$ so that its square $d^2$ is also a divisor (this forces $d=1$ if $N$ is prime) and consider the parabolic rectangle
\begin{equation*}
\theta \ := \ \big\{(\vec{\omega},t)\in [\Z/N\Z]_*^n : \ \|\vec{\omega}\| \preceq |d|, \ |t| \preceq |d^2| \big\}.
\end{equation*}
Unraveling the notation, one observes that an element $(\vec{\omega},t) = (\omega_1,\ldots,\omega_{n-1},t)$ belongs to $\theta$ if and only if $d \, | \ \omega_j$
for each $1\leq j \leq n-1$ and $d^2 \, | \ t$. Let $F \colon [\Z/N\Z]^n \to \C$ be defined by $\hat{F} := \chi_{\theta}$ and apply this function to \eqref{general N problem inequality}. It is easy to check that the left-hand side of the resulting inequality is equal to $d^{-(n-1)/s}$. On the other hand, the Fourier inversion formula can be used to show that $F = d^{-(n+1)}\chi_{\theta^*}$ where $\theta^*$ is the dual rectangle
\begin{equation*}
\theta^* \ := \ \big\{(\vec{x},x_n)\in [\Z/N\Z]^n : \ \|\vec{x}\| \preceq d, \ |x_n| \preceq d^2 \big\}.
\end{equation*} 
Hence the $\ell^r$-norm on the right-hand side of \eqref{general N problem inequality} is equal to $d^{-(n+1)/r'}$ and the inequality reduces to
\begin{equation*}
d^{-(n-1)/s} \ \le \ C_{\varepsilon} N^{\varepsilon} \, d^{ - (n+1)/r'} .
\end{equation*}
If this inequality is to hold for arbitrarily large $N$ and $d$, then it follows that 
\begin{equation}\label{exponent relation}
s \frac{n+1}{n-1} \le r';
\end{equation} 
this is the same restriction on the exponents as in the euclidean setting. An almost identical analysis applies in the setting of Problem \ref{p power problem}. The scaling argument does not work, however, over finite fields, where there are few divisors; if $F$ is defined by $\hat{F} := \delta_{\vec{0}}$ rather than $\hat{F} := \chi_{\theta}$, then plugging this function into either \eqref{p power problem inequality} or \eqref{general N problem inequality} yields the less restrictive necessary condition $s \frac{n}{n-1} \le r'$, as observed in \cite{Mockenhaupt2004}.

Given the relation \eqref{exponent relation} on the exponents, one now wishes to examine the viable $\ell^r$ range. By duality, \eqref{general N problem inequality}
is equivalent to 
\begin{equation}\label{E-paraboloid}
\Bigl(\sum_{\vec{x}\in [\Z/N\Z]^{n}} |{\mathcal E}H(\vec{x})|^{r'} \Bigr)^{1/r'} \ \le \ C_{\varepsilon} N^{\varepsilon} \,
\Bigl(\frac{1}{N^{n-1}} \sum_{\vec{\omega} \in [\Z/N\Z]_*^{n-1}} | H(\vec{\omega}) |^{s'} \Bigr)^{1/s'}
\end{equation}
where ${\mathcal E}$ is the \emph{extension operator}
\begin{equation*}
{\mathcal E} H (\vec{x}) \ := \ \frac{1}{N^{n-1}} \sum_{\vec{\omega} \in [\Z/N\Z]_*^{n-1}}  H(\vec{\omega})
e^{2 \pi i (x_1 \omega_1 + \cdots x_{n-1} \omega_{n-1} + x_n (\omega_1^2 + \cdots + \omega_{n-1}^2))/N}.
\end{equation*}
When $ H := 1$ is a constant function, ${\mathcal E} 1 (\vec{x}) = \prod_{j=1}^{n-1} G_N(x_j,x_n)$ where
\begin{equation*}
G_N(a,b) \ := \ \frac{1}{N} \sum_{t=0}^{N-1} e^{2 \pi i (a t + b t^2)/N}
\end{equation*}
is a Gauss sum. One easily checks that $G_N(a,b)$ vanishes unless $\gcd(b,N) \, | \ a$ in which case, if (say) $N$ is odd, $|G_N(a,b)| = \sqrt{\gcd(b,N)/N}$. Using the above notation, these observations are succinctly expressed by the formula\footnote{It is informative to compare this analysis with its euclidean counterpart. For the euclidean problem one wishes to analyse the decay rate of the Fourier transform of some smooth, compactly supported density $\mu$ on the paraboloid in $\R^n$. In particular, $\check{\mu}$ may be expressed in terms of the oscillatory integral
\begin{equation*}
I(a,b) := \int_{\R} e^{2 \pi i (a t + b t^2)} \psi(t)\,\ud t,
\end{equation*}
where $\psi \in C_0^{\infty}(\R)$ is supported in $[-1,1]$, say. If $|a| \geq C |b|$, then the phase has no critical points and therefore $I(a,b)$ is rapidly decreasing in $|b|$. Otherwise, stationary phase \cite{Stein1993}*{Chapter VIII} implies that $|I(a,b)| \sim C|b|^{-1/2}$. This is entirely analogous to the behaviour of the Gauss sum $G_N(a,b)$ highlighted by the identity \eqref{evaluating Gauss sum}. Deeper connections between the theory of complete exponential sums and the theory of oscillatory integrals have been pursued in a number of papers of the second author \cites{Wright, Wright2011a, Wright2011}. These ideas will be discussed further at the end of $\S$\ref{L2 restriction section}.}
\begin{equation}\label{evaluating Gauss sum}
|G_N(a,b)|  = \left\{\begin{array}{ll}
|b|^{-1/2} & \textrm{if $|a| \preceq |b|$}\\
0 & \textrm{otherwise}
\end{array} \right. .
\end{equation}
Plugging $H:=1$ into \eqref{E-paraboloid} and applying the identity \eqref{evaluating Gauss sum}, it follows that the right-hand side is equal to $C_{\varepsilon} N^{\varepsilon}$ whereas the $r'$ power of the left-hand side is given by (when $N$ is odd)
\begin{equation*}
\sum_{d \mid N} \ \sum_{\gcd(x_n,N)=d} \Bigl(\frac{N}{d}\Bigr)^{n-1} \Bigl(\frac{d}{N}\Bigr)^{(n-1)r'/2}
\ = \ \ \sum_{d \mid N} \phi(d) \, d^{-(n-1)(r'/2 -1)}.
\end{equation*}
This sum is precisely of the form of that considered in Example \ref{integral example}. In particular, if \eqref{E-paraboloid} is to hold, then Example \ref{integral example} implies that necessarily $r' \geq 2n/(n-1)$, which matches the euclidean range, at least up to the endpoint. Again, an almost identical analysis applies in the setting of Problem \ref{p power problem}, utilising the differences between the $\Z/N\Z$ and $\Z/p^{\alpha}\Z$ described in Example \ref{integral example}. Alternatively, the reasoning of \cite{Mockenhaupt2004} is valid in any finite abelian group $G$ and shows that if \eqref{FR-estimate} is to hold with a uniform bound $C_{r,s,n}$, then necessarily $r' \ge 2n/d$ where $|\Sigma| \sim |G|^d$. However, the line of argument presented above has the advantage over that of \cite{Mockenhaupt2004} in that it reinforces the need to formulate the Fourier restriction problem in $[\Z/N\Z]^n$ as in \eqref{general N problem inequality}. Indeed, strict adherence to a uniform bound for $C_{r,s,n}$ in the above argument leads to the more restrictive necessary condition $r' > 2(n+1)/(n-1)$ for the $ H=1$ example.

If $s=2$, then it follows from the preceding examples that \eqref{general N problem inequality} fails for the paraboloid if $r > 2(n+1)/(n+3)$. In $\S$\ref{L2 restriction section} it is shown that \eqref{general N problem inequality} in fact holds for the paraboloid when $s=2$ in the optimal range $1\le r \leq 2(n+1)/(n+3)$. The full range of $\ell^r - \ell^s$ restriction estimates will then be established in $\S$\ref{Fourier restriction for curves section} in the $n=2$ case. The numerology will again match that of the classical euclidean estimates, up to endpoints.




\section{Tools and considerations arising from restriction theory}\label{wave packet section}

The existence of an effective Knapp example in the discrete setting suggests that many of the underlying geometric features of the euclidean Fourier restriction problem should admit some analogue over $\Z/N\Z$. This is explored in detail in the current section; in particular, it is shown that there exists a notion of wave packet decomposition over $\Z/N\Z$ and this leads one to consider certain discrete variants of the Kakeya conjecture. For comparison, the relationship between Kakeya and restriction is far more tenuous over finite fields \cites{Mockenhaupt2004, Lewko}. 

For simplicity attention is restricted to the case where the underlying surface $\Sigma$ is a hypersurface given by a graph. In particular, for the duration of this section let $h \in \Z[X_1, \dots, X_{n-1}]$ be a fixed polynomial and $\Sigma$ denote the variety 
\begin{equation*}
\Sigma := \{ (\vec{\omega}, h(\vec{\omega})) :  \vec{\omega} \in [\Z/N\Z]^{n-1}_*\}. 
\end{equation*}
 Let $\mathcal{E}$ denote the extension operator associated to $\Sigma$, given by
\begin{equation}\label{Kakeya extension operator}
\mathcal{E} H(\vec{x}\,) := \frac{1}{N^{n-1}} \sum_{\vec{\omega} \in [\Z/N\Z]_*^{n-1}} e^{2\pi i \phi(\vec{x}\,;\vec{\omega})/N} H(\vec{\omega})
\end{equation}
for all $ H \colon [\Z/N\Z]_*^{n-1} \to \C$, where $\phi$ is the phase function
\begin{equation*}
\phi(\vec{x}\,;\vec{\omega}) := x' \cdot \vec{\omega} + x_n h(\vec{\omega}) \quad \textrm{for all $\vec{x} = (x', x_n) \in [\Z/N\Z]^n$.} 
\end{equation*}

Let $d \mid N$ be a divisor of $N$. The ball $\mathcal{B}_d\subseteq \Z/N\Z$, as defined in $\S$\ref{basic setup section}, is a subgroup of $\Z/N\Z$ and therefore its cosets form a partition of the ambient ring into disjoint subsets. Define
\begin{equation}\label{Lambda definition}
\Lambda(N;d) := \{[0], [1], \dots, [ (N/d) - 1] \} \subseteq \Z/N\Z,
\end{equation}
where the notation $[x]$ is used to indicate the congruence class of $x \in \Z$ modulo $N$. Thus, for all $k \in \N$ the set $\Lambda(N;d)^k$ forms a complete set of coset representatives for $\mathcal{B}_d \subseteq [\Z/N\Z]^k$.\footnote{Throughout the article the same notation ${\mathcal B}_d$ is used to denote a balls in $[Z/NZ]^k$ or $[Z/NZ]_{*}^k$ for various $k$ depending on the situation. The choice of ambient dimension $k$ should always be clear from the context.} In terms of the scaling structure, $\Lambda(N;d)^k$ corresponds to a choice of a maximal $d$-separated subset of $[\Z/N\Z]^k$, where the notion of `separation' is understood in terms of the `norm' $\|\,\cdot\,\|$ and the $\preceq$ ordering.   

Turning to the definition of the wave packets, fix some intermediate scale $d \mid N$, let $d' := N/d$ and $\Theta_d$ denote the collection of cosets of $\mathcal{B}_{d'}$ in $[\Z/N\Z]_*^{n-1}$ and define
\begin{equation*}
\mathbb{T}_d := \Theta_d \times \Lambda(N;d)^{n-1}.
\end{equation*}
The notation is chosen here to mirror that recently used in euclidean restriction theory (see, for example, \cites{Guth2016a, Guth}). For $(\theta, \vec{v}) \in \mathbb{T}_d$ the \emph{wave packet} 
$\psi_{\theta, \vec{v}} \colon  [\Z/N\Z]_{*}^{n-1} \to \C$ is defined to be the function given by
\begin{equation*}
\psi_{\theta, \vec{v}}(\vec{\omega}) := d^{n-1} e^{-2 \pi i \vec{v} \cdot \vec{\omega}/N} \chi_{\theta}(\vec{\omega}).  
\end{equation*}
Generalising the Fourier inversion formula for the discrete Fourier transform, any $\C$-valued function on $[Z/N\Z]_{*}^{n-1}$ can be written as a superposition of wave packets in a natural manner.

\begin{lemma}\label{wave packet decomposition lemma} For any divisor $d \mid N$ the formula
\begin{equation*}
 H(\vec{\omega}) = \sum_{(\theta, \vec{v}) \in \mathbb{T}_d} (\chi_{\theta}H)\,\widecheck{}\,(\vec{v}\,) \cdot \psi_{\theta, \vec{v}}(\vec{\omega})
\end{equation*}
holds for any function $ H \colon [\Z/N\Z]_{*}^{n-1} \to \C$ .
\end{lemma}

If $d = 1$, then the collection $\Theta_d$ comprises of a single set $\theta = [\Z/N\Z]^{n-1}_*$ and the above identity reduces to 
\begin{equation*}
 H(\vec{\omega}) = \sum_{\vec{v} \in [\Z/N\Z]^{ n-1}} \check{H}(\vec{v}\,) e^{-2 \pi i \vec{v} \cdot \vec{\omega}/N},
\end{equation*}
which is precisely the Fourier inversion formula over $\Z/N\Z$.

\begin{proof}[Proof (of Lemma \ref{wave packet decomposition lemma})] The functions $\chi_{\vec{v} + \mathcal{B}_{d}}$ for $\vec{v} \in \Lambda(N;d)^{n-1}$ form a partition of unity of $[\Z/N\Z]^{n-1}$ and thus, by the Fourier inversion formula,
\begin{equation*}
H(\vec{\omega}) = \sum_{\vec{v} \in \Lambda(N;d)^{n-1}} \hat{\chi}_{\vec{v}+\mathcal{B}_{d}} \ast H(\vec{\omega}).
\end{equation*}
A simple computation shows that 
\begin{equation*}
\hat{\chi}_{\mathcal{B}_d} = d^{n-1} \chi_{\mathcal{B}_{d'}} 
\end{equation*}
and therefore
\begin{equation*}
\hat{\chi}_{\vec{v} + \mathcal{B}_{d}} \ast  H(\vec{\omega}) = d^{n-1}(e^{2 \pi i\vec{v} \cdot (\,\cdot\,)/N}\chi_{\mathcal{B}_{d'}})\ast H(\vec{\omega}) = d^{n-1}(\chi_{\vec{\omega} + \mathcal{B}_{d'}}  H)\;\widecheck{}\;(\vec{v}\,) \cdot e^{-2 \pi i \vec{v} \cdot \vec{\omega}/N}. 
\end{equation*}
If $\vec{\omega} \in \theta$, then $\vec{\omega} + \mathcal{B}_{d'} = \theta$, and so the desired identity immediately follows.
\end{proof}

The extension operator $\mathcal{E}$ has a particularly simple action on wave packets, mapping each $\psi_{\theta, \vec{v}}$ to a modulated characteristic functions of a `tube'. In particular, given a divisor $d \mid N$ and $(\theta, \vec{v}) \in \mathbb{T}_d$, define the $d$-\emph{tube} $T_{\theta, \vec{v}}$ to be the set
\begin{equation*}
T_{\theta, \vec{v}} := \big\{ \vec{x} = (x',x_n) \in [\Z/N\Z]^n : \|x' + x_n \partial_{\omega}h(\vec{\omega}_{\theta}) - \vec{v}\,\| \preceq d\big\},
\end{equation*}
where $\vec{\omega}_{\theta} \in [\Z/N\Z]_*^{n-1}$ denotes the unique coset representative of $\theta$ lying in $\Lambda(N;d')^{n-1}$. With this definition, one has the following identity. 

\begin{lemma}\label{wave packet lemma} Suppose $N, d \in \N$ are such that $d \mid N$ and $N \mid d^2$. Then for all $(\theta, \vec{v}) \in \mathbb{T}_d$ one has
\begin{equation*}
\mathcal{E}\psi_{\theta, \vec{v}}(\vec{x}\,) = e^{2 \pi i (\phi(\vec{x}\,;\vec{\omega}_{\theta}) - \vec{v} \cdot \vec{\omega}_{\theta})/N} \chi_{T_{\theta, \vec{v}}}(\vec{x}\,).
\end{equation*}
\end{lemma}

\begin{proof} It follows from the definitions that
\begin{align*}
\mathcal{E}\psi_{\theta, \vec{v}}(\vec{x}\,) &= \frac{1}{(d')^{n-1}}  \sum_{\vec{\omega} \in \theta} e^{2 \pi i (\phi(\vec{x}\,;\vec{\omega}) - \vec{v}\cdot \vec{\omega}) /N} \\
&= \frac{1}{(d')^{n-1}} \sum_{\omega_1, \dots, \omega_{n-1} = 0}^{d' -1} e^{2 \pi i (\phi(\vec{x}\,;\vec{\omega}_{\theta} + d \vec{\omega}) - \vec{v}\cdot (\vec{\omega}_{\theta} + d \vec{\omega})) /N},
\end{align*}
where $\vec{\omega} = (\omega_1, \dots, \omega_{n-1})$. Since $N \mid d^2$, one may verify that
\begin{equation*}
\phi(\vec{x}\,;\vec{\omega}_{\theta} + d\vec{\omega})  \equiv \phi(\vec{x}\,;\vec{\omega}_{\theta}) + d\big(x' + x_n\partial_{\omega}h(\vec{\omega}_{\theta})\big)\cdot \vec{\omega} \mod N.
\end{equation*}
On the other hand, the basic properties of character sums imply the identity 
\begin{equation*}
\chi_{T_{\theta, \vec{v}}}(\vec{x}\,) =\frac{1}{(d')^{n-1}}\sum_{\omega_1, \dots, \omega_{n-1} = 0}^{d' -1} e^{2 \pi i (x' + x_n\partial_{\omega}h(\vec{\omega}_{\theta}) - \vec{v})\cdot \vec{\omega} /d'}. 
\end{equation*}
Combining these observations, the desired result immediately follows.
\end{proof}

Lemma \ref{wave packet lemma} provides a plethora of functions with which to test the extension operator $\mathcal{E}$. 

\begin{example} The Knapp example introduced in $\S$\ref{basic setup section} falls under the present framework, and simply corresponds to testing the extension operator against a single wave packet.
\end{example}

\begin{example} The constant function 1, which was again considered in $\S$\ref{basic setup section}, can also be analysed via wave packets. Indeed, in this case one has a particularly simple decomposition
\begin{equation*}
1 = \frac{1}{d^{n-1}} \sum_{\theta \in \Theta_d} \psi_{\theta, \vec{0}}(\vec{\omega}) \qquad \textrm{for all $\vec{\omega} \in [\Z/N\Z]_*^{n-1}$}.
\end{equation*} 
If $N \mid d^2$, then applying Lemma \ref{wave packet lemma} to the above identity yields
\begin{equation}\label{wave packets for 1}
\mathcal{E} 1 (\vec{x}\,) = \frac{1}{d^{n-1}} \sum_{\theta \in \Theta_d} e^{2 \pi i \phi(\vec{x}\,;\vec{\omega}_{\theta}) /N} \chi_{T_{\theta, \vec{0}}}(\vec{x}\,).
\end{equation}
Now consider the prototypical example $h(\vec{\omega}) := \omega_1^2 + \dots + \omega_{n-1}^2$ and suppose $N$ is odd. Using \eqref{wave packets for 1}, one may give a conceptually different proof of the estimate
\begin{equation}\label{Fourier decay}
|\mathcal{E} 1 (\vec{x}\,)| \leq \|\vec{x}\,\|^{-(n-1)/2} \qquad \textrm{for all $\vec{x} \in [\Z/N\Z]^n$}
\end{equation}
which was established in $\S$\ref{basic setup section}. For simplicity suppose that $N = d^2$ is a perfect square and $\vec{x} \in [\Z/N\Z]^n$ satisfies $\|\vec{x}\,\| = N$; extending the argument to general $N$ and $\vec{x}$ involves some technicalities which will not be discussed here.\footnote{The proof for general odd $N$ (that is, not necessarily given by a perfect square) relies on evaluating Gauss sums and therefore does not offer a truly alternative approach to \eqref{Fourier decay} from that used in $\S$\ref{basic setup section}. The wave packet method does, however, provide an interesting geometric interpretation of the estimate.} Under these hypotheses, it is easy to see that $\vec{x}$ can lie in at most one of the tubes $T_{\theta, \vec{0}}$ and \eqref{Fourier decay} follows immediately from \eqref{wave packets for 1}. 

The inequality \eqref{Fourier decay} can be interpreted as measuring the decay of the Fourier transform of the normalised counting measure on $\Sigma$; such estimates play an important r\^ole in $\S$\ref{L2 restriction section}. 
\end{example}

\begin{example} Consider once again the paraboloid $h(\vec{\omega}) := \omega_1^2 + \dots + \omega_{n-1}^2$. The full conjectured range of estimates for the extension operator (as computed in $\S$\ref{basic setup section}) would imply the following `endpoint' estimate. 

\begin{conjecture}[Fourier restriction conjecture for the paraboloid over $\Z/N\Z$]\label{endpoint restriction conjecture} For all $\varepsilon > 0$ there exists a constant $C_{\varepsilon} >0$ for which the inequality\footnote{ Recall that here one uses counting measure on the group $G = [\Z/N\Z]^{n-1}$ and normalised counting measure on dual group $[\Z/N\Z]_{*}^{n-1}$ and that these measures define the $\ell^r$ norms. In particular, 
\begin{equation*}
 \| H \|_{\ell^{r}([\Z/N\Z]_*^{n-1})} = \Big(\frac{1}{N^{n-1}} \sum_{\vec{\xi} \in [\Z/N\Z]_*^{n-1}} |H(\vec{\xi}\, )|^r \Big)^{1/r}.
\end{equation*}}
\begin{equation}\label{endpoint restriction}
\|\mathcal{E} H\|_{\ell^{2n/(n-1)}([\Z/N\Z]^n)} \leq C_{\varepsilon} N^{\varepsilon} \| H \|_{\ell^{2n/(n-1)}([\Z/N\Z]_*^{n-1})}
\end{equation}
holds for all odd $N \in \N$.
\end{conjecture}
Note that the constant in this inequality must involve some dependence on $N$, even if $N$ is restricted to powers of a fixed prime, owing to the behaviour of $\mathcal{E}1$, as discussed above and in $\S$\ref{basic setup section}. 

Assume Conjecture \ref{endpoint restriction conjecture} holds and let $N = d^2$ be an odd perfect square. Fix $\tilde{\Theta}_d \subseteq \Theta_d$, assign a choice of $\vec{v}_{\theta} \in [\Z/N\Z]^{n-1}$ to each $\theta \in \tilde{\Theta}_d$ and consider the function
\begin{equation*}
 H := \sum_{\theta \in \tilde{\Theta}_d} r_{\theta} \psi_{\theta, \vec{v}_{\theta}}
\end{equation*}
where each $r_{\theta}$ is a choice of complex coefficient with $|r_{\theta}| = 1$. One may easily compute that
\begin{equation*}
 \|H\|_{\ell^{2n/(n-1)}([\Z/N\Z]_*^{n-1})} =d^{(n+1)(n-1)/2n} |\tilde{\Theta}_d|^{(n-1)/2n} = \Big(\sum_{\theta \in \tilde{\Theta}_d} |T_{\theta}|\Big)^{(n-1)/2n}
\end{equation*}
whilst Lemma \ref{wave packet lemma} implies that
\begin{equation*}
\mathcal{E} H(\vec{x}\,) =  \sum_{\theta \in \tilde{\Theta}_d} r_{\theta} e^{2 \pi i (\phi(\vec{x}\,;\vec{\omega}_{\theta}) - \vec{v}_{\theta} \cdot \vec{\omega}_{\theta})/N} \chi_{T_{\theta}}(\vec{x}\,)
\end{equation*}
where, for notational simplicity, $T_{\theta} := T_{\theta, \vec{v}_{\theta}}$. 
If the $r_{\theta}$ are chosen to be independent, identically distributed random signs ($\pm 1$), then Khintchine's inequality (see, for instance, \cite{Stein1970}*{Appendix D}, or \cite{Haagerup1981} for the precise version used here) implies that the expected value of 
$|\mathcal{E} H (\vec{x}\,)|$ satisfies
\begin{equation*}
\mathbb{E}[|\mathcal{E} H(\vec{x}\,)|] \geq 2^{-1/2} \Big( \sum_{\theta \in \tilde{\Theta}_d} \chi_{T_{\theta}}(\vec{x}\,) \Big)^{1/2}.
\end{equation*}
Thus, the hypothesised endpoint restriction estimate implies that for all $\varepsilon >0$ the inequality
\begin{equation}\label{Kakeya maximal 1}
\big\| \sum_{\theta \in \tilde{\Theta}_d} \chi_{T_{\theta}} \big\|_{\ell^{n/(n-1)}([\Z/N\Z]^n)} \leq 2 C_{\varepsilon/2}^2 N^{\varepsilon} \Big( \sum_{\theta \in \tilde{\Theta}_d} |T_{\theta}| \Big)^{(n-1)/n}
\end{equation}
holds for all odd perfect squares $N \in \N$, where $C_{\varepsilon}$ is the same constant as that appearing in \eqref{endpoint restriction}. This estimate is a geometric statement concerning intersections of tubes in the module $[\Z/N\Z]^n$. In particular, the expression appearing on the left-hand side of \eqref{Kakeya maximal 1} is a discrete analogue of the Kakeya maximal operator (see, for instance, \cites{Katz2002, Wolff1999}). The theory of such maximal operators, which governs the underlying geometry of the restriction problem, is investigated systematically in the following section.
\end{example}





\section{Kakeya sets in \texorpdfstring{ $\Z/N\Z$}{ZNZ}}\label{Kakeya section}




\subsection{Discrete formulations of the Kakeya conjectures} The previous section highlighted a connection between estimates for the parabolic extension operator over $\Z/N\Z$ and a discrete variant of the Kakeya maximal operator. Here the theory of Kakeya sets over $\Z/N\Z$ is explored in a more systematic manner, beginning with a cleaner formulation of the maximal inequality \eqref{Kakeya maximal 1}.

The Kakeya problem over $\Z/N\Z$ concerns configurations of lines in $[\Z/N\Z]^n$ that point in `different directions'. An elegant way to formulate a notion of direction for lines lying in these modules is to use the ring-theoretic construction of the projective space.\footnote{This perspective was recently used in connection with the Kakeya problem by Caruso \cite{Caruso}.}

For future reference it is useful to formulate the definitions at the general level of unital rings. Given a ring $R$ with identity define the $(n-1)$-dimensional sphere $\mathbb{S}^{n-1}(R)$ to be the set of all elements of $R^n$ that have at least one invertible component. In particular, note that $\mathbb{S}^{0}(R) = R^{\times}$ is the group of units of $R$, which acts on the set $\mathbb{S}^{n-1}(R)$ by left multiplication. The $(n-1)$-dimensional projective space $\mathbb{P}^{n-1}(R)$ is defined to be the set of orbits of this action; that is,
\begin{equation*}
\mathbb{P}^{n-1}(R) := \mathbb{S}^{n-1}(R)/\mathbb{S}^0(R).
\end{equation*} 
Finally, given $\omega \in \mathbb{P}^{n-1}(R)$ a set $\ell_{\omega} \subseteq R^n$ is said to be a \emph{line in the direction of} $\omega$ if there exists some $\vec{v} \in R^n$ such that
\begin{equation*}
\ell_{\omega} = \big\{ t \vec{\omega} + \vec{v} : t \in R \big\}
\end{equation*}
for some (and therefore any) choice of representative $\vec{\omega} \in \mathbb{S}^{n-1}(R)$ for $\omega$. In the case $R = \Z/N\Z$ it will often be notationally convenient to write $\mathbb{S}^{n-1}(N)$ and $\mathbb{P}^{n-1}(N)$ rather than $\mathbb{S}^{n-1}(\Z/N\Z)$ and $\mathbb{P}^{n-1}(\Z/N\Z)$. 

\begin{conjecture}[Kakeya maximal conjecture over $\Z/N\Z$]\label{Kakeya maximal conjecture} For all $\varepsilon > 0$ there exists a constant $C_{\varepsilon} >0$ such that the following holds. If $N \in \N$ and $\ell_{\omega}$ is a choice of line in the direction of $\omega$ for each $\omega \in \mathbb{P}^{n-1}(N)$, then 
\begin{equation}\label{Kakeya maximal 2}
\big\| \sum_{\omega \in \mathbb{P}^{n-1}(N)} \chi_{\ell_{\omega}} \big\|_{\ell^{n/(n-1)}([\Z/N\Z]^n)} \leq C_{\varepsilon} N^{\varepsilon} \big( \sum_{\omega \in \mathbb{P}^{n-1}(N)} |\ell_{\omega}| \big)^{(n-1)/n}.
\end{equation}
\end{conjecture}

If all the lines $\ell_{\omega}$ happened to be disjoint, then the above inequality would hold with equality and the constant $C_{\varepsilon}N^{\varepsilon}$ replaced with 1. Thus, the estimate can be interpreted as stating that collections of direction-separated lines in $[\Z/N\Z]^n$ are `almost disjoint'. 

It is not difficult to adapt the analysis of the previous section to show that (at least for $N$ odd) the restriction conjecture for the paraboloid over $\Z/N\Z$ implies the Kakeya maximal conjecture over $\Z/N\Z$. This closely mirrors the euclidean case; as mentioned previously, the relationship between Kakeya and restriction over finite fields is far more tentative \cites{Mockenhaupt2004, Lewko}. 

Given a commutative ring with identity $R$, a set $K \subseteq R^n$ is said to be \emph{Kakeya} if for every $\omega \in \mathbb{P}^{n-1}(R)$ there exists a line $\ell_{\omega}$ in the direction of $\omega$ contained in $K$. The maximal inequality \eqref{Kakeya maximal 2} implies a lower bound on the cardinality of Kakeya sets in $[\Z/N\Z]^n$. Indeed, suppose $K \subseteq [\Z/N\Z]^n$ is Kakeya so that $\bigcup_{\omega \in \mathbb{P}^{n-1}(N)} \ell_{\omega} \subseteq K$ where each $\ell_{\omega}$ is a line in the direction of $\omega \in \mathbb{P}^{n-1}(N)$. Observe that \eqref{Kakeya maximal 2} together with H\"older's inequality imply that
\begin{align}
\nonumber
\sum_{\omega \in \mathbb{P}^{n-1}(N)} |\ell_{\omega}|  &= \big\| \sum_{\omega \in \mathbb{P}^{n-1}(N)} \chi_{\ell_{\omega}} \big\|_{\ell^{1}([\Z/N\Z]^n)} \\
\label{Kakeya maximal 3}
&\leq C_{\varepsilon} N^{\varepsilon} |K|^{1/n} \big( \sum_{\omega \in \mathbb{P}^{n-1}(N)} |\ell_{\omega}| \big)^{(n-1)/n}.
\end{align}
The sum $\sum_{\omega \in \mathbb{P}^{n-1}(N)} |\ell_{\omega}|$ appearing on both sides of this inequality can be explicitly computed. Indeed, 
\begin{equation}\label{line cardinality}
|\ell_{\omega}| = N \quad \textrm{for all $\omega \in \mathbb{P}^{n-1}(N)$}
\end{equation}
whilst the cardinality of the projective space is given by 
\begin{equation}\label{projective cardinality}
|\mathbb{P}^{n-1}(N)| =  N^{n-1}\prod_{p \mid N\, \mathrm{prime}} \sum_{j=0}^{n-1} p^{-j} \geq N^{n-1},
\end{equation} 
where the product is taken over the set of all distinct prime factors of $N$. The latter identity is a direct consequence of the formula 
\begin{equation}\label{sphere cardinality}
|\mathbb{S}^0(N)| = N \prod_{p \mid N\, \mathrm{prime}} (1 - 1/p);
\end{equation}
the details of the (simple) proofs of the identities \eqref{projective cardinality} and \eqref{sphere cardinality} are provided at the end of the section.
 
Rearranging \eqref{Kakeya maximal 3} and applying the identities \eqref{line cardinality} and \eqref{projective cardinality}, one concludes that the Kakeya maximal conjecture implies the following variant of the Kakeya set conjecture.

\begin{conjecture}[Kakeya set conjecture over $\Z/N\Z$]\label{Kakeya conjecture} For all $\varepsilon > 0$ there exists a constant $c_{\varepsilon,n} > 0$ such that the density bound 
\begin{equation*}
\frac{|K|}{N^{n}} \geq c_{\varepsilon,n} N^{-\varepsilon}
\end{equation*}
holds for any Kakeya set $K \subseteq [\Z/N\Z]^n$. 
\end{conjecture}

This can be understood as a discrete analogue of the upper-Minkowski dimension conjecture for Kakeya sets in $\R^n$. It is remarked that similar discrete variants of the Kakeya conjecture have previously appeared in the literature: see, for instance, \cites{Caruso, Dummit2013, Ellenberg2010}.




\subsection{Sharpness of the Kakeya conjecture}

It is natural to ask whether the $\varepsilon$-loss in $N$ is necessary in Conjecture \ref{Kakeya conjecture}: that is, whether there exists a dimensional constant $c_n > 0$ such that $N^{-n} |K| \geq c_{n}$ holds for all Kakeya sets $K \subseteq [\Z/N\Z]^n$ (independently of $N$). It transpires that such an estimate is false, even if one restricts $N$ to vary over powers of a fixed prime.

\begin{proposition}\label{measure zero Kakeya} For all primes $p$ there exists a strictly increasing integer sequence $(\alpha(s))_{s \in \N}$ and family of Kakeya sets $K_s \subseteq [\Z/p^{\alpha(s)}\Z]^n$ such that
\begin{equation*}
\lim_{s \to \infty} \frac{|K_s|}{p^{\alpha(s)n}} = 0. 
\end{equation*}
\end{proposition}

This observation should be contrasted with Dvir's theorem in the finite field setting \cite{Dvir2009}. The latter states that there exists a dimensional constant $c_n > 0$ such that $q^{-n}|K| \geq c_{n}$ holds whenever $K \subseteq \mathbb{F}_q^n$ is a finite field Kakeya set.\footnote{The finite field analogue of the stronger maximal function estimate was established by Ellenberg, Oberlin and Tao in \cite{Ellenberg2010}. It is also useful to contrast the form of the conjectured maximal function estimate \eqref{Kakeya maximal 2} (and, in particular, the (necessary) $\varepsilon$-loss in $N$ in the constant) with the finite field result from \cite{Ellenberg2010}.}

Proposition \ref{measure zero Kakeya} follows by adapting a (euclidean-based) construction due to Sawyer \cite{Sawyer1987} (see also \cites{Wisewell2005, Wolff2003} and \cite{Fraser2016}).  

\begin{proof} It suffices to consider the case $n=2$: the general case then follows by taking the Cartesian product of the set $K_s$ given by the 2-dimensional example with $[\Z/p^{\alpha(s)}\Z]^{n-2}$ for each $s \in \N$. Furthermore, it suffices to construct a sequence of sets containing lines in only those directions which can be represented by an element of the form $(1, \omega)$ for some $\omega \in \Z/p^{\alpha}\Z$. Indeed, one may then form a sequence of true Kakeya sets by taking finite unions of rotated copies of these objects. 

Fixing $p$ and $s \in \N$, let $\alpha := sp^s$. For each $\omega \in \Z/p^{\alpha}\Z$ let $\omega_j \in \{0,1, \dots, p-1\}$ for $0 \leq j \leq \alpha - 1$ denote the coefficients in the $p$-adic expansion of the unique class representative of $\omega$ in $\{0,1 \dots, p^{\alpha} - 1\}$.\footnote{That is, the $\omega_j \in \{0, 1, \dots, p-1\}$ are uniquely defined by the formula $\displaystyle \omega = \big[\sum_{j = 0}^{\alpha-1} \omega_j p^j\big]$.} Using this notation, for each $\omega \in \Z/p^{\alpha}\Z$ define a map $\phi_{\omega} \colon \Z/p^{\alpha}\Z \to  \Z/p^{\alpha}\Z$ by
\begin{equation*}
\phi_{\omega}(t) := t \omega + \big[\sum_{j=0}^{\alpha - 1}  \big\lfloor \frac{j}{s} \big\rfloor \cdot \omega_j p^j \big],
\end{equation*}
where $\lfloor \, \cdot \, \rfloor \colon \R \to \Z$ denotes the floor function. Thus,
\begin{equation*}
\ell_{[(1,\omega)]} := \big\{(t, \phi_{\omega}(t)) : t \in \Z/p^{\alpha}\Z\big\}
\end{equation*}
is a line in the direction of $[(1,\omega)] \in \mathbb{P}^1(p^{\alpha})$. Since, by Fubini,
\begin{equation}\label{measure zero Kakeya 1}
\big|\bigcup_{\omega \in \Z/p^{\alpha}\Z} \ell_{[(1,\omega)]}\big| = \sum_{t \in \Z/p^{\alpha}\Z} |\{ \phi_{\omega}(t) : \omega \in \Z/p^{\alpha}\Z\}|,
\end{equation}
it suffices to show that 
\begin{equation}\label{measure zero Kakeya 2}
|\{ \phi_{\omega}(t) : \omega \in \Z/p^{\alpha}\Z\}| \leq p^{\alpha - s} \qquad \textrm{for all $t \in \Z/p^{\alpha}\Z$.}
\end{equation}
Indeed, once this is established one may define $K_s$ to be the union of lines appearing on the left-hand side of \eqref{measure zero Kakeya 1}, noting that the above inequality implies that $p^{-2\alpha}|K_s| \leq p^{-s}$.

 Fix $t \in \Z/p^{\alpha}\Z$ and identify this element with a coset representative $t \in \{0,1, \dots, p^{\alpha} - 1\}$. Let $t' \in \{0,\dots, p^{s} - 1\}$ be the unique element satisfying $t' \equiv -t \bmod p^s$ and define $k := st'$, noting that $0 \leq k \leq s(p^s - 1) \leq \alpha-s$. It follows that
\begin{equation*}
t \equiv -\big\lfloor\frac{j}{s}\big\rfloor \mod p^s \quad \textrm{for all $k \leq j \leq k+ s-1$.}
\end{equation*}

For any $\omega' \in \Z/p^{\alpha}\Z$ there exists a unique $\omega \in \Lambda(p^{\alpha}, p^{\alpha-k})$ such that $|\omega - \omega'| \preceq p^{\alpha - k}$. Here $\Lambda(p^{\alpha}, p^{\alpha-k})$ is the maximal set of $p^{\alpha - k}$-separated points in $\Z/p^{\alpha}\Z$ defined in \eqref{Lambda definition}. In particular, $\omega_j = \omega_j'$ for $0 \leq j \leq k-1$ and so
\begin{equation*}
|\phi_{\omega}(t) - \phi_{\omega'}(t)| = \bigg|\sum_{j=k}^{\alpha-1} \big(t + \big\lfloor\frac{j}{s}\big\rfloor\big)\big(\omega_j-\omega_j'\big)p^j\bigg|.
\end{equation*}
The construction ensures that $\big|t + \lfloor\tfrac{j}{s}\rfloor\big| \preceq p^{\alpha-s}$ for all $k \leq j \leq k+s-1$ from which it follows that $|\phi_{\omega}(t) - \phi_{\omega'}(t)| \preceq p^{\alpha-k-s}$. Thus,
\begin{equation*}
\{ \phi_{\omega}(t) : \omega \in \Z/p^{\alpha}\Z\} \subseteq \bigcup_{\omega \in \Lambda(p^{\alpha}, p^{\alpha-k})} \mathcal{B}_{p^{\alpha-k-s}}(\phi_{\omega}(t)),
\end{equation*}
which immediately yields \eqref{measure zero Kakeya 2}.
\end{proof}




\subsection{Standard Kakeya estimates over \texorpdfstring{$\Z/N\Z$}{ZNZ}}

Many of the standard techniques used to investigate the euclidean Kakeya problem can be adapted to study Conjecture \ref{Kakeya maximal conjecture} and Conjecture \ref{Kakeya conjecture}. Here two examples are given: the standard $L^2$ maximal argument of C\'ordoba \cite{Cordoba1977} and a basic slicing argument. The former resolves Conjecture \ref{Kakeya maximal conjecture} (and therefore also Conjecture \ref{Kakeya conjecture}) in the $n=2$ case, whilst the latter provides a discrete analogue of the elementary $(n+1)/2$-dimensional bound for Kakeya sets.

\subsubsection*{C\'ordoba's argument} By adapting the classical argument of \cite{Cordoba1977}, one may establish the following elementary bound (which implies Conjecture \ref{Kakeya maximal conjecture} in the $n=2$ case). 

\begin{proposition}\label{Cordoba proposition} For all $\varepsilon > 0$ there exists a constant $C_{\varepsilon} >0$ such that the following holds. If $N \in \N$ and $\ell_{\omega}$ is a choice of line in the direction of $\omega$ for each $\omega \in \mathbb{P}^{n-1}(N)$, then 
\begin{equation}\label{Kakeya maximal 4}
\big\| \sum_{\omega \in \mathbb{P}^{n-1}(N)} \chi_{\ell_{\omega}} \big\|_{\ell^{2}([\Z/N\Z]^n)} \leq C_{\varepsilon} N^{n/2-1 + \varepsilon} \big( \sum_{\omega \in \mathbb{P}^{n-1}(N)} |\ell_{\omega}| \big)^{1/2}.
\end{equation}
\end{proposition}

At its heart, the proof of Proposition \ref{Cordoba proposition} (that is to say, C\'ordoba's argument) exploits the following simple geometric fact (here and below $\mathbb{F}$ denotes a field):

\begin{equation}\label{geometric fact 1}
\textrm{A pair of direction-separated lines in $\mathbb{F}^n$ can intersect in at most one point.}
\end{equation}

The relevance of \eqref{geometric fact 1} is most clearly understood by considering Proposition \ref{Cordoba proposition} in the finite field setting, given by restricting $N = p$ to vary over primes $p$. Indeed, in this case the argument is particularly elementary: by \eqref{geometric fact 1} one has
\begin{equation*}
\big\| \sum_{\omega \in \mathbb{P}^{n-1}(p)} \chi_{\ell_{\omega}} \big\|_{\ell^{2}([\Z/p\Z]^n)}^2 = \sum_{\omega, \omega' \in \mathbb{P}^{n-1}(p)} |\ell_{\omega} \cap \ell_{\omega'}| \leq |\mathbb{P}^{n-1}(p)|\big(|\mathbb{P}^{n-1}(p)| + p - 1\big),
\end{equation*}
and bounding the right-hand side of this inequality using \eqref{projective cardinality} yields the desired estimate. 

The original euclidean problem, as investigated in \cite{Cordoba1977}, studies configurations of $\delta$-tubes in $\R^n$ rather than lines. In this context one does not work with \eqref{geometric fact 1} \emph{per se}, but rather a quantitative version of this fact, which states that the measure of the intersection of two tubes is inversely proportional to the angle between their directions. In this respect, the $\Z/N\Z$ setting behaves much more like euclidean space than a vector space over a finite field. Indeed, owing to the presence of zero divisors, \eqref{geometric fact 1} can fail dramatically for lines over $\Z/N\Z$: a pair of direction-separated lines in $[\Z/N\Z]^n$ can meet at many points. However, in analogy with tubes in $\R^n$, the number of points of intersection is inversely proportional to the angle between the directions. 

To make the above discussion precise requires a notion of angle between elements of $\mathbb{P}^{n-1}(N)$; such a notion is formulated presently. Let $\vec{\omega} = (\omega_1, \dots, \omega_n)$, $\vec{\omega}' = (\omega_1', \dots, \omega_n') \in \mathbb{S}^{n-1}(N)$ be class representatives of $\omega, \omega' \in \mathbb{P}^{n-1}(N)$, respectively, and define the angle $\measuredangle(\omega, \omega')$ by
\begin{equation}\label{angle definition}
\measuredangle(\omega, \omega') := \max_{1 \leq i < j \leq n} \big| \det
\begin{pmatrix}
\omega_i  & \omega_j \\
\omega_i' & \omega_j'
\end{pmatrix}
\big|.
\end{equation}
Here $|\,\cdot\,|$ is the size function on $\Z/N\Z$ introduced at the beginning of $\S$\ref{basic setup section}: that is, $|x| := N/\gcd(x,N)$ for all $x \in \Z/N\Z$. Note that the right-hand side of \eqref{angle definition} does not depend on the choice of representatives $\vec{\omega}$ and $\vec{\omega}'$ and therefore $\measuredangle(\omega, \omega')$ is well-defined. A few further comments regarding the definition of $\measuredangle(\omega, \omega')$ are in order.
\begin{enumerate}[i)]
\item The definition \eqref{angle definition} is motivated by the formula $\sin \measuredangle(\omega, \omega') = |\omega \wedge \omega'|$ for the angle between unit vectors in $\R^n$. Note, in particular, that $|\omega \wedge \omega'|$ can be written in terms of the determinants 
\begin{equation*}
\det
\begin{pmatrix}
\omega_i  & \omega_j \\
\omega_i' & \omega_j'
\end{pmatrix} \qquad \textrm{for $1 \leq i < j \leq n$}
\end{equation*}
via the Cauchy--Binet formula.
\item By adapting an argument of Caruso \cite{Caruso}, one may easily show that
\begin{equation*}
\measuredangle(\omega, \omega') = \min_{(\vec{\omega},\vec{\omega}') \in \omega \times \omega' } \max_{1 \leq j \leq n} |\omega_j - \omega_j'|,
\end{equation*}
where the minimum is over all pairs of class representatives $\vec{\omega} = (\omega_1, \dots, \omega_n)$, $\vec{\omega}' = (\omega_1', \dots, \omega_n')$ for $\omega, \omega'$, respectively. 
\end{enumerate}

In place of the basic geometric fact \eqref{geometric fact 1} valid over fields, over $\Z/N\Z$ there is the following quantitative statement. 

\begin{lemma}\label{angle lemma} If $\omega, \omega' \in \mathbb{P}^{n-1}(N)$, then 
\begin{equation*}
|\ell_{\omega} \cap \ell_{\omega'}| \leq \frac{N}{\measuredangle(\omega, \omega')}.
\end{equation*}
\end{lemma}

\begin{proof} If $\measuredangle(\omega, \omega') = 1$, then the result trivially holds and so one may assume that $\measuredangle(\omega, \omega') \succ 1$. Let $(\omega_1, \dots, \omega_n)$, $(\omega_1',\dots, \omega_n') \in \mathbb{S}^{n-1}(N)$ be class representatives for $\omega, \omega'$, respectively. It is easy to see that the cardinality of $\ell_{\omega} \cap \ell_{\omega'}$ is given by the number of solutions $(t,t') \in [\Z/N\Z]^2$ to the system $\begin{pmatrix}
t & t'
\end{pmatrix}
\cdot \Omega = \vec{v}$ for some $\vec{v} \in [\Z/N\Z]^n$, where
\begin{equation*}
\Omega :=
\begin{pmatrix}
\omega_1  & \dots & \omega_n  \\
\omega_1' & \dots & \omega_n' 
\end{pmatrix} .
\end{equation*}
By definition, there exists some $2 \times 2$ submatrix $A$ of $\Omega$ such that $|\det A| = \measuredangle(\omega, \omega')$. Since $\measuredangle(\omega, \omega') \succ 1$, Lemma \ref{linear system lemma} of the appendix implies that for any $\vec{b} \in [\Z/N\Z]^2$ the system
\begin{equation*}
\begin{pmatrix}
t & t'
\end{pmatrix}
\cdot A = \vec{b}
\end{equation*}
has at most $N/|\det A|$ solutions, and the desired result follows. 
\end{proof}

Given this inequality, it is a simple exercise to translate C\'ordoba's approach \cite{Cordoba1977} into the current setting (see also \cites{Caruso, Dummit2013}) and thereby prove Proposition \ref{Cordoba proposition}. 

\begin{proof}[Proof (of Proposition \ref{Cordoba proposition})] Expanding the left-hand $\ell^2$-norm, one obtains
\begin{align*}
\big\| \sum_{\omega \in \mathbb{P}^{n-1}(N)} \chi_{\ell_{\omega}} \big\|_{\ell^{2}([\Z/N\Z]^n)}^2 &= \sum_{\omega \in \mathbb{P}^{n-1}(N)} \sum_{d \mid N} \sum_{\substack{\omega' \in \mathbb{P}^{n-1}(N) \\ \measuredangle(\omega, \omega') = d}} | \ell_{\omega} \cap \ell_{\omega'} |.
\end{align*}
By Lemma \ref{angle lemma} and \eqref{line cardinality}, it follows that
\begin{equation*}
\big\| \sum_{\omega \in \mathbb{P}^{n-1}(N)} \chi_{\ell_{\omega}} \big\|_{\ell^{2}([\Z/N\Z]^n)}^2 \leq \sum_{\omega \in \mathbb{P}^{n-1}(N)} |\ell_{\omega}| \sum_{d \mid N} \frac{|\{\omega' \in \mathbb{P}^{n-1}(N) : \measuredangle(\omega, \omega') = d\}|}{d}.
\end{equation*}
Recalling the standard asymptotics for the divisor function (see, for example, \cite{Hardy2008}*{Chapter XVIII}), it suffices to show that for all $\varepsilon > 0$ there exists a constant $C_{\varepsilon} >0$ such that 
\begin{equation*}
|\{\omega' \in \mathbb{P}^{n-1}(N) : \measuredangle(\omega, \omega') = d\}| \leq C_{\varepsilon} N^{ n-2 + \varepsilon}d.
\end{equation*}
The simple proof of this inequality is postponed until the end of the section.
\end{proof}

\subsubsection*{The slicing argument} Although C\`ordoba's argument is effective for $n=2$, it produces very poor estimates in higher dimensions. Here a $\Z/N\Z$-analogue of the elementary $(n+1)/2$-dimensional lower bound for Kakeya sets is established. This gives improved partial results towards Conjecture \ref{Kakeya conjecture} in higher dimensions. 

\begin{proposition}\label{slicing proposition} For all $\varepsilon > 0$ there exists a constant $c_{\varepsilon,n} > 0$ such that the density bound 
\begin{equation*}
\frac{|K|}{N^{n}} \geq c_{\varepsilon,n} N^{-(n-1)/2 - \varepsilon}
\end{equation*}
holds for any Kakeya set $K \subseteq [\Z/N\Z]^n$. 
\end{proposition}

In essence, the proof of Proposition \ref{slicing proposition} relies on the following variant of the key geometric fact \eqref{geometric fact 1} used above:

\begin{equation*}
    \textrm{For $x, y \in \mathbb{F}^n$ distinct there exists precisely one line in $\mathbb{F}^n$ passing through both $x$ and $y$.}
\end{equation*}

Once again, owing to the presence of zero divisors, this property no longer holds over $\Z/N\Z$. In its place there is the following quantitative version, where the separation between the points is quantified.  

\begin{lemma}\label{separated points lemma} Let $\ell_{\omega}, \ell_{\omega'}$ be lines in $[\Z/N\Z]^n$ in the directions $\omega, \omega' \in \mathbb{P}^{n-1}(N)$, respectively. If there exist points $\vec{x}, \vec{y} \in \ell_{\omega} \cap \ell_{\omega'}$ such that $\|\vec{x} - \vec{y}\| = N$, then $\ell_{\omega} = \ell_{\omega'}$. 
\end{lemma}  

\begin{proof} The simple proof is left to the reader.
\end{proof}

Proposition \ref{slicing proposition} is proved by combining Lemma \ref{separated points lemma} with (an adaptation of) a simple and well-known slicing argument from euclidean analysis (see, for instance, \cite{Katz2002}). 

\begin{proof} Suppose $K \subseteq [\Z/N\Z]^n$ is a Kakeya set and let
\begin{equation*}
K[t] := K \cap \{ \vec{x} \in [\Z/N\Z]^n : x_1 = t\} \quad \textrm{for all $t \in \Z/N\Z$.}
\end{equation*}
Let $\bar{C}> 0$ be a uniform constant, to be determined later in the proof, and define 
\begin{equation*}
E := \Big\{ t \in \Z / N\Z: |K[t]| \leq \frac{\bar{C}\log N}{N}|K| \Big\}.
\end{equation*}
It follows from Chebyshev's inequality that 
\begin{equation}\label{slicing 1}
|E| \geq (1 - 1/\bar{C}\log N)N.
\end{equation} 
If $\bar{C}$ is chosen to be sufficiently large, then, by pigeonholing, $E$ will necessarily contain a pair of well-separated points. Indeed, without loss of generality (by translating the Kakeya set) one may suppose that $\vec{0}\in E$ and, recalling \eqref{sphere cardinality}, it follows that
\begin{equation*}
|\{t \in \Z/N\Z : |t| < N \}| = N - |\mathbb{S}^0(N)| = N \cdot \big( 1 - \prod_{p \mid N \,\mathrm{prime}} (1 - 1/p) \big).
\end{equation*}
By Mertens' theorem (see, for instance, \cite{Hardy2008}*{Chapter XXII}), the constant $\bar{C} > 0$ may be chosen so that
\begin{equation*}
\prod_{p \mid N\, \mathrm{prime}} (1 - 1/p) \geq \prod_{\substack{2 \leq p \leq N \\ p\, \mathrm{prime}}} (1- 1/p) \geq \frac{2}{\bar{C}\log N} .
\end{equation*}
Combining these observations with \eqref{slicing 1}, one concludes that 
\begin{equation*}
|\{t \in \Z/ N\Z : |t| < N \}| \leq N(1 - 2/\bar{C}\log N) < |E|,
\end{equation*}
and so there must exist an element $t \in E$ with $|t| = N$. By applying a group automorphism to the set $K$ one may further assume that $t = 1$.   

Let $\Omega(N) \subseteq \mathbb{P}^{n-1}(N)$ denote the set of all $\omega \in \mathbb{P}^{n-1}(N)$ for which the first component of some (and therefore every) class representative $\vec{\omega}\in \mathbb{S}^{n-1}(N)$ lies in $\mathbb{S}^0(N)$. Thus, $\Omega(N)$ is the collection of orbits of the free action of $\mathbb{S}^0(N)$ on $\mathbb{S}^0(N) \times [\Z/N\Z]^{n-1}$ and therefore has cardinality $N^{n-1}$. Furthermore, for any $\omega \in \Omega(N)$ the line $\ell_{\omega}$ intersects each of the slices $K[0]$ and $K[1]$ at a unique point, denoted by $\ell_{\omega}[0]$ and $\ell_{\omega}[1]$, respectively. On the other hand, Lemma \ref{separated points lemma} implies that for any pair of points $(\vec{x}\,, \vec{y}\,) \in K[0] \times K[1]$ there exists at most one line $\ell_{\omega}$ such that $\ell_{\omega}[0] = \vec{x}$ and $\ell_{\omega}[1] = \vec{y}$. Consequently, 
\begin{equation*}
N^{n-1} = |\Omega(N)| \leq |K[0]||K[1]| \leq \frac{\bar{C}^2\log^2N}{N^2} |K|^2,
\end{equation*}
which implies the desired inequality. 
\end{proof}




\subsection{Remaining estimates and identities} The proofs of a small number of basic estimates and identities were not presented in the above text; these remaining issues are collected in the following lemma and addressed presently. 

\begin{lemma}\label{left over lemma} For all $N \in \N$ the following statements hold.
\begin{enumerate}[i)] 
\item\label{projective cardinality lemma} The cardinalities of $\mathbb{S}^0(N)$ and $\mathbb{P}^{n-1}(N)$ are given by the formulae
\begin{equation*}
|\mathbb{S}^0(N)| = N \prod_{p \mid N\, \mathrm{prime}} (1 - 1/p) \quad \textrm{and} \quad |\mathbb{P}^{n-1}(N)| =  N^{n-1}\prod_{p \mid N\, \mathrm{prime}} \sum_{j=0}^{n-1} p^{-j}.
\end{equation*}
\item\label{projective cap lemma} For all $\varepsilon > 0$ there exists a constant $C_{\varepsilon} >0$ such that 
\begin{equation*}
|\{\omega' \in \mathbb{P}^{n-1}(N) : \measuredangle(\omega, \omega') = d\}| \leq C_{\varepsilon} N^{n-2 - \varepsilon}d
\end{equation*}
holds for all $\omega \in \mathbb{P}^{n-1}(N)$ and all $d \mid N$.
\end{enumerate}
\end{lemma}

\begin{proof}[Proof of Lemma \ref{left over lemma} \ref{projective cardinality lemma})] The cardinality of the group of units is well-known, but nevertheless a proof is included in order to express the argument in terms of the $|\,\cdot\,|$, $\preceq$ notation introduced in this article. Observe that 
\begin{equation*}
\{t \in \Z/ N\Z : |t| \prec N \} = \bigcup_{p \mid N\, \mathrm{prime}} \mathcal{B}_{N/p}  
\end{equation*}
and, by the inclusion-exclusion principle,
\begin{align*}
\big|\bigcup_{p \mid N\, \mathrm{prime}} \mathcal{B}_{N/p} \big| &= \sum_{k=1}^{\omega(N)} (-1)^{k+1} \sum_{\substack{p_1 < \dots < p_k \\ p_j \mid N \, \mathrm{prime} }} \big| \bigcap_{j=1}^{k} \mathcal{B}_{N/p_{j}}\big| \\
&= \sum_{k=1}^{\omega(N)} (-1)^{k+1} \sum_{\substack{p_1 < \dots < p_k \\ p_j \mid N \, \mathrm{prime} }} \big|\mathcal{B}_{N/p_{1}\dots p_{k}}\big| \\
&= N\cdot\big(1 - \prod_{p \mid N\, \mathrm{prime}} (1 - 1/p)\big),
\end{align*}
where $\omega(N)$ is the number of distinct prime factors of $N$. The desired identity for $|\mathbb{S}^0(N)|$ immediately follows. 

Since $|\mathbb{P}^{n-1}(N)|$ is a multiplicative function of $N$, it suffices to establish the second formula for $N = p^{\alpha}$ a power of a prime $p$; in this case, the desired result was observed by Caruso \cite{Caruso}. A slightly different argument to that of \cite{Caruso} is as follows. Since the action of $\mathbb{S}^0(N)$ on $\mathbb{S}^{n-1}(N)$ is free, one deduces that
\begin{equation*}
|\mathbb{P}^{n-1}(N)| = \frac{|\mathbb{S}^{n-1}(N)|}{|\mathbb{S}^0(N)|} = \frac{N^n - (N-|\mathbb{S}^0(N)|)^n}{|\mathbb{S}^0(N)|}
\end{equation*}
and, after a short computation, the result follows from the above formula for $|\mathbb{S}^0(N)|$. 
\end{proof}

\begin{proof}[Proof of Lemma \ref{left over lemma} \ref{projective cap lemma})] The argument here is rather crude and more precise estimates could be obtained (see \cite{Caruso}). Nevertheless, the resulting bounds suffice for the purposes of this article. 

Fix $\omega \in \mathbb{P}^{n-1}(N)$ and $d \mid N$ and define
\begin{equation*}
P^{n-1}(\omega; d) := \{\omega' \in \mathbb{P}^{n-1}(N) : \measuredangle(\omega, \omega') = d\},
\end{equation*}
so that the desired estimate reads
\begin{equation*}
|P^{n-1}(\omega; d)| \leq C_{\varepsilon} N^{n-2 + \varepsilon}d.
\end{equation*}
The group of units $\mathbb{S}^0(N)$ acts freely on 
\begin{equation*}
S^{n-1}(\omega ;d) := \Big\{ (\omega_1', \dots, \omega_n') \in \mathbb{S}^{n-1}(N) : \max_{1 \leq i < j \leq n} \big| \det
\begin{pmatrix}
\omega_i  & \omega_j \\
\omega_i' & \omega_j'
\end{pmatrix}
\big| = d \Big\},
\end{equation*}
and it follows that
\begin{equation}\label{projective cap 1}
|P^{n-1}(\omega; d)| = \frac{|S^{n-1}(\omega; d)|}{|\mathbb{S}^0(N)|}.
\end{equation}

If $d = N$, then one may use the trivial estimate 
\begin{equation*}
|S^{n-1}(\omega; d)| \leq |\mathbb{S}^{n-1}(N)| \leq nN^{n-1}|\mathbb{S}^0(N)|
\end{equation*}
which, combined with \eqref{projective cap 1}, yields the desired bound. 

Now suppose $d \mid N$ is a proper divisor and fix a representative $(\omega_1, \dots, \omega_n) \in \mathbb{S}^{n-1}(N)$ of $\omega$. Without loss of generality, one may assume that $\omega_{1} \in \mathbb{S}^0(N)$ and, by possibly choosing an alternative class representative, moreover, that $\omega_1 = 1$. If $\vec{\omega}' = (\omega_1', \dots, \omega_n') \in S^{n-1}(\omega; d)$, then it follows that $\omega_1' \in \mathbb{S}^0(N)$. Indeed, otherwise $|\omega_1'| \prec N$ and there must exist some $2 \leq j \leq n$ such that $\omega_j' \in \mathbb{S}^0(N)$; in this case
\begin{equation*}
\big| \det
\begin{pmatrix}
1  & \omega_j \\
\omega_1' & \omega_j'
\end{pmatrix}
\big| = N,
\end{equation*}
contradicting the assumption that $\vec{\omega}' \in S^{n-1}(\omega; d)$. Thus, one deduces that
\begin{equation*}
S^{n-1}(\omega; d) \subseteq \{ \vec{\omega}' \in \mathbb{S}^0(N) \times [\Z/N\Z]^{n-1} : \big| \det
\begin{pmatrix}
1  & \omega_2 \\
\omega_1' & \omega_2'
\end{pmatrix}
\big| \leq d \Big\}.
\end{equation*}
For any fixed $\omega_1' \in \mathbb{S}^0(N)$ and $\varepsilon >0$ there exists some $C_{\varepsilon} >0$ such that 
\begin{equation*}
|\big\{ \omega_2' \in \Z/N\Z : |\omega_2' - \omega_1'\omega_2| \leq d \}| = \big| \bigcup_{d' \mid N : d' \leq d} \mathcal{B}_{d'} \big| \leq \sum_{d' \mid N : d' \leq d} d' \leq C_{\varepsilon} N^{\varepsilon} d.
\end{equation*}
Consequently, 
\begin{equation*}
|S^{n-1}(\omega; d) | \leq C_{\varepsilon} N^{n-2+\varepsilon} d|\mathbb{S}^0(N)|
\end{equation*}
and combining this inequality with \eqref{projective cap 1} concludes the proof.
\end{proof}




\section{Fourier restriction over \texorpdfstring{$\Z/p^{\alpha}\Z$}{ZpZ} and \texorpdfstring{$\Q_p$}{Qp}}\label{p adic section}




\subsection{Analysis over the \texorpdfstring{$p$}{p}-adic field} In this section the key features of the $\Z/p^{\alpha}\Z$ formulation of the restriction problem (that is, Problem \ref{p power problem}) are described. In particular, a correspondence principle is demonstrated that allows one to lift the analysis from the finite rings $\Z/p^{\alpha}\Z$ to the field of $p$-adic numbers $\Q_p$. This correspondence helps to explain many of the apparent similarities between the $\Z/p^{\alpha}\Z$ and euclidean theories, since the fields $\Q_p$ and $\R$ are in many respects related. Furthermore, there are a number of euclidean-based techniques which have no obvious counterpart or are considerably more difficult to implement in the discrete setting, but can be easily adapted to work over $\Q_p$. Thus, lifting the problem to the $p$-adics often significantly simplifies the analysis (a striking example of this occurs when one studies the restriction theory for the moment curve; this is described in detail in $\S$\ref{Fourier restriction for curves section}). The $p$-adic field also has a relatively simple algebraic structure since, in particular, there are no zero divisors.

Before proceeding some basic facts regarding analysis over $\Q_p$ are reviewed, and the relevant notational conventions are established. Fixing a prime $p$, recall that the $p$-adic absolute value $|\,\cdot\,|_p \colon \Z \to \{0, p^{-1}, p^{-2}, \dots\}$ is defined by 
\begin{equation*}
|x|_p := \left\{ \begin{array}{ll}
p^{-k} & \textrm{if $x \neq 0$ and $p^{k}\,\| x$ for $k \in \N_0$} \\
0 & \textrm{otherwise}
\end{array} \right. ,
\end{equation*}
where the notation $p^k \,\| \theta$ is used to denote that $p^k$ divides $\theta$ (that is, $p^k \mid \theta$) and no larger power of $p$ divides $\theta$. The function $|\,\cdot\,|_p$ uniquely extends to a non-archimedean absolute value on the rationals $\Q$.\footnote{That is, $|\,\cdot\,|_p \colon \Q \to [0, \infty)$ satisfies the following properties:
\begin{enumerate}[i)]
\item (Positive definite) $|x|_p \geq 0$ for all $x \in \Q$ and $|x|_p = 0$ if and only if $x=0$;
\item (Multiplicative) $|xy|_p = |x|_p|y|_p$ for all $x, y \in \Q$;
\item (Strong triangle inequality) $|x+y|_p \leq \max\{|x|_p, |y|_p\}$ for all $x, y \in \Q$. 
\end{enumerate}} The field of $p$-adic numbers $\Q_p$ is defined to be the metric completion of $\Q$ under the metric induced by $|\,\cdot\,|_p$. One may verify that $\Q_p$ indeed has a natural field structure and contains $\Q$ as a subfield. 

Any element $x \in \Q_p\setminus\{0\}$ admits a unique $p$-adic series expansion 
\begin{equation}\label{p adic expansion}
x = \sum_{j=J}^{\infty} x_j p^j
\end{equation}  
where $J \in \Z$, $x_j \in \{0,1, \dots, p-1\}$ for all $j \in \Z$ with $x_J \neq 0$ (and $x_j := 0$ for $j < J$). The sum is understood as the limit of a sequence of rationals, where the convergence is with respect to the $p$-adic absolute value. In this case, $|x|_p = p^{-J}$. The ring of $p$-adic numbers $\Z_p$ is defined to be the set comprised of 0 together with all the elements $x \in \Q_p \setminus\{0\}$ for which $J \geq 0$ in the expansion \eqref{p adic expansion}. Thus, $\Z_p = \{x \in \Q_p : |x|_p \leq 1\}$, and this clearly forms a subring of $\Q_p$ by the multiplicative property of the absolute value.

The field $\Q_p$ is a locally compact abelian group under the addition operation and the Haar measure of a Borel subset $E \subseteq \Q_p$ is denoted by $|E|$; this measure is normalised so that $|\Z_p| = 1$. The notation $\ud x$ is used to indicate that an integral is taken with respect to Haar measure (hence, $|E| = \int_{\Q_p} \chi_E(x)\,\ud x$ for all $E \subseteq \Q_p$ Borel). For any $r > 0$ and $x \in \Q_p$ the ball $B_r(x)$ is defined by
\begin{equation*}
B_r(x) := \{y \in \Q_p : |x - y|_p \leq r\};
\end{equation*}
these balls are \emph{not} defined using a strict inequality so that, for instance, $\Z_p = B_1(0)$. For each $\alpha \in \Z$ the ball $B_{p^{\alpha}}(0)  = p^{-\alpha}\Z_p$ is an additive subgroup of $\Q_p$ (furthermore, if $\alpha \leq 0$, then $B_{p^{\alpha}}(0)$ is an ideal of $\Z_p$), and all other balls of radius $p^{\alpha}$ arise as cosets of $B_{p^{\alpha}}(0)$. It immediately follows from the translation invariance property of the Haar measure (together with the choice of normalisation) that $|B_{p^{\alpha}}(x)| = p^{\alpha}$ for all $\alpha \in \Z$ and $x \in \Q_p$.  

There is an alternative algebraic description of $\Z_p$ as the inverse limit of the inverse system of groups $(\Z/p^{\alpha}\Z)_{\alpha \in \N}$: that is, 
\begin{equation*}
\Z_p = \varprojlim_{\alpha \in \N} \Z/p^{\alpha}\Z.
\end{equation*}
The $p$-adic numbers $\Q_p$ can then be described algebraically as the field of fractions of $\Z_p$. This perspective will not feature heavily here, but it is noted that that the inverse system induces a family of natural projection homomorphisms $\pi_{\alpha} \colon \Z_p \to \Z/p^{\alpha}\Z$. The $\pi_{\alpha}$ are given by reduction modulo the ideal $p^{\alpha}\Z_p$ and can be expressed in terms of the $p$-adic expansion; in particular,
\begin{equation*}
\pi_{\alpha}(x) = \big[ \sum_{j=0}^{\alpha-1} x_j p^j \big]  \qquad \textrm{for $x = \sum_{j=0}^{\infty} x_j p^j \in \Z_p$ and $\alpha \in \N$.} 
\end{equation*}

The vector space $\Q_p^n$ is endowed with the norm 
\begin{equation*}
|x|_p := \max_{1 \leq j \leq n} |x_j|_p \qquad \textrm{for all $x = (x_1, \dots, x_n) \in \Q_p^n$.}
\end{equation*}
All the above definitions and conventions then naturally extend to vector spaces $\Q_p^n$. 

By the first isomorphism theorem, the $\pi_{\alpha} \colon \Z_p^n \to [\Z/p^{\alpha}\Z]^n$ induce a natural isomorphism between $[\Z_p/p^{\alpha}\Z_p]^n$ and $[\Z/p^{\alpha}\Z]^n$. In particular, the sets $\pi_{\alpha}^{-1}(\vec{x}\,)$ for $\vec{x} \in [\Z/p^{\alpha}\Z]^n$ are precisely the cosets of $p^{\alpha}\Z_p$ and so
\begin{equation}\label{projection equivalence}
\pi^{-1}\{\vec{x}\} = B_{p^{-\alpha}}(y) \qquad \textrm{for all $\vec{x} \in [\Z/p^{\alpha}\Z]^n$ and $y \in \pi^{-1}\{\vec{x}\}$.}
\end{equation}
This is the key observation which governs the correspondence principle.




\subsection{Restriction and Kakeya over the \texorpdfstring{$p$}{p}-adics} The $p$-adic field $\Q_p$ is self-dual in the Pontryagin sense. In particular, if one fixes an additive character $e \colon \Q_p \to \mathbb{T}$ such that $e$ restricts to the constant function 1 on $\Z_p$ and to a non-principal character on $p^{-1}\Z_p$, then for any integrable $f \colon \Q_p^n \to \C$ the Fourier transform $\hat{f}$ can be defined by
\begin{equation*}
\hat{f}(\xi) := \int_{\Q_p^n} f(x)e(-x \cdot \xi ) \,\ud x \qquad \textrm{for all $\xi \in \Q_p^n$.}
\end{equation*}
It will be useful to work with an explicit choice of character $e$. Define the fractional part function $\{\,\cdot\,\}_p \colon \Q_p \to \Q$ as follows: given $x \in \Q_p$ with $p$-adic expansion $\sum_{j = J}^{\infty} x_jp^j$, let $\{x\}_p := \sum_{j=J}^{-1} x_j p^j$. Observe that $\{x\}_p = 0$ if and only if $x \in \Z_p$. Defining $e \colon \Q_p \to \mathbb{T}$ by
\begin{equation*}
e(x) := e^{2 \pi i \{x\}_p} \quad \textrm{for all $x \in \Q_p$,}
\end{equation*}
it is easy to check that this function has the desired properties.

Fix $1 \leq d \leq n-1$ and $P_{n-d+1}, \dots, P_n  \in \Z[X_1, \dots, X_d]$. Let $\Sigma \subseteq \Z_p^n$ denote the image of the the mapping 
\begin{equation}\label{graph parametrised surface}
\Gamma \colon \omega \mapsto (\omega, P_{n-d-1}(\omega), \dots, P_n(\omega))
\end{equation}
as a function $\Z_p^d \to \Z_p^n$ and $\mu$ the measure on $\Sigma$ given by the push-forward of the Haar measure on $\Z_p^d$ under \eqref{graph parametrised surface}. One is interested in studying Fourier restriction estimates of the form
\begin{equation}\label{p adic restriction}
\|\hat{f}|_{\Sigma}\|_{L^s(\mu)} \leq C\|f\|_{L^r(\Q_p^n)}.
\end{equation}

The conjectural range of estimates for restriction to, say, a compact piece of the paraboloid over $\Q_p$ is easily seen to imply a $p$-adic version of the Kakeya conjecture. To make this precise, first note that the ring-theoretic definitions of projective space, lines in a given direction and Kakeya sets, as described in $\S$\ref{Kakeya section}, can all be applied to $\Z_p$, and so it makes sense to discuss Kakeya sets $K \subseteq \Z_p^n$. By essentially a repeat of the discussion from $\S$\ref{wave packet section} and $\S$\ref{Kakeya section}, the study of restriction estimates over $\Q_p$ leads one to consider the following geometric problem.\footnote{For brevity, only the $p$-adic Kakeya set conjecture is stated, but it certainly makes sense to also consider the corresponding maximal conjecture.}

\begin{conjecture}[Kakeya set conjecture over $\Q_p$]\label{p adic Kakeya set conjecture} For all $\varepsilon > 0$ there exists a constant $c_{\varepsilon,n} > 0$ such that for any Kakeya set $K \subseteq \Z_p^n$ the bound 
\begin{equation*}
|\mathcal{N}_{p^{-\alpha}}(K)| \geq c_{\varepsilon,n} p^{-\varepsilon\alpha}
\end{equation*}
holds for all $\alpha \in \N$. 
\end{conjecture}
Here for any set $E \subseteq \Q_p^n$ and $\alpha \in \Z$ the $p^{-\alpha}$-neighbourhood of $E$ is defined to be the set
\begin{equation}\label{neighbourhood definition}
\mathcal{N}_{p^{-\alpha}}(E) := \bigcup_{y \in E} B_{p^{-\alpha}}(y). 
\end{equation}

It transpires that Conjecture \ref{p adic Kakeya set conjecture} is equivalent to the weakened version of Conjecture \ref{Kakeya conjecture} where $N$ varies only over powers of the fixed prime $p$. This equivalence is discussed below, and provides a simple instance of the correspondence principle.




\subsection{Equivalence of the Kakeya problem over \texorpdfstring{$\Q_p$}{Qp} and \texorpdfstring{$\Z/p^{\alpha}\Z$}{ZpZ}} The proof of the equivalence between the $\Q_p$ and $\Z/p^{\alpha}\Z$ formulations of the Kakeya set conjecture relies on two ingredients, described presently.

\subsubsection*{1. Correspondence for sets} Given a set $E \subseteq \Z_p^n$ one may easily deduce from \eqref{projection equivalence} and \eqref{neighbourhood definition} that 
\begin{equation}\label{set correspondence}
\mathcal{N}_{p^{-\alpha}}(E) = \bigcup_{\vec{x} \in \pi_{\alpha}(E)} \pi_{\alpha}^{-1}\{\vec{x}\},
\end{equation}
where the union is, of course, disjoint. The identity \eqref{projection equivalence} also implies that each of the sets $\pi_{\alpha}^{-1}\{\vec{x}\}$ has Haar measure $p^{-\alpha n}$ and, consequently,
\begin{equation*}
|\mathcal{N}_{p^{-\alpha}}(E)| = \sum_{\vec{x} \in \pi_{\alpha}(E)} |\pi_{\alpha}^{-1}\{\vec{x}\}| = \frac{|\pi_{\alpha}(E)|}{p^{\alpha n}}.
\end{equation*}
Note that the expression on the far left-hand side of this chain of equalities involves the Haar measure on $\Q_p^n$, whilst the expression on the right involves normalised counting measure. 

\subsubsection*{2. Correspondence for directions} The projective space $\mathbb{P}^{n-1}(\Z_p)$ is naturally related to the discrete projective spaces $\mathbb{P}^{n-1}(\Z/p^{\alpha}\Z)$. Indeed, as discussed in \cite{Caruso}, the inverse system on the family of groups $(\Z/p^{\alpha}\Z)_{\alpha \in \N}$ naturally induces an inverse system on the family of sets $(\mathbb{P}^{n-1}(\Z/p^{\alpha}\Z))_{\alpha \in \N}$ and the projective space $\mathbb{P}^{n-1}(\Z_p)$ can be realised as the inverse limit
\begin{equation*}
\mathbb{P}^{n-1}(\Z_p) = \varprojlim_{n \in \N} \mathbb{P}^{n-1}(\Z/p^{\alpha}\Z).
\end{equation*}
This gives rise to a family of natural projection mappings 
\begin{equation*}
\tilde{\pi}_{\alpha} \colon \mathbb{P}^{n-1}(\Z_p) \to \mathbb{P}^{n-1}(\Z/p^{\alpha}\Z).
\end{equation*}
The $\tilde{\pi}_{\alpha}$ may also be defined in terms of the $p$-adic expansion for (class representatives of) the $\omega \in \mathbb{P}^{n-1}(\Z_p)$; the details are left to the reader. In particular, using this observation one may show that $K \subseteq \Z_p^n$ is Kakeya if and only if $\pi_{\alpha}(K) \subseteq [\Z/p^{\alpha}\Z]^n$ is Kakeya for all $\alpha \in \N$. 

\subsubsection*{}Combining these correspondences, it is clear that the two formulations of the Kakeya problem are completely equivalent. Indeed, given any Kakeya set $K \subseteq \Z_p^n$, it follows that $K_{\alpha} := \pi_{\alpha}(K) \subseteq [Z/p^{\alpha}\Z]^n$ is Kakeya for all $\alpha \in \N$. Assuming Conjecture \ref{Kakeya conjecture} for $N = p^{\alpha}$, given $\varepsilon > 0$ there exists some $c_{\varepsilon, n} > 0$ such that
\begin{equation*}
|\mathcal{N}_{p^{-\alpha}}(K)| = \frac{|K_{\alpha}|}{p^{\alpha n}} \geq c_{\varepsilon, n}p^{-\varepsilon\alpha } \qquad \textrm{for all $\alpha \in \N$},
\end{equation*}
as required. Conversely, if $K_{\alpha} \subseteq [\Z/p^{\alpha}\Z]^n$ is Kakeya, then $K := \pi_{\alpha}^{-1}(K_{\alpha})$ is easily seen to be Kakeya with $\mathcal{N}_{p^{\alpha}}(K) = K$. Thus, assuming Conjecture \ref{p adic Kakeya set conjecture}, given $\varepsilon > 0$ there exists some $c_{\varepsilon, n} > 0$ such that
\begin{equation*}
\frac{|K_{\alpha}|}{p^{\alpha n}} = |\mathcal{N}_{p^{-\alpha}}(K)| \geq c_{\varepsilon, n}p^{-\varepsilon\alpha },
\end{equation*}
as required. 

The construction described in Proposition \ref{measure zero Kakeya} can be combined with the above observations to yield the following result. 

\begin{proposition} There exists a Kakeya set $K \subseteq \Z_p^n$ of Haar measure zero. 
\end{proposition}

\begin{proof} The sets constructed in Proposition \ref{measure zero Kakeya} can be lifted to produce Kakeya sets in $\Z_p^n$ of arbitrarily small measure. A measure zero set is then obtained by a limiting procedure. The details are omitted; see, for instance, \cite{Wolff2003} for a similar argument in the euclidean case. 
\end{proof}

It is remarked that such sets have been observed to exist in \cites{Caruso, Fraser2016} (see also \cite{Dummit2013}). The construction described here is closely related to that given in \cite{Fraser2016}, and stems from euclidean constructions described in \cites{Sawyer1987, Wisewell2005} (see also \cite{Wolff2003}).




\subsection{A correspondence principle for functions}

Developing a correspondence principle for the restriction problem is a little more involved than the Kakeya case. For restriction, one is required to lift \emph{functions} on $[\Z/p^{\alpha}\Z]^n$, rather than just sets, and the lifting procedure must behave well with respect to taking Fourier transforms. Moreover, one must also work over the entire vector space $\Q_p^n$, and not just the compact piece $\Z_p^n$, and this necessitates the use of a 2-parameter family of correspondences (the analysis of the previous subsection used just a 1-parameter family of maps, namely $\mathcal{N}_{p^{-\alpha}}(E) \mapsto \pi_{\alpha}(E)$). 

It is first remarked that, for each $\alpha \in \N_0$, the observations of the previous subsections imply a correspondence between functions $f \colon \Q_p^n \to \C$ that are supported in $\Z_p^n = B_1(0)$ and are constant on cosets of $p^{\alpha}\Z_p^n = B_{p^{-\alpha}}(0)$ and functions $F \colon [\Z/p^{\alpha}\Z]^n \to \C$. Indeed, given such a function $f$, one may simply define $F_{\alpha}[f] \colon [\Z/p^{\alpha}\Z]^n \to \C$ by
\begin{equation*}
F_{\alpha}[f](\vec{x}) := f(y) \quad \textrm{if $y \in \pi_{\alpha}^{-1}\{\vec{x}\,\}$}
\end{equation*}
for all $\vec{x} \in [\Z/p^{\alpha}\Z]^n$. This is well-defined by the hypotheses on $f$ and $f \mapsto F_{\alpha}[f]$ is an isomorphism between the relevant function spaces. 

For the purposes of restriction theory, it is useful to consider a 2-parameter family $F_{k.l}$ comprised of rescaled versions of the isomorphisms $F_{\alpha}$. For $k,l \in \N_0$ let $\mathscr{S}(\Q_p^n; k,l)$ denote the vector subspace of $L^1(\Q_p^n)$ consisting of all $f \colon \Q_p^n \to \C$ that are supported on $B_{p^{l}}(0)$ and constant on cosets of $B_{p^{-k}}(0)$. The union of the $\mathscr{S}(\Q_p^n; k,l)$ over all $k, l \in  \N_0$ is denoted $\mathscr{S}(\Q_p^n)$. Note that $f \in \mathscr{S}(\Q_p^n)$ if and only if it is a finite linear combination of characteristic functions of balls. 

Given $f \in \mathscr{S}(\Q_p^n; k,l)$, define $F_{k,l}[f] \colon [\Z/p^{k+l}\Z]^n \to \C$ by
\begin{equation}\label{T map}
F_{k,l}[f](\vec{x}) := p^{-kn}f(p^{-l}y) \quad \textrm{for $y \in \pi_{k+l}^{-1}\{\vec{x}\,\}$}
\end{equation}
for all $\vec{x} \in [\Z/p^{k+l}\Z]^n$. Once again, this is well-defined and $F_{k,l}[f] \colon \mathscr{S}(\Q_p^n; k,l) \to \ell^1([\Z/p^{k+l}\Z]^n)$ is an isomorphism. 

The space $\mathscr{S}(\Q_p^n)$ is the $p$-adic analogue of the Schwartz class $\mathscr{S}(\R^n)$ from Euclidean analysis. It is remarked that both spaces can be viewed as particular instances of a more general construction, namely the \emph{Schwartz--Bruhat class} on an arbitrary LCA group \cites{Bruhat1961, Osborne1975} (see also \cite{Taibleson1975}).\footnote{Strictly speaking, the Schwartz--Bruhat spaces are, by definition, \emph{topological} vector spaces and therefore their full definition requires a description of their topology. The topology is not discussed here as it plays no r\^ole in the forthcoming analysis.} It is a simple consequence of the Stone--Weierstrass theorem that $\mathscr{S}(\Q_p^n)$ is dense in $L^r(\Q_p^n)$ for $1 \leq r < \infty$. Furthermore, for all $k,l \in \N_0$ the Fourier transform restricts to a bijection from $\mathscr{S}(\Q_p^n; k,l)$ to $\mathscr{S}(\hat{\Q}_p^n; l, k)$.

It is useful to set up a similar correspondence between functions on the dual groups. In particular, given $g \in \mathscr{S}(\hat{\Q}_p^n; l, k)$ define  $\widehat{F}_{l,k}[g] \colon [\Z/p^{k+l}\Z]_*^n \to \C$ by 
\begin{equation}\label{T hat map}
\widehat{F}_{l,k}[g](\vec{\xi}\,) := g(p^{-k}\eta\,) \qquad \textrm{for $\eta \in \pi^{-1}_{k+l}\{\vec{\xi}\,\}$ }
\end{equation} 
for all $\vec{\xi} \in [\Z/p^{k+l}\Z]_*^n$. 

These definitions extend the correspondence for sets detailed above. Indeed, one may easily verify that \eqref{set correspondence} implies that
\begin{equation}\label{reformulated set correspondence}
\widehat{F}_{\alpha,0}[\chi_{\mathcal{N}_{p^{-\alpha}}(E)}] = \chi_{\pi_{\alpha}(E)}
\end{equation}
for all $E \subseteq \Z_p^n$ and $\alpha \in \N_0$.

The normalisation factors $p^{-kn}$ and 1 are chosen so that \eqref{T map} preserves the $L^1$-norm and \eqref{T hat map} preserves the $L^{\infty}$-norm; this is natural in view of the Riemann--Lebesgue lemma. More generally, the following norm identities hold. 

\begin{lemma}\label{norm identities} For $1 \leq r \leq \infty$ and any $f \in \mathscr{S}(\Q_p^n; k, l)$ and $g \in \mathscr{S}(\hat{\Q}_p^n; l,k)$ the following identities hold:\footnote{Recall that the $\ell^r$ norms on a finite abelian group $G$ and its dual ${\hat G}$ are defined with respect to counting and normalised counting measure, respectively.}
\begin{enumerate}[a)] \item $\|F_{k,l}[f]\|_{\ell^r([\Z/p^{k+l}\Z]^n)} = p^{-kn/r'}\|f\|_{L^r(\Q_p^n)}$;
\item $\|\widehat{F}_{l,k}[g]\|_{\ell^r([\Z/p^{k+l}\Z]_*^n)} = p^{ln/r}\|g\|_{L^r(\hat{\Q}_p^n)}$.
\end{enumerate}
\end{lemma}

\begin{proof} Since b) is essentially just a renormalised version of a), it suffices only to prove a). Fix $f \in \mathscr{S}(\Q_p^n; k, l)$ and observe that
\begin{equation*}
\|F_{k,l}[f]\|_{\ell^r([\Z/p^{k+l}\Z]^n)} = p^{-kn} \big(\sum_{\vec{x} \in [\Z/p^{k+l}\Z]^n}|f(p^{-l}y_{\vec{x}})|^r\big)^{1/r}
\end{equation*} 
where each $y_{\vec{x}} \in \Z_p^n$ is a fixed (but arbitrary) choice of element in $\pi_{k+l}^{-1}\{\vec{x}\,\}$. Since $f$ is constant on cosets of $p^k\Z_p^n$, it follows that the above expression may be written as
\begin{equation*}
p^{-knr/r'}\big(\sum_{\vec{x} \in [\Z/p^{k+l}\Z]^n)}\int_{B_{p^{-k}}(p^{-l}y_{\vec{x}})}|f(y)|^r\,\ud y\big)^{1/r}. 
\end{equation*}
One may easily observe that $\{p^{-l}y_{\vec{x}} \in \Q_p^n : \vec{x} \in [\Z/p^{k+l}\Z]^n\}$ forms a complete set of coset representatives of $p^k\Z_p^n$ in $p^{-l}\Z_p^n$ and a) immediately follows. 
\end{proof}

The mappings $F_{k,l}$ and $\widehat{F}_{l,k}$ behave well with respect to taking Fourier transforms.

\begin{lemma}\label{commutative lemma} The diagram

\begin{center}
\begin{tikzpicture}
  \matrix (m) [matrix of math nodes,row sep=3em,column sep=4em,minimum width=2em] {
     \mathscr{S}(\Q_p^n; k,l) & \mathscr{S}(\hat{\Q}_p^n; l,k) \\
     \ell^1([\Z/p^{k+l}\Z]^n) & \ell^{\infty}([\Z/p^{k+l}\Z]^n_*) \\};
  \path[-stealth]
    (m-1-1) edge node [left] {$F_{k,l}$} (m-2-1)
            edge node [below] {$\mathscr{F}$} (m-1-2)
    (m-2-1.east|-m-2-2) edge node [below] {$\mathscr{F}$} (m-2-2)
    (m-1-2) edge node [right] {$\widehat{F}_{l,k}$} (m-2-2);
		\end{tikzpicture} 
\end{center}
commutes, where each occurrence of $\mathscr{F}$ denotes the appropriate Fourier transform.
\end{lemma}

\begin{proof} Given $f \in \mathscr{S}(\Q_p^n; k,l)$ observe that
\begin{equation*}
(F_{k,l}[f])\,\hat{}\,(\vec{\xi}\,) = p^{-kn} \sum_{\vec{x} \in [\Z/p^{k+l}\Z]} f(p^{-l}y_{\vec{x}}) e^{- 2\pi i \vec{x} \cdot \vec{\xi}/p^{k+l}}
\end{equation*}
where, as in the proof of Lemma \ref{norm identities}, each $y_{\vec{x}} \in \Q_p^n$ is a fixed (but arbitrary) choice of element in $\pi_{k+l}^{-1}\{\vec{x}\,\}$. If $\eta \in \pi_{k+l}^{-1}\{\vec{\xi}\,\}$, then it follows that
\begin{equation*}
(F_{k,l}[f])\,\hat{}\,(\vec{\xi}\,)  = p^{-kn} \sum_{\vec{x} \in [\Z/p^{k+l}\Z]} f(p^{-l}y_{\vec{x}}) e(-p^{-l}y_{\vec{x}} \cdot p^{-k} \eta),
\end{equation*}
where $e \colon \Q_p \to \mathbb{T}$ is the additive character defined earlier. Since the function $y \mapsto f(y)e(- y \cdot p^{-k}\eta)$ is constant on cosets of $p^k\Z_p^n$, arguing as in the proof of Lemma \ref{norm identities}, one deduces that 
\begin{equation*}
(F_{k,l}[f])\,\hat{}\,(\vec{\xi}\,) = \int_{B_{p^l}(0)} f(y) e(- y\cdot p^{-k}\eta\,) \,\ud y.
\end{equation*}
Recalling the definition of $\widehat{F}_{l,k}[\hat{f}](\vec{\xi}\,)$, this concludes the proof.
\end{proof}




\subsection{Equivalence of restriction over \texorpdfstring{$\Q_p$}{Qp} and \texorpdfstring{$\Z/p^{\alpha}\Z$}{ZpZ}} Suppose $\Gamma$ is as in \eqref{graph parametrised surface} and, as above, define
\begin{equation*}
\Sigma := \{\Gamma(\omega) : \omega \in \Z_p^d\}.
\end{equation*}
Let $\mu$ denote the measure on $\Sigma$ given by the push-forward of the Haar measure on $\Z_p^d$ under $\Gamma$. Furthermore, for each $\alpha \in \N$ let $\Sigma_{\alpha} \subseteq [\Z/p^{\alpha}\Z]^n_*$ denote the image of the mapping \eqref{graph parametrised surface} as a function $ \Gamma \colon [\Z/p^{\alpha}\Z]^d \to [\Z/p^{\alpha}\Z]^n$. In the remainder of this section it is shown that $p$-adic restriction estimates for the surface $\Sigma$ are, in some strong sense, equivalent to discrete restriction estimates for the $\Sigma_{\alpha}$. 

\begin{proposition}\label{equivalence} With the above setup, the following are equivalent:
\begin{enumerate}[i)]
\item The $p$-adic restriction estimate
\begin{equation}\label{restriction}
\|\hat{f}|_{\Sigma}\|_{L^s(\mu)} \leq \bar{C} \|f\|_{L^r(\Q_p^n)}
\end{equation}
holds for all $f \in \mathscr{S}(\Q_p^n)$.
\item For all $\alpha \in \N_0$ the discrete estimate
\begin{equation*}
\big( \frac{1}{|\Sigma_{\alpha}|} \sum_{\vec{\xi} \in \Sigma_{\alpha}} |\hat{F}(\vec{\xi}\,)|^s\big)^{1/s} \leq \bar{C} \|F\|_{\ell^r([\Z/p^{\alpha} \Z]^n)}
\end{equation*}
holds for all $F \colon [\Z/p^{\alpha} \Z]^n \to \C$.
\end{enumerate}
The constants appearing in both inequalities are identical. 
\end{proposition}

The hypothesis that the surface is graph parametrised is essentially for convenience and could be weakened. The important property is that each $\Sigma_{\alpha}$ is a `$d$-dimensional object', in the sense that $|\Sigma_{\alpha}| = p^{d\alpha}$.

Proposition \ref{equivalence} is an immediate consequence of the following lemma. 

\begin{lemma}\label{reformulation} For all $l \in \N_0$ the following are equivalent:
\begin{enumerate}[i)]
\item The estimate
\begin{equation}\label{local restriction}
\|\hat{f}|_{\Sigma}\|_{L^s(\mu)} \leq \bar{C} \|f\|_{L^r(B_{p^{l}}(0))}
\end{equation}
holds for all $f \in \mathscr{S}(\Q_p^n)$ supported in $B_{p^{l}}(0)$.
\item The estimate\footnote{For any Borel set $E \subseteq \Q_p^n$ of positive Haar measure the norm $\|\,\cdot\,\|_{L^{s}_{\mathrm{avg}}(E)}$ is defined by
\begin{equation*}
\|f\|_{L^{s}_{\mathrm{avg}}(E)} := \Big(\fint_E |f(x)|^s\,\ud x\Big)^{1/s} := \Big(\frac{1}{|E|}\int_E |f(x)|^s\,\ud x\Big)^{1/s}.
\end{equation*}
Note that $\fint_E$ is used to denote the normalised integral over $E$.}
 \begin{equation}\label{restriction reformulation}
\|\hat{f}\|_{L^{s}_{\mathrm{avg}}(\mathcal{N}_{p^{-l}}(\Sigma))} \leq \bar{C} \|f\|_{L^{r}(B_{p^{l}}(0))}
\end{equation}
holds for all $f \in \mathscr{S}(\Q_p^n)$ supported in $B_{p^{l}}(0)$.
\item The estimate
\begin{equation}\label{quotient restriction lemma}
\big( \frac{1}{|\Sigma_{l}|} \sum_{\vec{\xi} \in \Sigma_{l}} |\hat{F}(\vec{\xi}\,)|^s\big)^{1/s} \leq \bar{C} \|F\|_{\ell^r([\Z/p^{l} \Z]^n)}
\end{equation}
holds for all $F \colon [\Z/p^{l} \Z]^n \to \C$.
\end{enumerate}
The constants appearing in all three inequalities are identical. 
\end{lemma}

Assuming Lemma \ref{reformulation}, Proposition \ref{equivalence} immediately follows.

\begin{proof}[Proof (of Proposition \ref{equivalence})] The (global) restriction estimate \eqref{restriction} is equivalent to the (local) estimates \eqref{local restriction} holding at each scale $l \in \N_0$ with the uniform choice of constant $\bar{C}$. The desired result now follows from Lemma \ref{reformulation}. 
\end{proof}

The equivalence i) $\Leftrightarrow$ ii) is the $p$-adic version of a well-known fact in the euclidean case (see, for instance, \cite{Tao2004}). The simple proof is postponed until the end of the section. On the other hand, the equivalence ii) $\Leftrightarrow$ iii) follows from the correspondence between $\mathscr{S}(\Q_p^n; k, l)$-functions and functions on the module $[\Z/p^{k+l}\Z]^n$, as developed in the previous subsection. 

\begin{proof}[Proof (of Lemma \ref{reformulation}, ii) $\Leftrightarrow$ iii))] It follows from \eqref{reformulated set correspondence} that $\chi_{\mathcal{N}_{p^{-{l}}}(\Sigma)} = \widehat{F}^{-1}_{l,0} [\chi_{\pi_{l}(\Sigma)}]$. More generally, the same argument yields
\begin{equation}\label{neighbourhood identity}
\chi_{\mathcal{N}_{p^{-l}}(\Sigma)} = \widehat{F}^{-1}_{l,k} [\chi_{\pi_{k+l}(p^k \Sigma)}] \qquad \textrm{for all $k \in \N_0$.}
\end{equation}

Fix $f \in \mathscr{S}(\Q_p^n; k,l)$ for some $k \in \N_0$ and observe \eqref{neighbourhood identity} together with Lemma \ref{commutative lemma} imply that
\begin{equation*}
\chi_{\mathcal{N}_{p^{-l}}(\Sigma)}\hat{f} = \widehat{F}^{-1}_{l,k} \big[\chi_{\pi_{k+l}(p^k \Sigma)}\widehat{F}_{l,k}[\hat{f}]\,\big] = \widehat{F}^{-1}_{l,k} \big[\chi_{\pi_{k+l}(p^k \Sigma)}(F_{k,l}[f])\;\widehat{}\;\big].
\end{equation*}
From this and Lemma \ref{norm identities} b) one deduces that
\begin{equation*}
\|\hat{f}\|_{L^{s}_{\mathrm{avg}}(\mathcal{N}_{p^{-l}}(\Sigma))}  =  \Big(\frac{1}{|\pi_{k+l}(p^k \Sigma)|}\sum_{\vec{\xi} \in\pi_{k+l}(p^k \Sigma)}|(F_{k,l}[f])\;\widehat{}\;(\vec{\xi}\,)|^s\Big)^{1/s}.
\end{equation*}
Combining these observations together with Lemma \ref{norm identities} a), it follows that the local restriction estimate \eqref{restriction reformulation} holds for all $f \in \mathscr{S}(\Q_p^n)$ supported in $B_{p^{l}}(0)$ if and only if for every $k \in \N_0$ the estimate
\begin{equation}\label{rescaled quotient restriction}
\Big(\frac{1}{|\pi_{k+l}(p^k \Sigma)|}\sum_{\vec{\xi} \in \pi_{k+l}(p^k \Sigma)}|
\hat{F}(\vec{\xi}\,)|^s\Big)^{1/s} \leq \bar{C} q^{kn/r'}\|F\|_{\ell^r([\Z/p^{k+l}\Z]^n)}
\end{equation}
holds for all $F \colon [\Z/p^{k+l}\Z]^n \to \C$. The $k=0$ case of the above inequality is precisely \eqref{quotient restriction lemma}. It therefore suffices to show that \eqref{quotient restriction lemma} implies \eqref{rescaled quotient restriction} holds for all $k \in \N_0$. Given $F \colon [\Z/p^{k+l}\Z]^n \to \C$ define the function $F_l \colon [\Z/p^{l}\Z]^n \to \C$ by
\begin{equation*}
F_l(\vec{z}\,) := \sum_{\substack{\vec{x} \in [\Z/p^{k+l}\Z]^n \\ \vec{x} \equiv \vec{z} \,\mathrm{mod}\, p^l}} F(\vec{x}\,) \qquad \textrm{for all $\vec{z} \in [\Z/p^{l}\Z]^n$}.
\end{equation*}
Consider the map $\delta_k \colon [\Z/p^{l}\Z]^n \to [\Z/p^{k+l}\Z]^n$ given by $\delta_k \vec{\xi} := [p^k\xi]$ where $\xi \in \Z^n$ is a choice of class representative for $\vec{\xi}$. One may readily check that this mapping is well-defined and restricts to a bijection from $\pi_{l}(\Sigma)$ to $\pi_{k+l}(p^k\Sigma)$. Furthermore, for any $\vec{\xi} \in [\Z/p^{l}\Z]^n$ it follows that 
\begin{align*}
\hat{F}(\delta_k\vec{\xi}\,) = \sum_{\vec{z} \in  [\Z/p^{l}\Z]^n} \sum_{\substack{\vec{x} \in [\Z/p^{k+l}\Z]^n \\ \vec{x} \equiv \vec{z} \,\mathrm{mod}\, p^l}} F(\vec{x}\,)e^{2 \pi i \vec{x} \cdot \vec{\xi}/p^l} = \hat{F}_l(\vec{\xi}\,)
\end{align*}
and, consequently, 
\begin{equation}\label{rescaled quotient restriction 1}
\frac{1}{|\pi_{k+l}(p^k \Sigma)|}\sum_{\vec{\xi} \in\pi_{k+l}(p^k \Sigma)} |\hat{F}(\vec{\xi}\,)|^s = \frac{1}{|\pi_l(\Sigma)|}\sum_{\vec{\xi} \in \pi_l(\Sigma)} |\hat{F}_l(\vec{\xi}\,)|^s.
\end{equation}
On the other hand, H\"older's inequality implies that 
\begin{equation}\label{rescaled quotient restriction 2}
\|F_l\|_{\ell^r([\Z/p^{l}\Z]^n)} \leq p^{kn/r'}\|F\|_{\ell^r([\Z/p^{k+l}\Z]^n)}.
\end{equation}
Thus, assuming \eqref{quotient restriction lemma} holds for the function $F_l$ and combining this estimate with \eqref{rescaled quotient restriction 1} and \eqref{rescaled quotient restriction 2}, one concludes that \eqref{rescaled quotient restriction} holds for the function $F$, as required.
\end{proof}

\begin{proof}[Proof (of Lemma \ref{reformulation}, i) $\Leftrightarrow$ ii))] The proof relies on the identity
\begin{equation}\label{reformulation 1}
 \fint_{\mathcal{N}_{p^{-l}}(\Sigma)}  |\hat{f}(\eta)|^s \, \ud \eta = \int_{\Sigma} \fint_{B_{p^{-l}}(0)} |\hat{f}(\xi + \eta)|^s \, \ud \eta\ud \mu(\xi), 
\end{equation}
valid for all $f \in \mathscr{S}(\Q_p^n)$. To prove \eqref{reformulation 1}, first observe that
\begin{equation*}
\mathcal{N}_{p^{-l}}(\Sigma) = \bigcup_{\vec{z} \in [\Z/p^{l}\Z]_*^d} \pi_l^{-1}\{\Gamma(\vec{z})\}.
\end{equation*}
This implies that $|\mathcal{N}_{p^{-l}}(\Sigma)| = p^{-l(n-d)}$ and
\begin{align*}
 \fint_{\mathcal{N}_{p^{-l}}(\Sigma)}  |\hat{f}(\eta)|^s \, \ud \eta &= p^{-ld} \sum_{\vec{z} \in  [\Z/p^{l}\Z]_*^d} \fint_{\pi_l^{-1}\{\Gamma(\vec{z})\}} |\hat{f}(\eta)|^s \, \ud \eta \\
&= p^{-ld} \sum_{\vec{z} \in  [\Z/p^{l}\Z]_*^d} \fint_{B_{p^{-l}}(\Gamma(\omega_{\vec{z}}))} |\hat{f}(\eta)|^s \, \ud \eta,
\end{align*}
where $\omega_{\vec{z}} \in \Z_p^n$ is an arbitrary choice of element of $\pi_l^{-1}\{\vec{z}\}$ for each $\vec{z} \in [\Z/p^l\Z]_*^d$. Averaging over all possible choices of the $\omega_{\vec{z}}$ one concludes that
\begin{equation*}
 \fint_{\mathcal{N}_{p^{-l}}(\Sigma)}  |\hat{f}(\eta)|^s \, \ud \eta = \sum_{\vec{z} \in  [\Z/p^{l}\Z]_*^d} \int_{\pi_l^{-1}\{\vec{z}\}} \fint_{B_{p^{-l}}(0)} |\hat{f}(\Gamma(\omega)+\eta)|^s \, \ud \eta \ud \omega,
\end{equation*} 
and \eqref{reformulation 1} immediately follows. 

Suppose that \eqref{local restriction} holds for all $f \in \mathscr{S}(\Q_p^n)$ with $\mathrm{supp}\,f \subseteq B_{p^l}(0)$. It follows from the modulation invariance of the $L^r$-norm that 
\begin{equation*}
\Big(\int_{\Sigma} \fint_{B_{p^{-l}}(0)} |\hat{f}(\xi + \eta)|^s \, \ud \eta\ud \mu(\xi) \Big)^{1/s} \leq \bar{C}  \|f\|_{L^r(B_{p^{l}}(0))}
\end{equation*}
for any such function $f$, and the identity \eqref{reformulation 1} immediately yields \eqref{restriction reformulation}. 

Conversely, suppose \eqref{restriction reformulation} holds for all $f \in \mathscr{S}(\Q_p^n)$ with $\mathrm{supp}\,f \subseteq B_{p^l}(0)$. If $f$ is of this type, then $f = f\chi_{B_{p^l}(0)}$, which leads to the reproducing formula 
\begin{equation*}
\hat{f}(\xi) = \hat{f} \ast \hat{\chi}_{B_{p^l}(0)}(\xi) = \fint_{B_{p^{-l}}(0)} \hat{f}(\xi + \eta)\,\ud \eta. 
\end{equation*}
Thus, by H\"older's inequality,
\begin{equation*}
\|\hat{f}|_{\Sigma}\|_{L^s(\mu)} \leq \Big(\int_{\Sigma} \fint_{B_{p^{-l}}(0)} |\hat{f}(\xi + \eta)|^s \, \ud \eta\ud \mu(\xi) \Big)^{1/s} 
\end{equation*}
and \eqref{local restriction} now follows from \eqref{reformulation 1}.
\end{proof}




\subsection{Restriction and Kakeya over local fields} The analysis of this section can be generalised to the setting of non-archimedean local fields. For brevity, the relevant definitions are not reviewed here; the interested reader may consult, for instance, \cite{Lang1984}, \cite{Lang1994} or \cite{Taibleson1975} for further information. Let $K$ be a field with a discrete non-archimedean absolute value $|\,\cdot\,|_K$, suppose $\pi \in K$ is a choice of uniformiser and let $\mathfrak{o}:= \{x \in K : |x|_K \leq 1\}$ denote the ring of integers of $K$. Assume that the residue class field $\mathfrak{o}/\pi \mathfrak{o}$ is finite. One may easily formulate versions of the restriction and Kakeya problems over the field $K$ or the quotient rings $\mathfrak{o}/\pi^{\alpha} \mathfrak{o}$. The resulting theories are then essentially equivalent via a correspondence principle which extends that described above. The details can be found in \cite{Hickman2015}. 

It is well-known that any field $K$ satisfying the above properties is isomorphic to either a finite extension of $\Q_p$ for some prime $p$ or the field $\mathbb{F}_q((X))$ of formal Laurent series over a finite field $\mathbb{F}_q$. The local fields $\mathbb{F}_q((X))$ are particularly well-behaved spaces which act as simplified models of Euclidean space. For instance, Fourier analysis over $\mathbb{F}_2((X))$ corresponds to the study of Fourier--Walsh series, which has played a prominent r\^ole as a model for problems related to Carleson's theorem and time-frequency analysis \cites{DiPlinio2014, Do2012}. Recently there has been increased interest in local field variants of other problems in Euclidean harmonic analysis and geometric measure theory, focusing on the Kakeya conjecture \cites{Ellenberg2010, Caruso, Dummit2013, Fraser2016}. In particular, in \cite{Ellenberg2010} it is shown that Dvir's \cite{Dvir2009} finite field Kakeya theorem can be used to prove strengthened bounds on the size of Kakeya sets over $\mathbb{F}_q((X))$. This simple observation stems from the fact that each quotient ring of  $\mathbb{F}_q((X))$ is a vector space over a finite field. It would be interesting to see whether it is possible to extend this result to the $p$-adic setting.




\section{\texorpdfstring{$\ell^2$}{l2} restriction in \texorpdfstring{$\Z/N\Z$}{ZNZ}}\label{L2 restriction section}




\subsection{An abstract restriction theorem} In this section a fairly abstract $\ell^2$ Fourier restriction estimate is established for general sets $\Sigma$ lying in $[\Z/N\Z]^n$, under certain dimensionality hypotheses. This result is then used to study various prototypical cases such as the paraboloid. In order to state the general form of the restriction theorem, it is necessary to revisit the scaling structure on $\Z/N\Z$ described earlier in the article. 

Recall the collection of balls $\{\mathcal{B}_d\}_{d \mid N}$ introduced in $\S$\ref{basic setup section}, given by 
\begin{equation*}
\mathcal{B}_d := \{\vec{x} \in [\Z/N\Z]^n : \|\vec{x}\| \preceq d \}.
\end{equation*}
It was noted in $\S$\ref{basic setup section} that when $N=p^{\alpha}$ is a power of a fixed prime $p$ these balls form a nested sequence. For general $N$ this property does not hold and it is therefore useful to consider the 1-parameter nested family of balls
\begin{equation*}
B_{\rho} (\vec{0}\,) \ := \ \bigcup_{d \mid N; d \le \rho} {\mathcal B}_d \qquad \textrm{for all $0 < \rho$.}
\end{equation*}
Note that the above union is taken over all divisors $d$ which are at most $\rho$ in the usual sense, whereas the inequality $\|\vec{x}\| \preceq d$ defining ${\mathcal B}_d$ is with respect to the division ordering. When $N = p^{\alpha}$ the sets ${\mathcal B}_d$ are already nested and $B_{\rho}(\vec{0}\,) = \mathcal{B}_{p^{\nu}}$ where $0 \leq \nu \leq \alpha$ is the largest value for which $p^{\nu} \leq \rho$ holds. The set system $B_{\rho}(\vec{0}\,)$ is extended to a translation invariant family on $[\Z/N\Z]^n$ by setting $B_{\rho} (\vec{x}\,) := \vec{x} + B_{\rho}(\vec{0}\,)$ for all $\vec{x} \in [\Z/N\Z]^n$ and $\rho > 0$.  

The term `balls' is used loosely here: the $B_{\rho}(\vec{x}\,)$ do not arise from a metric, or even a pseudo-metric. They do, however, satisfy the following properties:
\begin{enumerate}[i)]
\item Nesting: $B_{\rho}(\vec{0}\,) \subseteq B_{\rho'}(\vec{0}\,)$ for all $0 < \rho \leq \rho'$;
\item Symmetry: $B_{\rho}(\vec{0}\,) = -B_{\rho}(\vec{0}\,)$ for all $0 < \rho$;
\item Covering: $\bigcup_{\rho >0} B_{\rho}(\vec{0}\,) = [\Z/N\Z]^n$;
\item Translation invariance: $B_{\rho}(\vec{x}\,) = \vec{x} + B_{\rho}(\vec{0}\,)$ for all $\vec{x} \in [\Z/N\Z]^n$ and $0< \rho$.
\end{enumerate}
In addition, the balls satisfy a natural regularity condition with respect to the Haar (that is, counting) measure on $[\Z/N\Z]^n$. In particular, for all $\varepsilon > 0$ one has
\begin{flalign*}
&\textrm{(R)}\quad |B_{\rho}(\vec{0}\,)| \leq C_{\varepsilon} N^{\varepsilon}\rho^n \qquad \textrm{for all $0 < \rho$.}&
\end{flalign*}
Indeed, 
\begin{equation*}
|B_{\rho}(\vec{0}\,)| \ \le  \ \sum_{d \mid N; d \le \rho} |{\mathcal B}_d| \ = \
\sum_{d \mid N; d \le \rho} d^n \ \le \ \bigl[ \, \sum_{d \mid N} 1 \, \bigr] \ \rho^n
\end{equation*}
and (R) now follows from standard bounds for the divisor function. It is easy to see that when $N = p^{\alpha}$ the property (R) holds with a uniform constant (that is, without any $\varepsilon$-loss in $N$). 

The dual group $[\Z/N\Z]_*^n$ is also endowed with a family of balls $\widehat{B}_{\rho}(\vec{\xi}\,)$, which are naturally dual to the $B_{\rho}(\vec{x}\,)$. In particular, define
\begin{equation*}
\widehat{B}_{\rho}(\vec{0}\,) \ := \ \bigcup_{d \mid N; d\ge 1/\rho} \, {\mathcal B}_{N/d} \qquad \textrm{for all $0 < \rho$}
\end{equation*}
and let $\widehat{B}_{\rho}(\vec{\xi}\,) := \vec{\xi} + \widehat{B}_{\rho}(\vec{0}\,)$ for all $\vec{\xi} \in [\Z/N\Z]_*^n$ and $0 < \rho$. 

Having made these preliminary definitions, one may now state the abstract $\ell^2$ Fourier restriction theorem mentioned above. Fix $N \in \N$ and a set of frequencies $\Sigma \subseteq [\Z/N\Z]_*^n$. Mirroring the results in the Euclidean setting \cites{Bak2011, Mitsis2002, Mockenhaupt2000}, one assumes that the normalised counting measure on $\Sigma$ satisfies both a dimensional (or regularity) and Fourier-dimensional hypothesis; in particular, for some $0 < b \leq a < n$ assume that the following hold:
\begin{flalign*}
&\textrm{(R$\Sigma$)}\quad \frac{|\widehat{B}_{\rho}(\vec{\xi}\,) \cap \Sigma|}{|\Sigma|} \leq Ar^a \textrm{ for all $\vec{\xi} \in [\Z/N\Z]_*^n$;}&
\end{flalign*}
\begin{flalign*}
&\textrm{(F$\Sigma$)}\quad \big| \frac{1}{|\Sigma|} \sum_{\vec{\xi} \in \Sigma} e^{2\pi i \vec{x}\cdot\vec{\xi}/N} \big| \leq Br^{-b/2} \textrm{ for all $\vec{x} \notin B_r(\vec{0}\,)$.}&
\end{flalign*}
If $\mu$ denotes the normalised counting measure on $\Sigma$, then the above inequalities can be rewritten as:
\begin{flalign*}
&\textrm{(R$\mu$)}\quad \mu(\widehat{B}_r(\vec{\xi}\,)) \leq Ar^a \textrm{ for all $\vec{\xi} \in [\Z/N\Z]_*^n$;}&
\end{flalign*}
\begin{flalign*}
&\textrm{(F$\mu$)}\quad |\check{\mu}(\vec{x}\,)| \leq Br^{-b/2} \textrm{ for all $\vec{x} \notin B_r(\vec{0}\,)$.}&
\end{flalign*}
These conditions are therefore natural discrete analogues of those featured in \cites{Bak2011, Mitsis2002, Mockenhaupt2000}. 

With the various definitions now in place, the main theorem is as follows.

\begin{theorem}\label{L2-restriction} Fix $N \in \N$ and a set of frequencies $\Sigma \subseteq [\Z/N\Z]_*^n$ and suppose $\Sigma$ satisfies (R$\,\Sigma$) and (F$\,\Sigma$) for some $0 < b \leq a < n$ and $0 < A, B$. Then for all $\varepsilon >0$, there is a constant $C_{\varepsilon}$ such that the inequality
\begin{equation}\label{abstract Stein Tomas}
\Big(\frac{1}{|\Sigma|}\sum_{\vec{\xi}\in \Sigma}|\hat{F}(\vec{\xi}\,)|^2 \Big)^{1/2} \leq  C_{\varepsilon} A^{b/(4(n-a) + 2b)}B^{(n-a)/(2(n-a)+b)} N^{\varepsilon}\|F\|_{L^r([\Z/N\Z]^n)}
\end{equation}
holds for all $1 \leq r \leq r_0$ where 
\begin{equation}\label{exponent formula}
r_0 := \frac{4(n - a) + 2b}{4(n-a) + b}.
\end{equation}
\end{theorem}

Theorem \ref{L2-restriction} is, in fact, a special case of a more general result concerning $L^2$ Fourier restriction on locally compact abelian (LCA) groups. In particular, in \cite{Hickman3} it is observed that an argument of Bak and Seeger \cite{Bak2011} can be extended to a class of LCA groups which admit a primitive form of Littlewood--Paley theory. Unfortunately, the full details of the hypotheses of the main result in \cite{Hickman3} are somewhat involved and are therefore not reproduced here. In order to apply the result of \cite{Hickman3} in the current context, one considers a system of Littlewood--Paley projections defined with respect to the balls $B_{\rho}(\vec{x}\,)$ and $\widehat{B}_{\rho}(\vec{\xi}\,)$ introduced above. For each $\rho > 0$ let $\varphi_{\rho} := \chi_{B_{\rho}(\vec{0}\,)}$ denote the characteristic function of the ball $B_{\rho}(\vec{0}\,)$. For the purpose of the argument, one wishes to show that the projection operators $G \mapsto G \ast \hat{\varphi}_{\rho}$ (defined on the class of functions on the dual group $[\Z/N\Z]_*^n$) are well-behaved. Since $\|\varphi_{\rho}\|_{\ell^{\infty}([\Z/N\Z]^n)} \leq 1$, one immediately deduces the $\ell^2$-bound
\begin{equation*}
\|G \ast \hat{\varphi}_{\rho}\|_{\ell^2([\Z/N\Z]_*^n)} \leq \|G\|_{\ell^2([\Z/N\Z]_*^n)}
\end{equation*}
by Plancherel's theorem. On the other hand, favourable $\ell^1$-type bounds follow from pointwise estimates for the Fourier transform $\hat{\varphi}_{\rho}$.

\begin{proposition}\label{Littlewood-Paley proposition} For $\varphi_{\rho}$ as defined above, for all $\varepsilon >0$ there exists a constant $C_{\varepsilon} > 0$ such that the following condition holds:
\begin{flalign*}
&\textrm{(F)}\quad |\hat{\varphi}_{\rho}(\vec{\xi}\,)| \leq C_{\varepsilon}N^{\varepsilon}s^{-n} \textrm{ whenever $-\vec{\xi} \notin \widehat{B}_s(\vec{0}\,)$ and $s \geq 1/\rho$.}&
\end{flalign*}
\end{proposition}

The main theorem of \cite{Hickman3} reduces the proof of Theorem \ref{L2-restriction} to establishing the condition (F).\interfootnotelinepenalty=10000\footnote{The hypotheses of the main theorem in \cite{Hickman3} also require a uniform $\ell^1$ bound for the functions $\hat{\varphi}_{\rho}$, which in the current context is the property that for all $\varepsilon >0$ there exists a constant $C_{\varepsilon} >0$ such that
\begin{flalign*}
&\textrm{(F$'$)}\quad \frac{1}{N^n}\sum_{\vec{\xi} \in [\Z/N\Z]_*^n}|\hat{\varphi}_{\rho}(\vec{\xi}\,)| \leq C_{\varepsilon}N^{\varepsilon} \textrm{ for all $0 < \rho$.}&
\end{flalign*}
However, in \cite{Hickman3}*{Lemma 3.2} it is shown that \eqref{abstract Stein Tomas} holds in the non-endpoint range $1 \leq r < r_0$ without the hypothesis (F$'$). Since here an $\varepsilon$-loss in $N$ is permitted in the constant, the non-endpoint range and endpoint range of inequalities are equivalent via H\"older's inequality.

If one wishes to apply the results of \cite{Hickman3} to study Problem \ref{p power problem}, then an $\varepsilon$-loss in the cardinality of the ring is no longer acceptable and condition (F$'$) must now also be verified (with a uniform constant appearing on the right-hand side). However, in this situation the functions $\hat{\varphi}_{\rho}$ admit a clean, explicit formula and the computations are substantially simpler: see \cite{Hickman3}*{$\S$2}. \nopagebreak} The proof of Proposition \ref{Littlewood-Paley proposition}, which is slightly involved, is given in the following subsection. Some consequences of Theorem \ref{L2-restriction} are then discussed. 




\subsection{The proof of Proposition \ref{Littlewood-Paley proposition}}

The proof of the proposition will make repeated use of the following elementary observation.

\begin{lemma}\label{gcd lemma}  For $p$ be prime and $m, L \in \N$ define
\begin{equation}\label{gcd}
I(m, p^L) := \#\bigl\{(x_1, \ldots, x_m) \in [{\mathbb Z}/p^L{\mathbb Z}]^m : \gcd(x_1, \ldots, x_m, p) = 1 \bigr\}.
\end{equation}
Then $I(m, p^L) =  p^{Lm} (1 - p^{-m})$.
\end{lemma}

\begin{proof} The case $m=1$ is readily verified. Let $m \geq 2$ and suppose, by way of induction hypothesis, that $I(m-1,p^L) =\ p^{L(m-1)} (1 - p^{-(m-1)})$. Clearly $I(m, p^L)$ can be expressed as the sum of
\begin{equation*}
\#\bigl\{(x_1, \ldots, x_m) \in  [{\mathbb Z}/p^L{\mathbb Z}]^{m} : p \nmid x_m \bigr\} = p^{L(m-1)}\cdot p^{L}(1 - p^{-1})
\end{equation*}
and
\begin{equation*}
 \#\bigl\{(x_1, \ldots, x_m) \in  [{\mathbb Z}/p^L{\mathbb Z}]^{m} : (x_1, \ldots, x_{m-1}, p) = 1; \, p \mid x_m \bigr\} = I(m-1,p^L)\cdot p^{L-1}.
\end{equation*}
Applying the induction hypothesis, it then follows that
\begin{equation*}
I(m, p^L) = p^{L(m-1)}\cdot p^{L}(1 - p^{-1}) + p^{L(m-1)} (1 - p^{-(m-1)}) \cdot p^{L-1} = p^L(1-p^{-m}),
\end{equation*}
which closes the induction and completes the proof. 
\end{proof}

\begin{proof}[Proof (of Proposition \ref{Littlewood-Paley proposition})] Given $\varepsilon > 0$, the problem is to show that there exists a constant $C_{\varepsilon} > 0$ such that
\begin{equation*}
\Bigl| \sum_{\vec{x} \in B_{\rho}(\vec{0}\,) } e^{2 \pi i \vec{\xi} \cdot \vec{x} / N} \Bigr|
\ \le \ C_{\varepsilon}N^{\varepsilon} \, s^{-n} \ \ \ \textrm{for all}  \\ -\vec{\xi} \notin \widehat{B}_s(\vec{0}\,), \ \ \textrm{whenever} \ s \geq 1/\rho.
\end{equation*}
Recalling the definition of $B_{\rho}(\vec{0}\,)$, the left-hand sum may be written as
\begin{equation*}
\sum_{d|N; d\ge N/{\rho}} \ S_{N,d}(\vec{\xi}\,)
\end{equation*}
where
\begin{equation*}
S_{N,d}(\vec{\xi}\,) := \sum_{\substack{\vec{x} \in [\Z/N\Z]^n \\ \gcd(x_1,\ldots, x_n, N) = d}} e^{2 \pi i \vec{\xi} \cdot \vec{x} / N}.
\end{equation*} 
By the divisor bound, it suffices to show that for a fixed divisor $d|N$ with $d\ge N/\rho$ one has
\begin{equation}\label{Phi-d}
| S_{N,d}(\vec{\xi}\,)| \le  C s^{-n} \quad \textrm{for all} \ \  -\vec{\xi} \notin \widehat{B}_s(\vec{0}\,), \ \ {\rm whenever} \ s\geq 1/\rho.
\end{equation}
The inequality \eqref{Phi-d} is trivial when $d = N$ and so one may assume that $d \prec N$ is a proper divisor of $N$.

By rescaling it follows that $S_{N, d}(\vec{\xi}\,) = S_{N/d, 1}(\vec{\xi}\,)$. Furthermore, it is not difficult to see that $M \mapsto S_{M,1}(\vec{\xi}\,)$ is a multiplicative function and so, writing $N/d = p_1^{L_1} \cdots p_r^{L_r}$ where $p_1, \dots, p_r$ are distinct primes, it follows that 
\begin{equation}\label{multiplicative}
S_{N,d}(\vec{\xi}\,) \ = \ \prod_{t=1}^r S_{p_t^{L_t},1}(\vec{\xi}\,).
\end{equation}

Let $\vec{\xi} = (\xi_1, \dots, \xi_n) \in [\Z/N\Z]_*^n$ and suppose that there exists some $1\le t \le r$ and $1\le k \le n$ such that $p_t^{L_t -1} \nmid \xi_k$. In this case it follows that $S_{N,d}(\vec{\xi}\,) = 0$. To see this, write $S_{p_t^{L_t},1}(\vec{\xi}\,) = \RN{1} + \RN{2}$ where
\begin{equation*}
\RN{1}  :=  \sum_{\substack{(x_1,\dots \widehat{x_k} \dots, x_n) \in [\Z/p_t^{L_t}\Z]^{n-1} \\ \gcd(x_1, \dots \widehat{x_k} \dots, x_n, p_t) = 1}} \prod_{\substack{1 \leq l \leq n \\ l \neq k}}e^{2\pi i x_l \xi_l /p_t^{L_t}}  \times 
\sum_{\substack{0 \leq x_k \leq p_t^{L_t}-1 \\ p_t \mid x_k}} e^{2\pi i x_k \xi_k / p_t^{L_t}}
\end{equation*}
and
\begin{equation*}
\RN{2} := \sum_{(x_1,\dots \widehat{x_k} \dots, x_n) \in [\Z/p_t^{L_t}\Z]^{n-1}} \prod_{\substack{1 \leq l \leq n \\ l \neq k}} e^{2\pi i x_l \xi_l /p_t^{L_t}} \times \sum_{\substack{0 \leq x_k \leq p_t^{L_t}-1 \\ p_t \nmid x_k}} e^{2\pi i x_k \xi_k/ p_t^{L_t}}.
\end{equation*}
Here the notation $\widehat{x_k}$ is used to denote omission. Since $p_t^{L_t -1} \nmid \xi_k$, it follows that
\begin{equation}\label{II}
\sum_{\substack{0 \leq x_k \leq p_t^{L_t}-1 \\ p_t \mid x_k}} e^{2\pi i x_k \xi_k / p_t^{L_t}} = \sum_{x_k=0}^{p_t^{L_t -1} -1} e^{2\pi i x_k\xi_k / p_t^{L_t -1}} = 0,
\end{equation}
implying that $\RN{1} = 0$. On the other hand, 
\begin{equation*}
\sum_{\substack{0 \leq x_k \leq p_t^{L_t}-1 \\ p_t \nmid x_k}} e^{2\pi i x_k \xi_k/ p_t^{L_t}} = \sum_{x_k = 0}^{p_t^{L_t} - 1} e^{2\pi i x_k \xi_k/ p_t^{L_t}} - 
\sum_{\substack{0 \leq x_k \leq p_t^{L_t}-1 \\ p_t \mid x_k}} e^{2\pi i x_k \xi_k / p_t^{L_t}}.
\end{equation*}
Since $p_t^{L_t} \nmid \xi_k$, the first sum on the right-hand side is 0, whilst the second sum is 0 by \eqref{II}. Thus, $\RN{2} = 0$, and so $S_{N,d}(\vec{\xi}\,) = 0$ in this case.

Next suppose that $p_t^{L_t -1} | \, \xi_k$ for all $1\le t \le r$ and all $1\le k \le n$. Split the prime factors of $N/d$ into two sets by defining
\begin{equation*}
A := \{1\le t \le r: p_t^{L_t} | \, \xi_k \ \textrm{ for } 1 \leq  k \leq n \} \quad \textrm{and} \quad B := \{1,\ldots, r\} \setminus A.
\end{equation*}
The hypotheses on $\vec{\xi}$ and the definition of $A$ now imply that $\vec{\xi}\in \mathcal{B}_{N/M}$ where
\begin{equation*}
M := \prod_{t\in A} p_t^{L_t} \prod_{t\in B} p_t^{L_t -1}.
\end{equation*}
On the other hand, if one assumes that $-\vec{\xi} \notin \widehat{B}_s(\vec{0})$, then, by definition, $\vec{\xi} \notin \mathcal{B}_{N/d}$ for all $d \ge 1/s$. Combining these observations, one deduces the inequality 
\begin{equation}\label{M bound}
M \le 1/s.
\end{equation}
It therefore suffices to show that
\begin{equation}\label{A index bound}
|S_{p_t^{L_t}, 1}(\vec{\xi}\,)| \leq p_t^{nL_t} \quad \textrm{for all $t \in A$}
\end{equation}
and
\begin{equation}\label{B index bound}
|S_{p_t^{L_t}, 1}(\vec{\xi}\,)| \leq p_t^{n(L_t-1)} \quad \textrm{for all $t \in B$.}
\end{equation}
Indeed, combining these estimates with \eqref{multiplicative} and \eqref{M bound} yields
\begin{equation*}
|S_{N,d}(\vec{\xi}\,)| \leq \prod_{t \in A} p_t^{nL_t}  \prod_{t \in B} p_t^{n(L_t-1)} =  M^n \leq s^{-n} \quad \textrm{for all $-\vec{\xi} \notin \widehat{B}_s(\vec{0}\,)$,}
\end{equation*}
as required. 

Observe that $S_{p_t^{L_t}, 1}(\vec{\xi}\,) = I(n, p_t^{L_t})$ for any $t \in A$, where $I(n, p_t^{L_t})$ is as defined in \eqref{gcd}. In this case, Lemma \ref{gcd lemma} implies that  
\begin{equation*}
S_{p_t^{L_t},1}(\vec{\xi}\,) = p_t^{n L_t}(1- p_t^{-n}) \le p_t^{n L_t},
\end{equation*}
which establishes \eqref{A index bound}.

It remains to verify \eqref{B index bound}. Fix $t\in B$ and assume, without loss of generality, that the components of $\vec{\xi} = (\xi_1,\ldots, \xi_n)$ are ordered so that there exists some $1\le k_0\le n$ satisfying $p_t^{L_t} \mid  \xi_1,\ldots, \xi_{n-k_0}$ and $p_t^{L_t} \nmid \xi_{n-k_0 +1}, \ldots, \xi_n$. Arguing by induction, it follows that for all $0 \leq k \leq k_0-1$ the identity
\begin{equation}\label{B index induction}
S_{p_t^{L_t},1}(\vec{\xi}\,) =  p_t^{k(L_t-1)} \sum_{\substack{(x_1, \dots, x_{n-k}) \in [\Z/p_t^{L_t}\Z]^{n-k} \\ \gcd(x_1,\ldots, x_{n-k}, p_t) = 1}} \prod_{l = 1}^{n-k} e^{2\pi i x_l\xi_l /p_t^{L_t}}
\end{equation}
holds. Indeed, when $k = 0$ this is vacuous. Assume that \eqref{B index induction} is valid for some $0 \leq k \leq k_0 -2$ and decompose the sum appearing on the right-hand side of \eqref{B index induction} into two terms $\RN{1}_{k+1} + \RN{2}_{k+1}$ where
\begin{equation*}
\RN{1}_{k+1} := \sum_{\substack{(x_1, \dots, x_{n-k-1}) \in [\Z/p_t^{L_t}\Z]^{n-k-1} \\ \gcd(x_1,\ldots, x_{n-k-1}, p_t) = 1}} \prod_{l = 1}^{n- k -1} e^{2\pi i x_l\xi_l /p_t^{L_t}} \times \sum_{\substack{0 \leq x_{n-k} \leq p_t^{L_t} -1  \\ p_t \mid x_{n-k} }} e^{2\pi i x_{n-k}\xi_{n-k} /p_t^{L_t}}
\end{equation*}
and
\begin{equation*}
\RN{2}_{k+1} := \sum_{(x_1, \dots, x_{n-k-1}) \in [\Z/p_t^{L_t}\Z]^{n-k-1}} \prod_{l = 1}^{n-k-1} e^{2\pi i x_l\xi_l /p_t^{L_t}} \times \sum_{\substack{0 \leq x_{n-k} \leq p_t^{L_t} - 1 \\ p_t \nmid x_{n-k} }} e^{2\pi i x_{n-k}\xi_{n-k} /p_t^{L_t}}.
\end{equation*}
Recall that, by hypothesis, $p_t^{L_t -1} \mid \xi_{n-k}$, from which one deduces that
\begin{equation*}
\RN{1}_{k+1} = p_t^{L_t -1}\sum_{\substack{(x_1, \dots, x_{n-k-1}) \in [\Z/p_t^{L_t}\Z]^{n-k-1} \\ \gcd(x_1,\ldots, x_{n-k-1}, p_t) = 1}} \prod_{l = 1}^{n- k -1} e^{2\pi i x_l\xi_l /p_t^{L_t}}.
\end{equation*}
On the other hand, since the choice of $k$ and definition of $k_0$ ensure that $p \nmid \xi_{n-k-1}$, it immediately follows that $\RN{2}_{k+1} = 0$. Combining these observations establishes the inductive step and, in particular, verifies \eqref{B index induction} for $k = k_0 -1$. 

Finally, repeat the preceding decomposition to arrive at the identity
\begin{equation*}
S_{p_t^{L_t},1}(\vec{\xi}\,) =  p_t^{(k_0 - 1)(L_t-1)}\big( \RN{1}_{k_0} + \RN{2}_{k_0} \big), 
\end{equation*}
where $\RN{1}_{k_0}$ and $\RN{2}_{k_0}$ are as defined above. Applying Lemma \ref{gcd lemma}, it is easy to verify that
\begin{equation*}
\RN{1}_{k_0} = p_t^{L_t -1} I(n-k_0, p_t^{L_t}) = p_t^{L_t -1}\cdot p_t^{L_t(n-k_0)}(1 - p^{-(n-k_0)}),
\end{equation*}
whilst 
\begin{equation*}
\RN{2}_{k_0} = - p_t^{L_t(n-k_0)} p_t^{L_t-1}.
\end{equation*}
Together these identities yield \eqref{B index bound}.

\end{proof}




\subsection{\texorpdfstring{$\ell^2$}{l2} restriction for the paraboloid: the proof of Theorem \ref{Tomas}} One is now in a position to employ the $\ell^2$ restriction theorem, Theorem \ref{L2-restriction}, whenever one has a set of frequencies $\Sigma$ in the dual group $[\Z/N\Z]_*^n$ satisfying the regularity condition (R$\Sigma$) and the Fourier decay estimate (F$\Sigma$) (or, equivalently, the normalised counting measure on $\Sigma$ satisfies (R$\mu$) and (F$\mu$)). For simplicity, first consider the prototypical case where $\Sigma$ is the paraboloid \eqref{paraboloid definition}. 

Let $\mu$ denote the normalised counting measure on $\Sigma$. In this case one may easily verify that (R$\mu$) holds with $a = n-1$. Indeed, if $\vec{\xi} \in [\Z/N\Z]_*^n$ and $0 < \rho$, then 
\begin{equation*}
\mu(\widehat{B}_{\rho}(\vec{\xi}\,)) \le \sum_{d \mid N; d\geq 1/\rho} \mu(\vec{\xi} + {\mathcal B}_{N/d})  \le \sum_{d \mid N; d\geq 1/\rho} d^{-(n-1)} \le  \bigl(\sum_{d \mid N} 1  \bigr) \, \rho^{n-1},
\end{equation*}
and the assertion now follows from the divisor bound.

It remains to establish the Fourier decay condition (F$\mu$) for a favourable choice of parameters $B$ and $b$, which requires the estimation of the exponential sum
\begin{equation*}
\check{\mu}(\vec{x}) \ = \ \frac{1}{N^{n-1}} \sum_{\vec{\omega} \in [\Z/N\Z]_*^{n-1}} e^{2 \pi i (x_1 \omega_1 + \cdots + x_{n-1} \omega_{n-1} + x_n (\omega_1^2 + \cdots + \omega_{n-1}^2))/N} .
\end{equation*}
As shown in $\S$\ref{basic setup section}, the above expression can be written as a product of Gauss sums $\check{\mu}(\vec{x}) = \prod_{j=1}^{n-1} G_N (x_j, x_n)$. Recalling \eqref{evaluating Gauss sum}, this vanishes unless $\gcd(x_1,\ldots,x_n, N) = \gcd(x_n,N)$ (or, using the established notation, $\|\vec{x}\| = |x_n|$), in which case, assuming that $N$ is odd,
\begin{equation*}
|\check{\mu}(\vec{x})| \ \le \ |x_n|^{-\frac{n-1}{2}} \ = \, \|\vec{x}\|^{-\frac{n-1}{2}} .
\end{equation*}
If $\vec{x} \notin B_{\rho}(\vec{0}\,)$, then $\vec{x} \notin {\mathcal B}_d$ for all divisors $d$ of $N$ satisfying $d\le \rho$. This implies that $\rho\le N/d' = \|\vec{x}\|$ where $d' = \gcd(x_1,\ldots,x_n, N)$. Therefore,
\begin{equation*}
|\check{\mu}(\vec{x})| \ \le  \rho^{-(n-1)/2} \quad {\rm whenever} \quad \vec{x} \notin B_{\rho}(\vec{0}\,),
\end{equation*}
showing that (F$\mu$) holds with $B = 1$ and $b = n-1$. 

Appealing to Theorem \ref{L2-restriction} now completes the proof of Theorem \ref{Tomas} with the precise meaning given in \eqref{general N problem inequality}. Explicitly, one has the following result. 

\begin{theorem}\label{rigorous Tomas} For all $\varepsilon > 0$ there exists a constant $C_{r,\varepsilon} > 0$ such that the estimate
\begin{equation*}
\Bigl(\frac{1}{N^{n-1}} \sum_{\vec{\omega} \in [\Z/N\Z]_*^{n-1}} |{\hat F}(\vec{\omega}, \omega_1^2 + \cdots + \omega_{n-1}^2)|^2 \Bigr)^{1/2} \ \le \ C_{r,\varepsilon} N^{\varepsilon} \, \|F\|_{\ell^r([\Z/N\Z]^n)}
\end{equation*}
holds for all odd $N \geq 1$ if and only if $1 \leq r \leq 2(n+1)/(n+3)$. 
\end{theorem}

It is remarked that in the finite field setting Theorem \ref{rigorous Tomas} is far from sharp. In $\S$\ref{basic setup section} it was observed that necessarily $r' \ge 2n/(n-1)$ but Theorem \ref{rigorous Tomas} only gives a positive result\footnote{Strictly speaking, one needs to be slightly careful when running the above argument in the finite field setting to ensure that the various constants are independent of the cardinality of the field. See \cite{Mockenhaupt2004} or \cite{Hickman3} for details.} for $r' \geq 2(n+1)/(n-1)$; see \cites{Iosevich, Iosevich2008, Iosevich2010, Iosevich2010a, Koh2014, Lewko2012, Lewko, Lewko2015, Mockenhaupt2004} for further improvements in the finite field setting.

Finally, the above arguments simplify when one restricts $N$ to prime powers. In particular, in this setting a stronger version of Theorem \ref{rigorous Tomas} holds.

\begin{theorem}[\cite{Hickman3}]\label{prime power rigorous Tomas} There exists a constant $C_r > 0$ such that the estimate
\begin{equation*}
\Bigl(\frac{1}{p^{\alpha(n-1)}} \sum_{\vec{\omega} \in [\Z/p^{\alpha}\Z]_*^{n-1}} |{\hat F}(\vec{\omega}, \omega_1^2 + \cdots + \omega_{n-1}^2)|^2 \Bigr)^{1/2} \le C_r \, \|F\|_{\ell^r([\Z/p^{\alpha}\Z]^n)}
\end{equation*}
holds for all odd primes $p$ and all $\alpha \in \N$ if and only if $1 \leq r \leq 2(n+1)/(n+3)$. 
\end{theorem}

This theorem appears in \cite{Hickman3} as a simple application of the aforementioned abstract $L^2$ restriction result for LCA groups. The uniformity of the constant in Theorem \ref{prime power rigorous Tomas} is consistent with the formulation of the restriction problem described in  Problem \ref{p power problem}.




\subsection{\texorpdfstring{$\ell^2$}{l2} restriction for the moment curve}

As in the euclidean setting, $\ell^2$ restriction arguments based only on the isotropic decay of the Fourier transform of $\mu$ will not give sharp results outwith the setting of hypersurfaces $\Sigma$. To illustrate this (and to initiate a discussion for $\S$\ref{Fourier restriction for curves section}), consider the case where $\Sigma$ is the moment curve, given by 
\begin{equation*}
\Sigma := \{(t,t^2,\ldots,t^n) : \, t\in [\Z/N\Z]_* \}.
\end{equation*}
The normalised counting measure $\mu$ on $\Sigma$ in this case has Fourier transform
\begin{equation*}
\check{\mu}(\vec{x}) \ = \ \frac{1}{N} \sum_{t=0}^{N-1} e^{2 \pi i (x_1 t + x_2 t^2 + \cdots + x_n t^n)/N}.
\end{equation*}
This exponential sum has been thoroughly studied by number theorists, beginning with Hua's \cite{Hua1940} classical estimate $|\check{\mu}(\vec{x})| \le C_{\varepsilon,n} N^{\varepsilon} \|\vec{x}\|^{-1/n}$ (using the notation of the present article), and improved so that there is no epsilon loss; for example, in \cite{Chen1977} it was shown that $|\check{\mu}(\vec{x})| \le B_{n} \|\vec{x}\|^{-1/n}$ for a constant $B_n$ depending only on the degree $n$ of the phase in the exponential sum (in fact $B_n := e^{4n}$ suffices for $n\ge 10$). Therefore, arguing as above, one verifies that in this case (F$\mu$) holds with $B = B_n$ and $b = 2/n$ and for any $\varepsilon >0$ the condition (R$\mu$) holds with $A = C_{\varepsilon}N^{\varepsilon}$ and $a = 1$. Thus, Theorem \ref{L2-restriction} gives an $\ell^2$ restriction estimate for the curve $\Sigma$ in the range $1 \le r \leq (n^2 - n +1)/(n^2 - n + 1/2)$. The non-optimality of this range is suggested by the scaling argument used in $\S$\ref{basic setup section} for the paraboloid. Here
one considers the anisotropic boxes
\begin{equation*}
\theta \ = \ \{(x_1,\ldots,x_n) \in [\Z/N\Z]^n : \ d | x_1, \, d^2 | x_2, \ldots, d^n | x_n \},
\end{equation*}
where $d$ is divisor of $N$ such that $d^n$ is also a divisor. One may then check as before that \eqref{general N problem inequality} can only hold when
\begin{equation}\label{nec-curve}
s \, \frac{n(n+1)}{2} \ \le \ r',
\end{equation}
which corresponds to condition on the euclidean exponents. When $s=2$ this gives the larger range $1 \le r \le (n^2 + n)/(n^2 + n - 1)$ (a strictly larger range when $n\ge 3$, the case when the curve $\Sigma$ is not a hypersurface). 

It is remarked that this scaling argument does not work in the setting of finite fields. Here, testing the Fourier restriction estimate against $f$ defined by
${\hat f} := \delta_{\vec{0}}$ leads to the necessary condition $n s \le r'$. Recall that there is also a necessary condition $2n/d \le r'$ in the finite field setting with $d=1$ for curves. These two necessary conditions were shown in \cite{Mockenhaupt2004} to be sufficient for the moment curve in the finite field setting if the characteristic of the field is larger than $n$. 

The Fourier restriction theory for the moment curve over $\Z/N\Z$ will be investigated in detail in the following section.




\subsection{\texorpdfstring{$\ell^2$}{l2} restriction for other surfaces} One could, of course, consider more general algebraic varieties $\Sigma$, say 
\begin{equation*}
\Sigma: = \{(\vec{\omega}, P_{1}(\vec{\omega}), \ldots, P_{n-d}(\vec{\omega})) : \ \vec{\omega} = (\omega_1,\ldots,\omega_d) \in [\Z/N\Z]_*^d \}
\end{equation*}
for some $1\le d \leq n-1$ and polynomials $P_j \in {\mathbb Z}[X_1, \ldots, X_d]$ for $1 \le j \le n-d$. In this case, for any $\varepsilon > 0$ the normalised counting measure $\mu$ is easily seen to satisfy (R$\mu$) with $A = C_{\varepsilon} N^{\varepsilon}$ and $a = d$. Therefore, given an exponential sum estimate of the form (F$\mu$) for the Fourier transform $\check{\mu}$, one may  employ Theorem \ref{L2-restriction} to obtain an $\ell^2$ restriction estimate. The natural question arises whether such a result is sharp. If the polynomials $P_j$ are homogeneous, then one may use the scaling argument as before to deduce a necessary condition on the exponents $r$ and $s$ for \eqref{general N problem inequality} to hold: namely, that $s (d + \sum_{j=1}^{n-d} m_j) \le d r'$ where $m_j$ is the degree of homogeneity of the polynomial $P_j$. Now further restrict attention to hypersurfaces $\Sigma$, so that $d=n-1$, and let $m$ denote the homogeneous degree of $h(\vec{\omega}) := p_n(\omega_1,\ldots,\omega_{n-1})$. The necessary condition for \eqref{general N problem inequality} to hold when $s=2$ then reads
\begin{equation}\label{necessary-homo}
2(1 + \frac{m}{n-1}) \ \le \ r'
\end{equation}
so that Theorem \ref{L2-restriction} would give a sharp $\ell^2$ restriction result if (F$\mu$) were to hold for $b = 2(n-1)/m$. Such decay estimates for exponent sums are known for the Fourier transform of the normalised counting measures $\mu_{h}$ on 
\begin{equation*}
\Sigma_{h} := \{(\omega_1,\ldots,\omega_{n-1}, h(\omega_1,\ldots,\omega_{n-1}) : (\omega_1,\ldots,\omega_{n-1}) \in [\Z/N\Z]_*^{n-1} \}
\end{equation*}
for particular choices of $h$. Here the exponential sum in question is
\begin{equation*}
\check{\mu}_h (\vec{x},x_n) \ = \  \frac{1}{N^{n-1}} \sum_{\vec{\omega} \in [{\mathbb Z}/N{\mathbb Z}]_*^{n-1}}
e^{2 \pi i (\vec{x}\cdot\vec{\omega} + x_n h(\vec{\omega}))/N}.
\end{equation*}
When $N=p^{\alpha}$ is a power of a prime $p$, sharp estimates for this object follow, for instance, from work of Denef and Sperber \cite{Denef2001} (see also \cite{Cluckers2008} and \cite{Cluckers2010}), resolving a conjecture of Igusa under a non-degeneracy condition on the homogeneous polynomial $h$. The decay rate $b$ in (F$\mu$) obtained by Denef and Sperber is given by the so-called \emph{Newton distance} $d(h)$ of $h$ which often matches the necessary condition \eqref{necessary-homo} but can be larger. The authors hope to investigate sharp exponential sum bounds for certain classes of homogeneous varieties (not necessarily hypersurfaces) and corresponding sharp $\ell^2$ restriction results in a future paper.




\section{Fourier restriction for curves}\label{Fourier restriction for curves section}




\subsection{Preliminary discussion} In this section the Fourier restriction problem for the moment curve 
\begin{equation*}
\Sigma := \{(t,t^2,\ldots,t^n) : \, t \in \Z/N\Z \}
\end{equation*}
is considered. If $N$ is only allowed to vary over powers of a fixed prime $p$, then, using the correspondence principle developed in $\S$\ref{p adic section}, it is a straight-forward exercise to adapt existing euclidean arguments to prove sharp restriction estimates for $\Sigma$. 

\begin{theorem}[\cite{Hickman2015}]\label{p power curve theorem} If $r' > n(n+1)/2 +1$ and $r' \geq s n(n+1)/2$, then the restriction estimate
\begin{equation*}
\Big(\frac{1}{p^{\alpha}}\sum_{t \in \Z/p^{\alpha}\Z} |\hat{F}(t, t^2, \dots, t^n)|^s \Big)^{1/s} \leq C_{n,r} \Big(\sum_{\vec{x} \in [\Z/p^{\alpha}\Z]^n} |F(\vec{x}\,)|^r \Big)^{1/r}
\end{equation*}
holds uniformly over all primes $p > n$ and all $\alpha \in \N$.
\end{theorem}

It is remarked that the range of Lebesgue exponents in Theorem \ref{p power curve theorem} is sharp, as shown in the following subsection. 

The proof of this theorem follows by lifting the problem to the $p$-adics using Proposition \ref{equivalence} and then adapting the classical euclidean argument of Drury \cite{Drury1985} to apply in this setting (one could also approach the $p$-adic formulation of the problem using alternative methods, such as those of \cite{Ham2014}); see \cite{Hickman2015} for details where similar restriction estimates are established over more general local fields.\footnote{The theorem stated in \cite{Hickman2015} suggests that the constant in the restriction estimate depends on $p$. Analysing the argument, however, shows that it yields a uniform estimate.}  The key advantage of working $p$-adically is that there is a well-developed calculus on $\Q_p^n$ which includes, significantly, a change of variables formula (see, for instance, \cite{Schikhof2006}*{\S 27}, \cite{Igusa2000}*{\S 7.4}, or \cite{Hickman2015}). This facilitates an easy and direct translation of various euclidean arguments over to the $p$-adics. 

The formulation of the problem for general $N$, rather than powers of a fixed prime, presents a number of significant additional difficulties. Recall that here one wishes to prove estimates of the form
\begin{equation}\label{discrete restriction}
\Big(\frac{1}{N}\sum_{t \in \Z/N\Z} |\hat{F}(t, t^2, \dots, t^n)|^s \Big)^{1/s} \leq C_{\varepsilon, r,s, n}N^{\varepsilon} \Big(\sum_{\vec{x} \in [\Z/N\Z]^n} |F(\vec{x}\,)|^r \Big)^{1/r}
\end{equation}
for all $\varepsilon > 0$ and a large class of integers $N \in \N$ (for instance, all $N$ for which every prime factor $p \mid N$ satisfies $p > n$). In this case one can no longer lift the analysis to the $p$-adic setting\footnote{For instance, the restriction estimate \eqref{discrete restriction} is \emph{not} multiplicative.} and the discrete problem must be tackled directly. Consequently, many fundamental tools from calculus are no longer available, and this leads to some new and interesting questions.

The purpose of this section is to describe the difficulties one encounters when attempting to prove estimates of the form \eqref{discrete restriction}. In particular, the Fourier restriction problem is related to a number-theoretic conjecture concerning factorisations of polynomials over $\Z/N\Z$. Some partial progress on the number-theoretic conjecture is described which, for instance, allows one to establish the modulo $N$ analogue of Theorem \ref{p power curve theorem} in the $n=2$ case. 




\subsection{Necessary conditions}\label{necessary conditions subsection} The first step is to determine necessary conditions on the Lebesgue exponents $(r,s)$ for \eqref{discrete restriction} to hold. As a by-product of this analysis, it will also be shown that the range of $(r,s)$ in the statement of Theorem \ref{p power curve theorem} is sharp. 

As remarked in the previous section, a simple scaling argument gives rise to the necessary condition \eqref{nec-curve} for the Fourier restriction estimates \eqref{discrete restriction} to hold. One now wishes to determine the possible $\ell^r$ range. By duality, \eqref{discrete restriction} is equivalent to
\begin{equation}\label{E-curve}
\Bigl(\sum_{\vec{x}\in [\Z/N\Z]^{n}} |{\mathcal E} H(\vec{x})|^{r'} \Bigr)^{1/r'} \ \le \ C_{\varepsilon, r,s,n} N^{\varepsilon} \Bigl(\frac{1}{N} \sum_{t \in \Z/N\Z} |  H(t) |^{s'} \Bigr)^{1/s'}
\end{equation}
where $\mathcal{E}$ is the extension operator
\begin{equation*}
{\mathcal E} H (\vec{x}) \ = \ \frac{1}{N} \sum_{t \in \Z/N\Z}  H(t) e^{2 \pi i (x_1 t + x_2 t^2 + \cdots + x_n t^n)/N}.
\end{equation*}
When $H = 1$ the right-hand side of \eqref{E-curve} is $C_{\varepsilon,r,s,n} N^{\varepsilon}$ whilst the left-hand side is the $\ell^{r'}$-norm of the function
\begin{equation}\label{exp-sum}
\mathcal{E}1(x_1,\ldots, x_n) \ = \ \frac{1}{N} \sum_{t \in \Z/N\Z} e^{2 \pi i (x_1 t + x_2 t^2 + \cdots + x_n t^n)/N}.
\end{equation}
Thus, it becomes of interest to determine the $\ell^{ r'}$ range for which
\begin{equation}\label{Lr-range}
\|\mathcal{E}1\|_{\ell^{ r'}([\Z/N\Z]^n)} \ \le \ C_{\varepsilon, r} N^{\varepsilon}
\end{equation}
holds for every $\varepsilon>0$. The corresponding euclidean problem is to determine the $L^{ r'}({\mathbb R}^n)$ spaces to which the oscillatory integral
\begin{equation*}
E1(x_1,\ldots, x_n) := \int_{0}^1 e^{2 \pi i (x_1 t + x_2 t^2 + \cdots + x_n t^n)} \, \ud t
\end{equation*}
belongs; this in turn gives rise to a necessary condition on the $L^{r}$ range for restriction problem for the curve $t \to (t,t^2,\ldots,t^n)$ in ${\mathbb R}^n$. It turns out that $\|E 1 \|_{L^{ r'}(\R^n)}$ also appears as a constant in the main term of an asymptotic formula for the number of solutions to a system of Diophantine equations known as Tarry's Problem.  Hence, knowing when $\|E 1\|_{L^{r'}(\R^n)}$ is finite has significance for harmonic analysts and number theorists for different reasons. Motivated by these number-theoretic considerations, Arkhipov, Chubarikov and Karatsuba \cite{Arkipov1979} (see also \cite{Arkipov1987}) showed that $E1 \in L^{ r'}(\R^n)$ if and only if $ r' > n(n+1)/2 + 1$. Reinforcing the theme of the paper, the following discrete analogue holds.

\begin{proposition}\label{range-Lr} The inequality \eqref{Lr-range} 
 fails if $r' < n(n+1)/2 + 1$.
\end{proposition}

This gives necessary conditions on the exponent $r$ for the restriction estimate \eqref{discrete restriction}. The proof of Proposition \ref{range-Lr} will also show the following. 

\begin{corollary}\label{endpoint range-Lr} If $p > n$ is a fixed prime, then 
$\|\mathcal{E}1\|_{\ell^{ r'}([\Z/p^{\alpha}\Z]^n)}$ is unbounded in $\alpha$ for $ r' \leq n(n+1)/2 + 1$. 
\end{corollary}

Combining Corollary \ref{endpoint range-Lr} with the previous discussion verifies that Theorem \ref{p power curve theorem} is sharp. On the other hand, for $N=p^{\alpha}$ Theorem \ref{p power curve theorem} implies that \eqref{Lr-range} holds for $r' \geq n(n+1)/2 + 1$ with a constant independent of $\alpha \in \N$ (but depending on $p$). 

Proposition \ref{range-Lr} is closely related to work of Arkhipov, Chubarikov and Karatsuba \cite{Arkipov1987} on Diophantine equations. In particular, restricting to $N=p^{\alpha}$, observe that
\begin{eqnarray*}
\| \mathcal{E}1 \|_{\ell^r([{\mathbb Z}/p^{\alpha}{\mathbb Z}]^n)}^r \ = \ \sum_{m=0}^{\alpha} &&
\sum_{x_1 =0}^{p^m -1} \cdots \sum_{x_n =0}^{p^m -1}  \bigl| S_m(x_1, \ldots, x_n)\bigr|^r \\
&&\scriptstyle{p\, \nmid \  \gcd(x_1,\ldots,x_n)}
\end{eqnarray*}
where
\begin{equation*}
S_m(x_1, \ldots, x_m) := \frac{1}{p^m}\sum_{t=0}^{p^m-1} e^{2 \pi i (x_1 t + \cdots + x_n t^n)/p^m}
\end{equation*}
 The sums $S_m(x_1, \ldots, x_m)$ play a key r\^ole in the analysis of the singular series $\sigma_{k,m}$ in Tarry's Problem in \cite{Arkipov1987}.

\begin{proof}[Proof (of Proposition \ref{range-Lr})] To establish  Proposition \ref{range-Lr} (and Corollary \ref{endpoint range-Lr}) a lower bound is obtained for the 
$\ell^{r'}$-norm of the $S_m$ above for any prime $p>n$, $m=nL$ and $r' \le n(n+1)/2 +1$.

First observe that $\| S_{nL} \|_{\ell^{r'}([{\mathbb Z}/p^{nL}{\mathbb Z}]^n)}^{r'}$ may be bounded below by 
\begin{equation*}
 \sum_{m=0}^{L-1} \sum_{\substack{x_n = 0 \\[5pt] p^{nm} \,\|  x_n}}^{p^{nL} -1} \:
          \sum_{\substack{x_{n-1} = 0 \\[5pt] p^{nm} \mid  x_{n-1}}}^{p^{nL} -1} \: \cdots \:
					 \sum_{\substack{x_{1} = 0 \\[5pt] p^{nm} \mid  x_{1}}}^{p^{nL}-1}  
					\Bigl| p^{-nL} \sum_{t=0}^{p^{nL} -1} e^{2 \pi i (x_1 t + \cdots + x_n t^n)/p^{nL}} \Bigr|^{r'};
\end{equation*}
recall, the notation $p^k \,\| \theta$ is used to denote that $p^k$ divides $\theta$ and no larger power of $p$ divides $\theta$. Splitting the above exponential sum by writing $t = y + p^{n(L-m)} z$ where $0\le y \le p^{n(L-m)} -1$ and $0\le z \le p^{nm} -1$, it follows that
\begin{equation}\label{sharpness 1}
\| S_{nL} \|_{\ell^{r'}([{\mathbb Z}/p^{nL}{\mathbb Z}]^n)}^{r'} \geq \sum_{m=0}^{L-1} A(L-m)
\end{equation}
where
\begin{equation*}
A(M) := \sum_{\substack{x_n = 0 \\[5pt] p \,\nmid \, x_n}}^{p^{nM}-1} \:  
         \sum_{x_{n-1} = 0}^{p^{nM} -1} \: \cdots \:  
         \sum_{x_1 = 0}^{p^{nM} -1}  \bigl| S_{nM}(x_1, \dots, x_n) \bigr|^{r'}.  
\end{equation*}
for $M \in \N$. 

\begin{claim} The inequality $A(M) \ge  p^{n(n+1)/2 + 1 - r'}A(M-1)$ holds for all $M \in \N$.
\end{claim} 

Once the claim is established it may be applied iteratively to bound each of the summands in \eqref{sharpness 1} and thereby deduce that 
\begin{equation*}
\| S_{nL} \|_{\ell^{r'}([{\mathbb Z}/p^{nL}{\mathbb Z}]^n)}^{r'}  \ge  \sum_{m=0}^{L-1} p^{(L-m)(n(n+1)/2 + 1 - r')}.
\end{equation*} 
This yields the desired blowup for $r' < n(n+1)/2 + 1$ (and for $r' = n(n+1)/2 +1$, in the context of Corollary \ref{endpoint range-Lr}). 

In order to verify the claim, first note that
\begin{equation}\label{lower bound for AM}
A(M) \geq \sum_{\substack{x_n = 0 \\[5pt] p \, \nmid \, x_n}}^{p^{nM} -1} \:
          \sum_{\substack{x_{n-1} = 0 \\[5pt] p\, \mid \, x_{n-1}}}^{p^{nM} -1} \! \cdots \!
					\sum_{\substack{x_1 = 0 \\[5pt] p^{n-1} \, \mid \, x_1}}^{p^{nM} -1} \:
					\sum_{c=0}^{p-1}   \Bigl| p^{-nM} \sum_{t=0}^{p^{nM} -1} e^{2 \pi i (x_1 (t+c) + \cdots + x_n (t+c)^n)/p^{nM}} \Bigr|^{r'},
\end{equation}
where the right-hand side of the above display can be expressed as
\begin{equation}\label{sharpness 2}
p \sum_{\substack{x_n = 0 \\[5pt] p \, \nmid \, x_n}}^{p^{nM} -1} \:
          \sum_{x_{n-1} = 0}^{p^{nM -1} -1} \!\! \cdots \!\!
					\sum_{x_1 = 0}^{p^{nM -n+1} - 1}
         \Bigl| p^{-nM}\sum_{t=0}^{p^{nM}-1} e^{2 \pi i (x_1 t/p^{n(M-1)+1} + \cdots + x_n t^n/p^{nM})} \Bigr|^{r'}	.
\end{equation}
To see this, consider the map $\Phi \colon \Z/p\Z \times [\Z/p^{nM}\Z]^n \to [\Z/p^{nM}\Z]^n$ given by
\begin{equation*}
\Phi(c; x_1,\dots,x_n) := 
\begin{pmatrix}
\binom{1}{1} & \binom{2}{1} c & \dots  & \binom{n}{1} c^{n-1} \\[2pt]
0            & \binom{2}{2}   & \dots  & \binom{n}{2} c^{n-2} \\[2pt]
\vdots       & \vdots         & \ddots & \vdots               \\[2pt]
0            & 0              & \dots  & \binom{n}{n} 
\end{pmatrix}
\begin{pmatrix}
x_1    \\
x_2    \\
\vdots \\
x_n
\end{pmatrix} ,
\end{equation*}
noting that
\begin{equation*}
x_1 (t+c) + \cdots + x_n (t+c)^n = \sum_{j=1}^n \Phi_j(c;x_1,\dots,x_n) t^j. 
\end{equation*}
To establish the lower bound \eqref{lower bound for AM} for $A(M)$ it suffices to show  that $\Phi$ restricts to an injection on the set
\begin{equation*}
\Omega := \big\{(c; x_1, \dots, x_n) \in {\mathbb Z}/p{\mathbb Z} \times [{\mathbb Z}/p^{nM}{\mathbb Z}]^n : p \mid x_{n-1} \textrm{ and } p \nmid x_n\big\}.
\end{equation*}
If $\Phi(c; \vec{a}) = \Phi(d; \vec{b})$ for $(c; \vec{a}), (d;\vec{b}) \in \Omega$, then it immediately follows that $x_n = b_n$ and $nx_nc + x_{n-1} = nb_nd + b_{n-1}$. Combining these identities, reducing modulo $p$ and using the fact that $p \nmid nx_n$, one concludes that $c = d$. The injectivity of $\Phi$ is now immediate, since the matrix in the definition of $\Phi$ has determinant 1.  

Now consider the inner exponential sum in \eqref{sharpness 2}; by decomposing the sum by writing $t = z + p^{nM-1} y$ this can be expressed as
\begin{equation*}
\sum_{z=0}^{p^{nM-1}-1} e^{2 \pi i (x_1 z/p^{n(M-1)+1} + \cdots + x_n z^n/p^{nM})}
\sum_{y=0}^{p-1} e^{2 \pi i nx_n z^{n-1} y / p}.
\end{equation*}
Since $p>n$ and $p \nmid x_n$, the sum in $y$ vanishes unless $p \mid z$ and hence the above expression is equal to
\begin{equation*}
p \sum_{w=0}^{p^{nM-2}-1} e^{2 \pi i (x_1 w/p^{n(M-1)} + \cdots + x_nw^n/p^{n(M-1)})} = p^{n-1} S_{n(M-1)}(x_1, \dots, x_n).
\end{equation*}
Thus, the right-and side of \eqref{sharpness 2} is equal to
\begin{equation*}
p^{-r'}  \sum_{\substack{x_n = 0 \\[5pt] p \,  \nmid \, x_n}}^{p^{nM} -1} \:
        \sum_{x_{n-1} = 0}^{p^{nM -1} -1} \: \cdots \;
				\sum_{x_1 = 0}^{p^{nM- n+1} - 1}
         |S_{n(M -1)}(x_1, \dots, x_n)|^{r'};
\end{equation*}
this can, in turn, be written as
\begin{equation*}
p^{-r' + n + (n-1) + \cdots + 1}
        \sum_{\substack{x_n = 0 \\[5pt] p \, \nmid \, x_n}}^{p^{n(M -1)}-1} \:
        \sum_{x_{n-1} = 0}^{p^{n(M -1)} -1} \: \cdots \;
				\sum_{x_1 = 0}^{p^{n(M-1)} - 1}
|S_{n(M -1)}(x_1, \dots, x_n)|^{r'},
\end{equation*}
which establishes the claim.
\end{proof}

\subsection{Sufficient conditions} Combining the necessary conditions discussed in the previous section, it is natural to conjecture the following.

\begin{conjecture}\label{general N curve conjecture} If $1 \leq r, s \leq \infty$ satisfy $r' \geq n(n+1)/2 + 1$ and $r' \geq s n(n+1)/2$, then the discrete restriction estimate \eqref{discrete restriction} holds whenever each prime factor $p$ of $N$ satisfies $p > n$.
\end{conjecture}

In contrast with the $\Z/p^{\alpha}\Z$ case treated in Theorem \ref{p power curve theorem}, there appear to be significant challenges in establishing Conjecture \ref{general N curve conjecture}. To exemplify this, it is instructive to attempt to follow the classical argument of Prestini \cite{Prestini1978} and Christ \cite{Christ1985} in the mod $N$ setting, with the aim of establishing restriction estimates for $\Sigma := \{(t, t^2, \dots, t^n) : t \in \Z/N\Z \}$ in the restricted $\ell^r$ range $r' \geq n(n+2)/2$. 

Proceeding by duality, one wishes to prove \eqref{E-curve} holds for exponents $(r,s)$ satisfying $r' \geq n(n+2)/2$ and $r' \geq sn(n+1)/2$. The desired estimate can be written succinctly as 
\begin{equation}\label{dual-result}
\|(H \ud\mu)\;\widecheck{}\;\|_{\ell^{r'}([\Z/N\Z]^n)} \ \le \ C_{\varepsilon} N^{\varepsilon} \| H \|_{\ell^{s'}_{\mathrm{avg}}([\Z/N\Z]_*)}
\end{equation}
where $\mu$ is the measure whose Fourier transform is the exponential sum in \eqref{exp-sum}. Letting $H \ud\mu * \cdots * H \ud\mu$ denote the $n$-fold convolution of $H \ud\mu$, one observes by the Hausdorff--Young inequality that
\begin{align}
\nonumber
\|(H \ud\mu)\;\widecheck{}\;\|_{\ell^{r'}([\Z/N\Z]^n)}^n &=
\|\big[(H \ud\mu)\;\widecheck{}\;\big]^n\|_{\ell^{r'/n}([\Z/N\Z]^n)} \\
\label{convineq}  &\leq \| (H \ud\mu) * \cdots * (H \ud\mu)\|_{\ell^{\rho}([\Z/N\Z]_*^n)}
\end{align}
where $n\rho'=r'$ (note that, since $r' \ge n(n+2)/2$, the exponent $\rho$ satisfies $1\le \rho \le 2$).
Now, for any test function $\phi \colon [\Z/N\Z]_*^n \to \C$ one has
\begin{equation*}
H \ud \mu * \cdots * H \ud \mu (\phi) \: = \: \frac{1}{N^n} \sum_{\vec{t} \in [\Z/N\Z]_*^n} \phi\bigl(\sum_{i=1}^n \gamma(t_i)\bigr) \prod_{i=1}^n
H(t_i) ,
\end{equation*}
where $\vec{t}=(t_1,\ldots,t_n)$ and $\gamma(s) = (s,s^2,\ldots,s^n)$ is the map parametrising the curve $\Sigma$. Set $\Phi(\vec{t}\,) := \sum_{i=1}^n \gamma(t_i)$ so that
\begin{equation*}
H \ud \mu * \cdots * H \ud \mu (\phi) \: = \: \frac{1}{N^n} \sum_{\vec{y} \in [\Z/N\Z]_*^n} \phi(\vec{y}\,) \sum_{\vec{t}: \Phi(\vec{t}\,) = \vec{y}} \ \ \prod_{i=1}^n
H(t_i) \:
\end{equation*}
and hence
\begin{equation*}
H \ud \mu * \cdots * H \ud \mu (\vec{y}\,) \ = \ \sum_{\vec{t}: \Phi(\vec{t}\,) = \vec{y}} \ \ \prod_{i=1}^n H(t_i).
\end{equation*}
Taking the $\ell^{\rho}$-norm, one therefore deduces that
\begin{equation*}
\| H \ud\mu * \cdots * H \ud\mu \|_{\ell^{\rho}([\Z/N\Z]_*^n)}^{\rho} \ = \ \frac{1}{N^n} \sum_{\vec{y} \in [\Z/N\Z]_*^n}
\bigl| \sum_{\vec{t} : \Phi(\vec{t}\,) = \vec{y}} \prod_{i=1}^n H(t_i) \ \bigr|^{\rho} .
\end{equation*}
Let $\mathbf{N}(\vec{s};N)$ denote the number of solutions $\vec{t} \in [\Z/N\Z]^n$ to the system $\Phi(\vec{X}) = \Phi(\vec{s}\,)$; that is, 
\begin{equation*}
\mathbf{N}(\vec{s};N) := \{ \vec{t} \in [\Z/N\Z]^n : \Phi(\vec{t}\,) = \Phi(\vec{s}\,) \}.
\end{equation*}
With this notation, one may write the above $\ell^{\rho}$-norm as
\begin{equation*}
\|H \ud\mu * \cdots * H \ud\mu \|_{\ell^{\rho}([\Z/N\Z]_*^n)}^{\rho} \ = \ \frac{1}{N^n} \sum_{\vec{s} \in [\Z/N\Z]_*^n} \frac{1}{\mathbf{N}(\vec{s};N)} \bigl| \sum_{\vec{t} : \Phi(\vec{t}\,) = \Phi(\vec{s})} \prod_{i=1}^n H(t_i) \ \bigr|^{\rho}.
\end{equation*}
Applying H\"older's inequality in the ${\vec{t}}$ sum yields
\begin{equation}\label{n-fold}
\|H \ud\mu * \cdots * H \ud\mu \|_{\ell^{\rho}([\Z/N\Z]_*^n)}^{\rho} \ \le \ \frac{1}{N^n} \sum_{\vec{t}\in [\Z/N\Z]_*^n} \prod_{i=1}^n |H(t_i)|^r \, \mathbf{N}(\vec{t};N)^{{\rho}-1}.
\end{equation}
Up to this point the analysis has closely followed the euclidean argument of Prestini \cite{Prestini1978} and Christ \cite{Christ1985}. In the euclidean case the change of variables $\vec{y} = \Phi(\vec{t}\,)$ (performed twice) introduces a power of the Jacobian factor
\begin{equation*}
J_{\Phi}(\vec{t}\,) \ = \ \prod\limits_{1\le j<k \le n} |t_j - t_k|,
\end{equation*}
leading to the estimate
\begin{equation}\label{euclidean n-fold}
\| h \ud\mu * \cdots * h \ud\mu \|_{L^{\rho}(\R^n)}^{\rho} \ \le C_n \int_{\vec{t}\in {\mathbb R}^n} \prod_{i=1}^n |h(t_i)|^{\rho} \,
\prod_{1\leq j<k \leq n} \frac{1}{|t_j - t_k|^{{\rho}-1}} \, d\vec{t} .
\end{equation}

Comparing \eqref{n-fold} and \eqref{euclidean n-fold} suggests the following conjecture on the number of solutions $\mathbf{N}(\vec{t}; N)$.

\begin{conjecture}\label{congruence conjecture} For all $\varepsilon > 0$ there exists a constant $C_{\varepsilon, n} > 0$ such that if $\vec{t} \in [\Z/N\Z]^n$, then 
\begin{equation}\label{congruence conjecture inequality}
\mathbf{N}(\vec{t}; N) \leq \min\big\{C_{\varepsilon, n} N^{\varepsilon + n(n-1)/2} \prod_{1 \leq j < k \leq n} |t_j - t_k|^{-1}, N^n\big\}.
\end{equation}
\end{conjecture} 
The right-hand expression in \eqref{congruence conjecture inequality} is written in terms of the absolute value on $\Z/N\Z$ which, by definition, satisfies $N|y|^{-1} = \gcd(y, N)$ for all $y \in \Z/N\Z$. Thus, to establish Conjecture \ref{congruence conjecture} (since the estimate $\mathbf{N}(\vec{t}; N) \leq N^n$ holds trivially) it suffices to show that
\begin{equation*}
 \mathbf{N}(\vec{t}; N) \leq C_{\varepsilon, n} N^{\varepsilon} \prod_{1 \leq j < k \leq n} \gcd(t_j - t_k, N).
\end{equation*}

Assuming the conjecture holds, one is in position to appeal to a multilinear fractional integral inequality of Christ \cite{Christ1985}; although presented in the euclidean setting in \cite{Christ1985}, the statement and proof of the multilinear inequality translate directly into the mod $N$ setting. The result may be stated in the following way: if $F_i$ and $G_{j,k}$ are functions on $[\Z/N\Z]_*$, then the multilinear form 
\begin{equation}\label{g}
\frac{1}{N^n} \sum_{\vec{t}\in [\Z/N\Z]_*^n} \prod_{i=1}^n F_i(t_i)  \prod_{1 \leq j<k \leq n}G_{j,k}(t_j - t_k)
\end{equation}
is dominated by (a constant multiple of)
\begin{equation}\label{g right}
 \prod_{i=1}^n \|F_i \|_{\ell^{\alpha}([\Z/N\Z]_*)} \prod_{1\leq j < k \leq n} \|G_{j,k} \|_{\ell^{\beta,\infty}([\Z/N\Z]_*)}
\end{equation}
whenever $\alpha^{-1} + \beta^{-1}(n-1)/2 \le 1$ and $1\le \alpha < n$, where 
\begin{equation*}
\| G \|_{\ell^{\beta,\infty}([\Z/N\Z]_*)} \ := \ \sup_{\lambda>0} \lambda \, \Bigl( N^{-1} \bigl| \{ t \in \Z/N\Z : G(t) \ge \lambda \}\bigr| \Bigr)^{1/\beta}
\end{equation*}
denotes the weak-type Lorentz norm. The first step in Christ's proof is to observe that a simple interpolation argument bounds \eqref{g} by an expression given by replacing the weak-type norms ${\|\cdot\|_{\ell^{\beta,\infty}([\Z/N\Z]_*)}}$ by the strong-type $\ell^{\beta}$-norms ${\| \cdot \|_{\ell^{\beta}([\Z/N\Z]_*)}}$ in \eqref{g right} in the larger $\ell^{\alpha}$-range $1\le \alpha \le n$. In fact, here this weakened estimate is all one needs. Indeed, the multilinear inequality will be applied to the functions $G_{j,k}(t) := \gcd(t,N)^{\rho-1}$ with $\beta = 1/(\rho-1)$; note that
\begin{align*}
\| \gcd(\cdot,N) \|_{\ell^1([\Z/N\Z]_*)} = \frac{1}{N} \sum_{t\in \Z/N\Z} \gcd(t,N) = \frac{1}{N} \sum_{d \mid N} d \sum_{t: \gcd(t,N) = d} 1,
\end{align*}
where the latter expression is clearly bounded above by the divisor function and therefore grows sub-polynomially in $N$. Since an $\varepsilon$-loss in $N$ is permissible for the present purpose of establishing inequalities of the form \eqref{discrete restriction}, one may work with the strong-type $\ell^{\beta}([\Z/N\Z]_*)-$norms of the $G_{j,k}$. It is remarked that if $N = p^M$ where $p$ is prime, then the divisors are totally ordered and the weak-type norms of $\gcd(\cdot,p^M)$ are uniformly bounded, whereas the $\ell^1([\Z/N\Z]_*)$ norm is equal to $M$. Thus, if one were to restrict $N$ to powers of $p$ and seek stronger Fourier restriction estimates with bounds which are uniform in the power $M$, then the full strength of Christ's multilinear inequality would be needed.

Returning to the present situation, one obtains via the first step of Christ's argument (simple interpolation) the inequality
\begin{equation}\label{multilinear}
\frac{1}{N^n} \sum_{\vec{t}\in [\Z/N\Z]_*^n} \prod_{i=1}^n F(t_i) \prod_{1 \leq j<k \leq n} \gcd(t_j - t_k, N)^{\gamma}
\ \le \ C_{s,\gamma,\varepsilon} N^{\varepsilon} \|F\|_{\ell^{\alpha}([\Z/N\Z]_*)}^n
\end{equation}
for  $\gamma \le 2/n$ and $\alpha^{-1} + \gamma (n-1)/2 \le 1$.

\begin{remark} As in the euclidean setting, the stated range of exponents for \eqref{multilinear} is sharp. In fact, the familiar scaling argument given by taking $F := \chi_{{\mathcal B}_d}$ for $d$ a divisor of $N$ shows necessarily that $\alpha^{-1} + \gamma (n-1)/2 \le 1$. Furthermore, plugging the function $F := 1$ into \eqref{multilinear} shows that $\gamma \le 2/n$ must hold. Indeed, if $N = p^M$ where $M > n$ and $p$ is a prime, then, by restricting the range of summation, the left-hand side of \eqref{multilinear} is bounded below by
\begin{equation*}
\frac{1}{p^{Mn}} \sum_{0 \leq u_1 < \dots < u_{n-1} \leq M} \prod_{i=1}^n \sum_{\substack{t_i = 0 \\[5pt] p^{u_i} \| t_i - t_{i-1}}}^{p^M -1} \prod_{1 \leq j < k \leq n}\gcd(t_j - t_k, p^M)^{\gamma}.
\end{equation*}
Each summand $\prod_{j<k} \gcd(t_j - t_k, p^M)^{\gamma}$ over this restricted range of summation is equal to
\begin{equation*}
\prod\limits_{j=1}^{n-1} \gcd(t_j - t_{j+1}, p^M)^{\gamma (n - j)}
\end{equation*}
and this readily shows that
\begin{equation*}
p^{-Mn} \sum_{\vec{t}\in [{\mathbb Z}/p^M{\mathbb Z}]^n} \prod_{1 \leq j<k \leq n} \gcd(t_j - t_k, p^M)^{\gamma} \ \ge \
2^{-(n-1)} p^{(n-1) (\gamma n /2 - 1) M},
\end{equation*}
forcing $\gamma \le 2/n$.
\end{remark}

Assuming Conjecture \ref{congruence conjecture} one may apply \eqref{multilinear} to bound the right-hand side of \eqref{n-fold} with $\gamma = \rho-1$ and $\alpha$ satisfying $\alpha \rho = r'$ (the restriction $\gamma \le 2/n$ needed for the application of \eqref{multilinear} is equivalent to $r' \ge n(n+2)/2$ and the condition $\alpha^{-1} + \gamma(n-1)/2 \le 1$ is equivalent to $r' \ge n(n+1)/2 s$). Combining this inequality with \eqref{convineq}, one concludes that, conditionally on Conjecture \ref{congruence conjecture}, the desired restriction estimate \eqref{dual-result} holds with $s n(n+1)/2 \le r'$ in the range $r' \ge n(n+2)/2$.




\subsection{Remarks and partial progress towards Conjecture \ref{congruence conjecture}} In the previous subsection restriction estimates for the moment curve $\{ (t, t^2, \dots, t^n) : t \in \Z/N\Z\}$ were shown to follow from (the purely number-theoretic) Conjecture \ref{congruence conjecture} which concerns the number of mutually incongruent solutions to a simple system of equations. In particular, for each fixed $\vec{y} \in [\Z/N\Z]^n$ one wishes to determine an upper bound for the number $\mathbf{N}(\vec{y}; N)$ of solutions in $[\Z/N\Z]^n$ to the polynomial system
\begin{equation}\label{system of congruences}
\begin{array}{rcl}
X_1 + \dots + X_n     &\equiv& y_1   + \dots + y_n \\
& \vdots & \\
X_1^n + \dots + X_n^n &\equiv& y_1^n + \dots + y_n^n 
\end{array} \mod N. 
\end{equation}
This problem is arguably of interest in its own right; for instance, it can be reinterpreted as a natural question regarding factorisations of polynomials over $\Z/N\Z$. 

\begin{lemma} Suppose every prime factor $p$ of $N \in \N$ satisfies $p > n$ and that $F \in \Z[X]$ splits over $\Z/N\Z$, so that $F(X) \equiv \prod_{j=1}^n (X-y_j) \bmod N$ for some choice of roots $\vec{y} = (y_1, \dots, y_n) \in [\Z/N\Z]^n$. The number of ways $F$ can be factorised as a product of linear factors over $\Z/N\Z$ is $\mathbf{N}(\vec{y}; N)$.  
\end{lemma}

\begin{proof} Under the hypotheses of the lemma, it suffices to show that the set of solutions to \eqref{system of congruences} is precisely
\begin{equation*}
\big\{\vec{x} \in [\Z/N\Z]^n : \prod_{j=1}^n (X-x_j) \equiv \prod_{j=1}^n (X-y_j) \bmod N \big\}.
\end{equation*}
Since the coefficients of a polynomial are elementary symmetric functions of the roots, it follows that $\prod_{j=1}^n (X-x_j) \equiv \prod_{j=1}^n (X-y_j) \bmod N$ for some $\vec{x} \in [\Z/N\Z]^n$ if and only if $e_k(\vec{x}\,) \equiv e_k(\vec{y}\,) \bmod N$ for $1 \leq k \leq n$, where $e_k \in \Z[X_1, \dots, X_n]$ is the $k$th elementary symmetric polynomial. By the classical Newton--Girard formul\ae, if $p > n$, then this is equivalent to the condition that $\vec{x}$ solves \eqref{system of congruences}. 
\end{proof}

Although Conjecture \ref{congruence conjecture} remains open, there has been some partial progress on the problem. First observe that, by the Chinese remainder theorem, the function $\mathbf{N}(\vec{y}; N)$ is multiplicative in $N$ and it therefore suffices to prove that 
\begin{equation*}
\mathbf{N}(\vec{y}; p^{\alpha}) \leq C_n \prod_{1 \leq j < k \leq n} \gcd(y_j - y_k, p^{\alpha}) 
\end{equation*}
uniformly over all primes $p > n$ and $\alpha \in \N$. Indeed, this follows from the asymptotics for the distinct divisor function $\omega(N) := \sum_{p \mid N} 1$, as discussed in the previous section. 

If the $y_1, \dots, y_n$ are sufficiently separated in the $p$-adic sense, then Conjecture \ref{congruence conjecture} is a simple consequence of Hensel's classical lemma. 

\begin{lemma}\label{non-degenerate lemma} Suppose $p > n$ and $\alpha \in \N$. If $\vec{y} = (y_1, \dots, y_n) \in [\Z/p^{\alpha}\Z]^n$ satisfies $\delta <\alpha/2$ where $p^{\delta} \| \prod_{j < k} (y_j - y_k)$, then  
\begin{equation*}
\mathbf{N}(\vec{y}; p^{\alpha}) \leq n!p^{\delta} = n!p^{n(n-1)/2} \prod_{1 \leq j < k \leq n} |y_j - y_k|^{-1}.
\end{equation*}
\end{lemma}

This bound quickly leads to a resolution of Conjecture \ref{congruence conjecture} in the $n=2$ case.

\begin{corollary}\label{2d congruence corollary} Suppose $p$ is an odd prime and $\alpha \in \N$. For all $\vec{y} =(y_1, y_2) \in [\Z/p^{\alpha}\Z]^2$ one has
\begin{equation}\label{2d congruence}
\mathbf{N}(\vec{y}; p^{\alpha}) \leq 2 p |y_1 - y_2|^{-1}.
\end{equation}
\end{corollary}

\begin{proof} Let $p^{\delta} \| \gcd(y_1 - y_2, p^{\alpha})$. Lemma \ref{non-degenerate lemma} implies that \eqref{2d congruence} holds whenever $\delta <\alpha/2$ and so one may assume without loss of generality that $\delta \geq \alpha/2$. Thus, in particular, 
\begin{equation}\label{close y}
p^{\lceil \alpha/2 \rceil} | (y_1 - y_2). 
\end{equation}

Let $\vec{x} \in [\Z/p^{\alpha}\Z]^2$ be a solution to the system
\begin{equation}\label{2d system}
\begin{array}{rcl}
X_1 + X_2     &\equiv& y_1 + y_2 \\
X_1^2 + X_2^2 &\equiv& y_1^2 + y_2^2
\end{array} \mod p^{\alpha}.
\end{equation}
By the elementary formula $(X_1 - X_2)^2 = 2(X_1^2 + X_2^2) - (X_1 + X_2)^2$ one deduces that $(x_1 - x_2)^2 \equiv (y_1 - y_2)^2 \bmod p^{\alpha}$ and, recalling \eqref{close y}, the solution $\vec{x}$ satisfies $p^{\lceil \alpha/2 \rceil} | (x_1 - x_2)$.  Consequently, $x_1$ is uniquely determined modulo $p^{\lceil \alpha/2 \rceil}$ by $\vec{y}$ whilst, by the first equation in \eqref{2d system}, $x_2$ is determined by $x_1$ and $\vec{y}$ modulo $p^{\alpha}$. One now concludes that there are at most $p^{\alpha - \lceil \alpha/2 \rceil} \leq p^{\delta}$ solutions in this case.
\end{proof}

Combining the above solution count with the analysis of the previous subsection, one obtains the following discrete analogue of the Fefferman--Zygmund restriction theorem \cite{Zygmund1974} in the plane. 

\begin{theorem}\label{sharp restriction theorem} If $1 \leq r, s \leq \infty$ satisfy $r' \geq 4$ and $r' \geq 3s$, then for all $\varepsilon > 0$ there exists a constant $C_{\varepsilon, r, s} > 0$ such that  
\begin{equation*}
\Big(\frac{1}{N}\sum_{t \in \Z/N\Z} |\hat{F}(t, t^2)|^s \Big)^{1/s} \leq C_{\varepsilon, r,s}N^{\varepsilon} \Big(\sum_{\vec{x} \in [\Z/N\Z]^2} |F(\vec{x}\,)|^r \Big)^{1/r}
\end{equation*}
holds for all odd $N$.
\end{theorem}

The range of Lebesgue exponents in Theorem \ref{sharp restriction theorem} is sharp, as shown by the discussion in $\S$\ref{necessary conditions subsection}.

The proof of Lemma \ref{non-degenerate lemma} relies on the following (well-known) multivariate version of Hensel's classical lemma. 

\begin{lemma}[Hensel]  Let $f_1, \dots, f_n$ be polynomials in $\Z_p[X_1, \dots, X_n]$ and consider the polynomial mapping $\vec{f} := (f_1, \dots, f_n)$. Suppose $\vec{x} \in \Z^n$ satisfies the system of congruences $\vec{f}(\vec{x}\,)\equiv 0 \mod p^s$ and, further, that $p^{\delta} \| J_{\vec{f}}\,(\vec{x}\,)$ with $2\delta < s$, where $J_{\vec{f}}\,(\vec{x}\,)$ denotes the Jacobian determinant of $\vec{f}$ at $\vec{x}$. Then there exists a unique $\vec{x}_* \in  \Z_p^n$ such that $\vec{f}(\vec{x}_*) = 0$ and $\vec{x}_* \equiv \vec{x} \mod p^{s-\delta}$. 
\end{lemma}

For a proof of this version of Hensel's lemma see, for instance, \cite{Wooley1996} or \cite{Greenberg1969}*{Proposition 5.20}.

\begin{proof}[Proof (of Lemma \ref{non-degenerate lemma})] Fix $\vec{y} \in [{\mathbb Z}/p^{\alpha}{\mathbb Z}]^n$ satisfying the hypotheses of the lemma. Recall that one wishes to estimate the number $\mathbf{N}(\vec{y}; p^{\alpha})$ of solutions $\vec{x} \in [{\mathbb Z}/p^{\alpha}{\mathbb Z}]^n$ to the system of congruences \eqref{system of congruences}. Recalling the mapping 
\begin{equation*}
\Phi(X_1, \dots, X_n) := \big( X_1 + \cdots + X_n, \dots, X_1^n + \cdots + X_n^n\big)
\end{equation*}
introduced in the previous subsection, this system can be concisely written as
\begin{equation}\label{congruence}
\Phi(\vec{X}) \ \equiv \ \Phi(\vec{y}\,) \mod p^{\alpha}.
\end{equation}
Let $\Phi'(\vec{x}\,)$ denote the Jacobian matrix of first order partial derivatives of the components of $\Phi$ and $J_{\Phi}(\vec{x}\,) = n!\prod_{j<k} (x_j - x_k)$ the corresponding determinant. The hypothesis on $\vec{y}$ is therefore that $p^{\delta} \, \| \ J_{\Phi}(\vec{y}\,)$ for $\delta < \alpha/2$. If $\vec{x}$ is a solution to \eqref{congruence}, then $p^{\delta} \, \| \ J_{\Phi}(\vec{x}\,)$ also holds. Indeed, this simply follows by expressing the symmetric discriminant $\prod_{j<k} (X_j - X_k)^2$ as a polynomial of the symmetric power functions $\sum_{j=1}^n X_j^k$ via the Newton--Girard formul\ae, from which one concludes that 
\begin{equation*}
\prod_{1 \leq j<k \leq n} (x_j - x_k)^2 \equiv \prod_{1 \leq j<k \leq n} (y_j - y_k)^2 \mod p^{\alpha}
\end{equation*}
for a solution $\vec{x}$ to \eqref{congruence}.\footnote{This argument is carried out explicitly for $n=2$ in the proof of Corollary \ref{2d congruence corollary}} One is now in a position to apply Hensel's lemma which shows that for every solution $\vec{x}$ of \eqref{congruence} there is a unique $p$-adic solution $\vec{z} \in {\mathbb Z}_p^n$ to $\Phi(\vec{z}\,) = \Phi(\vec{y}\,)$ in ${\mathbb Z}_p^n$ such that ${\vec{x}} \equiv {\vec{z}} \bmod p^{\alpha-\delta}$. Since ${\mathbb Z}_p$ is an integral domain with characteristic 0, a standard argument using the Newton--Girard 
 formul\ae\, shows that there are at most $n!$ $p$-adic solutions ${\vec{z}}$, which all arise by permuting the components of the solution $\vec{y} = (y_1,\ldots,y_n)$. It therefore suffices to count the solutions $\vec{x}$ to \eqref{congruence} which satisfy $\vec{x} \equiv \vec{y} \bmod p^{\alpha - \delta}$; all other solutions arise by permuting the components of some solution $\vec{x}$ of this form and so the total solution count $\mathbf{N}(\vec{y}; p^{\alpha})$ will differ from this partial count by at most a factor of $n!$. For such a solution $\vec{x} \in [\Z/p^{\alpha}\Z]^n$ one has
\begin{equation*}
\vec{x} = \vec{y}_{\delta} + p^{\alpha-\delta} \vec{x}_{\alpha-\delta} \quad \textrm{and} \quad \vec{y} = \vec{y}_{\delta} + p^{\alpha-\delta} \vec{y}_{\alpha-\delta}
\end{equation*}
for some $\vec{y}_{\delta}, \vec{x}_{\alpha-\delta}, \vec{y}_{\alpha-\delta} \in [\Z/p^{\alpha}\Z]^n$.  Since $\delta < \alpha/2$, it follows that 
\begin{equation*}
\begin{array}{rcl}
\Phi(\vec{x}\,) &\equiv& \Phi(\vec{y}_{\delta}) + p^{\alpha-\delta} \Phi'(\vec{y}_{\delta}) \vec{x}_{\alpha-\delta} \\
\Phi(\vec{y}\,) &\equiv& \Phi(\vec{y}_{\delta}) + p^{\alpha-\delta} \Phi'(\vec{y}_{\delta}) \vec{y}_{\alpha-\delta} 
\end{array} 
\mod p^{\alpha}
\end{equation*}
and so 
\begin{equation}\label{matrix equation}
\Phi'(\vec{y}_{\delta}) (\vec{x}_{\alpha-\delta} - \vec{y}_{\alpha-\delta} ) \ \equiv \ \vec{0} \mod p^{\delta}.
\end{equation}
Note that $J_{\Phi}(\vec{y}\,) \equiv J_{\Phi}(\vec{y}_{\delta}) \bmod p^{\alpha - \delta}$ and therefore, again using the hypothesis $\delta < \alpha/2$, one deduces that $p^{\delta} \, \|  J_{\Phi}(\vec{y}_{\delta})$. Applying Lemma \ref{linear system lemma}, one concludes that there are at most $p^{\delta}$ solutions $\vec{x}_{\alpha - \delta} \in [\Z/p^{\alpha}\Z]^n$ to \eqref{matrix equation} which are mutually incongruent modulo $p^{\delta}$. This immediately yields the desired bound on $\mathbf{N}(\vec{y};p^{\alpha})$. 
\end{proof}

The $n=3$ case of Conjecture \ref{congruence conjecture} can also be treated using similar (but somewhat more involved) arguments. This line of reasoning tends to be rather \emph{ad hoc}, however, and it is unclear whether it can produce a systematic approach which resolves the conjecture for all values of $n$ (already in the $n=4$ case significant complications arise and, indeed, the problem remains open for $n \geq 4$). 

A counterpoint to Lemma \ref{non-degenerate lemma} was established by the authors in \cite{Hickman2} 

\begin{proposition}[\cite{Hickman2}]\label{degenerate proposition}  If $n = r(r+1)/2$ for some $r \in \N$ with $r \geq 2$ and $p > n$ is prime, then\footnote{The notation of the present article differs slightly with that of \cite{Hickman2}: in the latter, $\mathbf{N}(\vec{0}_n; p^{\alpha})$ denotes a \emph{normalised} solution count.}
\begin{equation*}
\mathbf{N}(\vec{0}_n; p^{\alpha}) \leq C_n  \alpha p^{\alpha (n-r)}
\end{equation*}
holds for all $\alpha \in \N$. The result is sharp in the sense that, provided  $n\neq 3$ and $p$ is sufficiently large depending only on $n$, the reverse inequality also holds for infinitely many $\alpha$. 
\end{proposition}

Proposition \ref{degenerate proposition} treats a very different situation from that considered in Lemma \ref{non-degenerate lemma}: here the components of $\vec{y}$ are identical and therefore have no $p$-adic separation. This result also shows that, in general, Conjecture \ref{congruence conjecture} is \emph{not} sharp (this can already be observed from the proof of Corollary \ref{2d congruence corollary}) since Conjecture \ref{congruence conjecture} only predicts the trivial bound $\mathbf{N}(\vec{0}_n; p^{\alpha}) \leq p^{\alpha n}$. The restriction to triangular degrees $n = r(r+1)/2$ in Proposition \ref{degenerate proposition} is merely for expository purposes: in \cite{Hickman2} sharp estimates for $\mathbf{N}(\vec{0}_n; p^{\alpha})$ are obtained in all dimensions, but the statement of the result for general $n$ is slightly involved. Curiously, the $n=3$ case behaves differently from all other degrees. 

The proof of Proposition \ref{degenerate proposition} uses methods akin to those employed by Denef and Sperber \cite{Denef2001} (see also \cites{Cluckers2008, Cluckers2010}) to study exponential sum bounds related to the Igusa conjecture. An interesting feature of the analysis in \cite{Hickman2} is that it applies to \emph{systems} of polynomial congruences, rather than just a single polynomial congruence as considered in \cite{Denef2001}. There is mounting evidence that these methods can be pushed to prove more substantial partial results on Conjecture \ref{congruence conjecture}, and perhaps even lead to a full resolution of the problem, and the authors hope to investigate this further in future work.




\appendix 
\section*{Appendix}

\section{Counting solutions to linear systems of congruences}

The following lemma was used a number of times in the text. 

\begin{lemma}\label{linear system lemma} Let $N \in \N$, $\vec{b} \in [\Z/N\Z]^n$ and suppose $A \in \mathrm{Mat}_n(\Z/N\Z)$ satisfies $\det A \not\equiv 0 \bmod N$. The number of solutions $\vec{x} \in [\Z/N\Z]^n$ to the system of linear congruences $A \vec{x} = \vec{b}$ is either 0 or $N/|\det A|$. 
\end{lemma}

\begin{proof} Since the desired estimate is multiplicative, one may assume without loss of generality that $N = p^{\alpha}$ is a power of a fixed prime. Furthermore, for any $A \in \mathrm{Mat}_n(\Z/N\Z)$ there exist unimodular matrices $U, V \in \mathrm{GL}_n(\Z/N\Z)$ such that $VAU$ is diagonal (this a consequence of the existence of the \emph{Smith normal form} of the matrix $A$, which holds for arbitrary (that is, not necessarily square) matrices over any principal ideal domain: see, for instance, \cite{Lang1984}). Since $|\det A| = |\det VAU|$ and
\begin{equation*}
|\{\vec{x} \in [\Z/N\Z]^n : A \vec{x} = \vec{b} \}| = |\{\vec{x} \in [\Z/N\Z]^n : VAU \vec{x} = V\vec{b} \}|,
\end{equation*}
one may further assume that $A$ itself is diagonal. 

Let $\lambda_i$ denote the $(i,i)$-entry of $A$ and write $p^{\psi_i} = \gcd(\lambda_i, p^{\alpha})$ for $i = 1, \dots, n$. The number of solutions $x \in \Z/p^{\alpha}\Z$ to the univariate system
\begin{equation*}
\lambda_i x = b_i
\end{equation*}
is equal to 0 if $p^{\psi_i} \nmid b_i$ and is equal to $p^{\psi_i}$ otherwise. Thus, the total number of solutions to the system is at most $p^{\psi_1 + \dots + \psi_n}$. Finally, by hypothesis $p^{\alpha} \nmid \det A =  \lambda_1 \dots \lambda_n$ and so 
\begin{equation*}
p^{\psi_1 + \dots + \psi_n} = \gcd(\det A, p^{\alpha}) = \frac{N}{|\det A|},
\end{equation*}
as required. 
\end{proof}

\section*{Acknowledgments} This paper updates and expands work that was initiated by the second author over 20 years ago. The authors would like to thank M. Cowling for proposing the original line of investigation and for his encouragement. This material is based upon work supported by the National Science Foundation under Grant No. DMS-1440140 while the first author was in residence at the Mathematical Sciences Research Institute in Berkeley, California, during the Spring 2017 semester.

\bibliographystyle{amsplain}


\begin{dajauthors}
\begin{authorinfo}[Hickman]
  Jonathan Hickman\\
 Eckhart Hall Room 414, \\
 Department of Mathematics, \\
 University of Chicago, \\
 5734 S. University Avenue, \\
 Chicago, Illinois,  60637, US.\\
  jehickman\imageat{}uchicago\imagedot{}edu \\
  \url{http://math.uchicago.edu/~j.e.hickman/}
\end{authorinfo}
\begin{authorinfo}[Wright]
  James Wright\\
  Room 4621 James Clerk Maxwell Building, \\
  The King's Buildings, \\
  Peter Guthrie Tait Road, \\
  Edinburgh, \\
  EH9 3FD, \\
  UK.\\
  j.r.wright\imageat{}ed\imagedot{}ac\imagedot{}uk 
\end{authorinfo}

\end{dajauthors}

\end{document}